\documentclass[11pt]{amsart}

\makeatletter
\def\@settitle{\begin{center}%
  \baselineskip14\p@\relax
  \bfseries
  \uppercasenonmath\@title
  \@title
  \ifx\@subtitle\@empty\else
     \\[1ex]\uppercasenonmath\@subtitle
     \footnotesize\mdseries\@subtitle
  \fi
  \end{center}%
}
\def\subtitle#1{\gdef\@subtitle{#1}}
\def\@subtitle{}
\makeatother

\usepackage{amssymb}	
\usepackage{tikz}
\usepackage{tikz-cd}
\usepackage{amsmath}
\usepackage{mathrsfs}
\usepackage{amsthm}
\usepackage{verbatim}
\usepackage{subeqnarray}
\usepackage{graphicx}
\usepackage{subfigure}

\usepackage{cancel}
\usepackage{color}
\usepackage{float}
\usepackage{bm}
\usepackage{hyperref}
\usepackage{amsfonts}
\usepackage{amscd}
\usepackage{cases}
\usepackage{xcolor}
\usepackage{eucal}
\usepackage[most]{tcolorbox}

\usepackage{cleveref}

\usepackage{centernot}

\definecolor{red}{rgb}{0.8,0,0}
\definecolor{darkorange}{rgb}{1,0.4,0}
\definecolor{lightorange}{rgb}{1,0.6, 0}
\definecolor{yellow}{rgb}{1,0.8, 0}

\usepackage{adjustbox}

\newtheorem{theorem}{Theorem}
\newtheorem{remark}{Remark}

\newtheorem{definition}{Definition}

\newtheorem{example}{Example}

\newcommand\tr{\operatorname{tr}}
\newcommand\inc{\operatorname{inc}}
\newcommand\skw{\operatorname{skw}}
\newcommand\vskw{\operatorname{vskw}}
\newcommand\mskw{\operatorname{mskw}}

\newcommand\sym{\operatorname{sym}}
\newcommand\grad{\operatorname{grad}}
\newcommand\deff{\operatorname{def}}
\renewcommand\div{\operatorname{div}}
\renewcommand\ker{{\operatorname{ker}}}

\newcommand\curl{\operatorname{curl}}
\newcommand\rot{\operatorname{rot}}

\newcommand\dev{\operatorname{dev}}

\newcommand\alt{\mathrm{Alt}}

\newcommand\hess{\operatorname{hess}}

\newcommand\ran{\mathrm{Ran}}

\newcommand\R{\mathbb{R}}

\newcommand\Alt{{\mathrm{Alt}}}

\usepackage{cite}

\usepackage{sansmath}

\usepackage{chngpage}

\usepackage{lipsum}

\usepackage{booktabs}

\newcommand{\bs}{{\scriptscriptstyle \bullet}}

\makeatletter
\@addtoreset{equation}{section}
\makeatother

\usepackage{geometry}
\usepackage[OMLmathbf,OMLmathsf]{isomath}
\renewcommand{\vec}{\vectorsym}
\newcommand{\ten}{\tensorsym}

\usepackage{dsfont}

\newcounter{quotecount}

\begin{document}
\title{Many facets of cohomology}
\subtitle{Differential complexes and structure-aware formulations}

 \author{Kaibo Hu}
 \address{Mathematical Institute, University of Oxford,
Radcliffe Observatory, Andrew Wiles Building,
Woodstock Rd, Oxford OX2 6GG, UK.}
\email{kaibo.hu@maths.ox.ac.uk}
\date{\today. }

\maketitle

\begin{abstract}


Complexes and cohomology, traditionally central to topology, have emerged as fundamental tools across applied mathematics and the sciences. This survey explores their roles in diverse areas, from partial differential equations and continuum mechanics to reformulations of the Einstein equations and network theory. Motivated by advances in compatible and structure-preserving discretisation such as Finite Element Exterior Calculus (FEEC), we examine how differential complexes encode critical properties such as existence, uniqueness, stability and rigidity of solutions to differential equations. We demonstrate that various fundamental concepts and models in solid and fluid mechanics are essentially formulated in terms of differential complexes.   

\end{abstract}

\tableofcontents

\section{Complexes and cohomology: a perspective from solving equations}

A basic question in the sciences is to clarify the relationships among physical observables. Newton's law
\begin{equation}\label{eqn:newton}
\bm f = m \bm a = m \ddot{\bm x},
\end{equation}
indicates that the force exerted on a body is proportional to its acceleration. The (incompressible) Navier--Stokes equations
\begin{subeqnarray}\label{NS}
\bm u_{t} + (\bm u \cdot \nabla)\, \bm u - R_{e}^{-1}\Delta \bm u + \nabla p &=& \bm f, \\
\nabla \cdot \bm u &=& 0,
\end{subeqnarray}
indicate a relationship between the velocity $\bm u$, the pressure $p$, and the external force $\bm f$ on a flow. This is a nonlinear differential and algebraic relationship. Solving differential equations is fundamental: 
Newton’s famous statement, often paraphrased as ``The laws of nature are expressed by differential equations'' \cite{arnold1992ordinary}, emphasises their fundamental roles.

 Both \eqref{eqn:newton} and \eqref{NS} involve continuous space and time. Through the developments from Newton and Leibniz to Cauchy, Bolzano, and Weierstrass, \emph{calculus} established rigorous concepts of limits and convergence of discrete objects (such as values on a lattice) to continuous counterparts (such as functions) via infinitesimal calculus. Since then, the interaction between discrete and continuous has been a central topic in mathematics and physics \cite{bell2019continuous}.

Starting in the 20th century, a reverse trend of translating continuous theories back to discrete ones emerged. Quantum theory and numerical computation have been two major motivations: the former aims to quantise continuous theories such as general relativity, and the latter aims to reduce continuous problems to a finite number of degrees of freedom to fit them into computer simulations. Both directions raised the fundamental question of finding discrete counterparts to continuous systems. A natural idea would be to revisit the invention of calculus and replace derivatives and limits with finite differences. Indeed, this led to the rise of finite difference methods, playing a crucial role from the early days of von Neumann’s applications \cite{vonneumann1950method} to modern computational fluid dynamics and numerous other areas. 

\medskip
{\noindent\bf Compatible discretisation.} Numerical methods may face subtle and serious issues. Puzzling examples come from computational electromagnetism and fluid dynamics. Extending finite element or finite difference schemes (that are standard for scalar problems) to solving the Maxwell equations or the Navier--Stokes equations can lead to spurious (wrong) numerical solutions. These issues reflect a fundamental question: 
\begin{quote}
\emph{how to discretise a system with multiple variables?}
\end{quote}
The problem of discretising a multi-field system fits into the broader topic of \emph{compatible discretisation}. 





The examples of discretising Maxwell’s equations and the Navier--Stokes equations show that discretisations of different variables must satisfy certain conditions, rather than being chosen arbitrarily (such as replacing all derivatives by standard finite differences). Consequently, a long line of research has sought to identify such conditions and design \emph{discrete patterns} (proper staggered grids or finite elements -- for the latter, how to distribute degrees of freedom over a triangulation and how to construct local shape functions to match those degrees of freedom) for different variables and problems.

The principle of identifying such compatibility conditions for multi-variable systems has been established in what is known as the Ladyzhenskaya-Babu\v{s}ka-Brezzi (LBB) condition. Different problems have different versions of the LBB condition, typically involving conditions on differential operators connecting the spaces of different variables. For incompressible fluids, for example, the velocity space must map \emph{onto} the pressure space via the $\div$ operator. Such relations are summarised in diagrams and complexes like
\begin{equation}\label{deRham}
\begin{tikzcd}[row sep=tiny, column sep = large]
0 \arrow{r} & V^{0}  \arrow{r}{\grad} & V^{1}  \arrow{r}{\curl} & V^{2}  \arrow{r}{\div} & V^{3}  \arrow{r}{} & 0\\
&&&\mbox{velocity}&\mbox{pressure}&\\\vspace{-0.2cm}
&&\mbox{electric field}&\mbox{magnetic field}&&
\end{tikzcd}
\end{equation}
Recall that a sequence of spaces $V^{k}, k=0, 1, \cdots$ and operators $d^{k}: V^{k}\to V^{k+1}, k=0, 1, \cdots$ is called a {\it complex} if $d^{k+1}\circ d^{k}=0$, for any $k$. The $k$-th cohomology is defined as the quotient space $\mathscr{H}^{k}:=\ker (d^{k})/\ran(d^{k-1})$. Roughly speaking, the $k$-th cohomology includes those fields fields $u\in V^{k}$ satisfying $d^{k}u=0$ but not in the range of the previous operator $d^{k-1}V^{k-1}$\footnote{For example, \eqref{deRham} is a complex since $\curl\circ \grad=0$ and $\div\circ \curl =0$. The cohomology at $V^{0}$ can be represented by the kernel of $\grad$; the cohomology at $V^{1}$ can be represented by $\curl$-free fields which cannot be written as a gradient; the cohomology at $V^{2}$ is represented by $\div$-free fields that cannot be written as a curl; and the cohomology at $V^{3}$ is represented by those elements in $V^{3}$ that are not in the range $\div V^{2}$}. 

In fact, \emph{differential complexes} (sequences of spaces connected by differential operators) such as \eqref{deRham} and the homological algebra for them not only hold for special cases, but reflects rather general principles for solving equations, such as existence, uniqueness and stability of solutions.
\begin{tcolorbox}[
    enhanced, breakable,
    toggle enlargement=evenpage,
    colback=red!5!white,colframe=red!75!black,
    title={Homological algebra for solving equations}]
Natural mathematical questions arising from solving equations are related to the \emph{existence}, \emph{uniqueness}, and \emph{stability} of solutions. We can illustrate these concepts somewhat formally using \emph{diagrams}. Suppose we want to solve
\begin{equation}\label{eqn:general}
\mathscr{L}(w) = f,
\end{equation}
where $f$ is an external force and $w$ is the variable to solve for. Note that $w$ may be a combination of several physical quantities. For example, in \eqref{eqn:newton}, given the mass $m$ and the external force $\bm f$, one solves for the position $w := \bm x$, where $\mathscr{L} := \frac{d^{2}}{dt^{2}}$ is the second order derivative; in the Navier--Stokes equations \eqref{NS}, given the force $\bm f$, one solves for velocity and pressure $w := (\bm u, p)$, and the (nonlinear) operator $\mathscr{L}$ is defined by those equations, mapping $(\bm u, p)$ to $\bm f$.

Solving the equation \eqref{eqn:general} can be stated as: given $f$ in a data space $G$, does there exist a solution $w$ in a solution space $W$? Then \emph{existence} is equivalent to whether the map $\mathscr{L} : W \to G$ is surjective. In diagram form, we say
\[
\begin{tikzcd}
W \arrow{r}{\mathscr{L}} & G \arrow{r}{} & 0
\end{tikzcd}
\]
is \emph{exact} at $G$, meaning that any element in $G$ mapped to zero must come from $\mathscr{L}$ of some element in $W$. By definition of the two arrows, exactness is precisely the surjectivity of $\mathscr{L}$, thus the existence of solutions to \eqref{eqn:general} (for general right hand sides). \emph{Uniqueness} (injectivity of $\mathscr{L}$) is similarly equivalent to the exactness of
\[
\begin{tikzcd}
0 \arrow{r}{} & W \arrow{r}{\mathscr{L}} & G,
\end{tikzcd}
\]
i.e., anything in $W$ mapped to $0$ by $\mathscr{L}$ must be the zero element in $W$. In summary, the existence and uniqueness of \eqref{eqn:general} boil down to the exactness at both $W$ and $G$ of
\[
\begin{tikzcd}
0 \arrow{r}{} & W \arrow{r}{\mathscr{L}} & G \arrow{r}{} & 0.
\end{tikzcd}
\]
If the uniqueness does not hold, we may talk about \emph{rigidity}:
\[
\begin{tikzcd}
0 \arrow{r}{} & V \arrow{r}{\mathscr{T}} & W \arrow{r}{\mathscr{L}} &  G \arrow{r}{} & 0.
\end{tikzcd}
\]
Here we add another space $V$, and elements in $V$ are mapped into $W$ by $\mathscr{T}$. Relevant to our problem, we require that the image of $\mathscr{T}$ is mapped to zero by $\mathscr{L}$, i.e., $\mathscr{L} \circ \mathscr{T} = 0$. Then the nonuniqueness of $\mathscr{L}$ can be expressed by the exactness at $W$: if $w \in W$ satisfies $\mathscr{L}(w) = 0$, there exists $v \in V$ such that $w = \mathscr{T}(v)$. In other words, the solution to \eqref{eqn:general} is unique only up to elements in $V$. Another way of putting this is: if $w_1$ and $w_2$ have the same image under $\mathscr{L}$, then $w_1 - w_2$ is in the image of $\mathscr{T}$.

Nonexistence or nonuniqueness can be also phrased as the presence of nontrivial cohomology.

\emph{Stability} means that small perturbations in the data yield only small errors in the solutions. Assume $W$ and $G$ are Banach spaces with norms $\|\cdot\|_W$ and $\|\cdot\|_G$, respectively. Stability can be expressed by boundedness of the solution operator $\mathscr{L}^{-1}$: there exists a positive constant $C$ such that $\|\mathscr{L}^{-1}\|_{G \to W} \leq C$, where
\[
\|\mathscr{L}^{-1}\|_{G \to W} := \sup_{0 \neq g \in G}\,\frac{\|\mathscr{L}^{-1}(g)\|_W}{\|g\|_G}.
\]
In other terms, if $\|f\|_G$ is small, then the corresponding solution $\|w\|_W$ remains small when $\mathscr{L}$ is invertible with a bounded inverse. Another standard form is the \emph{Ladyzhenskaya-Babu\v{s}ka-Brezzi inf-sup condition}:
\[
\exists \,\gamma > 0 \quad \text{such that}\quad
\inf_{g \in G}\,\sup_{v \in W}\,\frac{(\mathscr{L}v,\, g)}{\|v\|_W\,\|g\|_G} \;\geq\; \gamma \;>\;0.
\]
All these considerations indicate that complexes and cohomology are fundamental for equations and models behind them, as we will discuss in detail in this paper.
\end{tcolorbox}

On the discrete level, the LBB conditions guide the design of finite element spaces or discrete patterns. Concretely, the LBB condition requires that discrete spaces (e.g., finite elements) $V^{i}$ fit into \eqref{deRham} and satisfy corresponding algebraic relations (for instance, $\div V^{2} = V^{3}$). Pioneering numerical analysts such as N\'ed\'elec, Raviart, and Thomas developed these finite elements individually \cite{Nedelec.J.1986a,Nedelec.J.1980a,brezzi1985two,Raviart.P;Thomas.J.1977a}. Later, Bossavit observed that they can be interpreted as a complex and are special cases of \emph{Whitney forms} \cite{hiptmair2001higher,whitney2012geometric,Bossavit.A.1998a}. This observation led to finite element differential forms and the foundation of finite element exterior calculus \cite{hiptmair2001higher,arnold2018finite,Arnold.D;Falk.R;Winther.R.2006a,Arnold.D;Falk.R;Winther.R.2010a}.

Although numerics is not the focus of this paper, as a motivation, we include two numerical examples to demonstrate the significance of compatible discretisation (particular with respect to differential complex structures) for computation. 
\begin{example}[compatible discretisation of multi-fields]
Let $\Omega$ be a bounded polyhedral domain. 
We recall the finite element methods for solving the Poisson equation on $\Omega$ with homogeneous boundary condition
\begin{subeqnarray}\label{eqn:poisson}
-\Delta u=f, \quad\rm{~in~} \Omega, \\
u=0, \quad\rm{~on~} \partial \Omega.
\end{subeqnarray}
Testing the equation with $v$ and integrate by parts, we get the weak form:
 find $u\in H_{0}^{1}(\Omega)$, such that
$$
\int_{\Omega} \nabla u\cdot \nabla v\, dx=\int_{\Omega}fv, \quad\forall v\in H_{0}^{1}(\Omega). 
$$
The Galerkin method is a discretisation of the weak formulation by replacing the infinite dimensional space $H^{1}_{0}$ by a finite dimensional subspace $V_{h}\subset H^{1}_{0}$:
 find $u_{h}\in V_{h}$, such that
$$
\int_{\Omega} \nabla u_{h}\cdot \nabla v_{h}\, dx=\int_{\Omega}fv_{h}, \quad\forall v_{h}\in V_{h}. 
$$
Different choices of $V_{h}$ lead to different numerical methods. Let $\mathcal{T}_{h}$ be a triangulation of $\Omega$. The Lagrange finite element space consists of continuous functions which are piecewise polynomials: 
$$
V_{h}:=\{u_{h}|_{T}\in \mathcal{P}_{1}(T), \forall T\in \mathcal{T}_{h}\}\cap H^{1}_{0}(\Omega). 
$$
This definition implies that $V_{h}\subset C^{0}(\Omega)$. If $\mathcal{T}_{h}$ consists of simplices (more precisely, if $\mathcal{T}_{h}$ is a simplicial complex), then $V_{h}$ is spanned by the hat function at each interior vertex which is equal to one at this vertex and linearly decays in the patch of the vertex (see Figure \ref{fig:hat-function}). Detailed discussions on finite element methods can be found in standard textbooks, e.g., \cite{brenner2008mathematical,brezzi2008mixed,philippe1978finite}.
\begin{figure}[htbp]
\begin{center}
\includegraphics[width=1.5in]{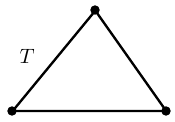} \quad\quad
\includegraphics[width=1.5in]{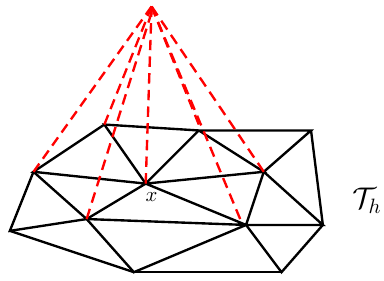} 
\caption{Hat function on a triangulation $\mathcal{T}_{h}$.}
\label{fig:hat-function}
\end{center}
\end{figure}

Consider an example where $f=1$ and $\Omega$ is the unit square. The following solution is obtained with NGSolve \index{NGSolve} \cite{schoberl2014c++} (see Figure \ref{fig:poisson-primal-ngsolve}).
\begin{figure}[htbp]
\centering
\begin{subfigure} 
 { \includegraphics[width=.50\linewidth]{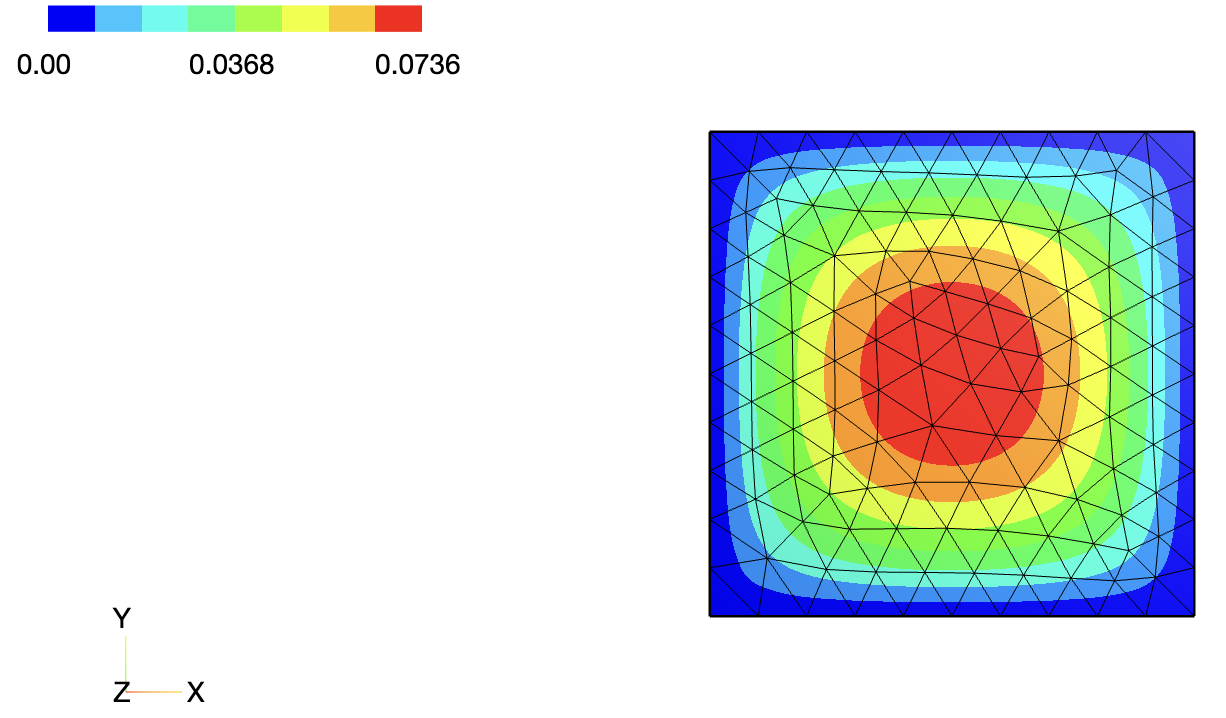}}
\end{subfigure}%
\begin{subfigure} 
  \centering
  \includegraphics[width=.48\linewidth]{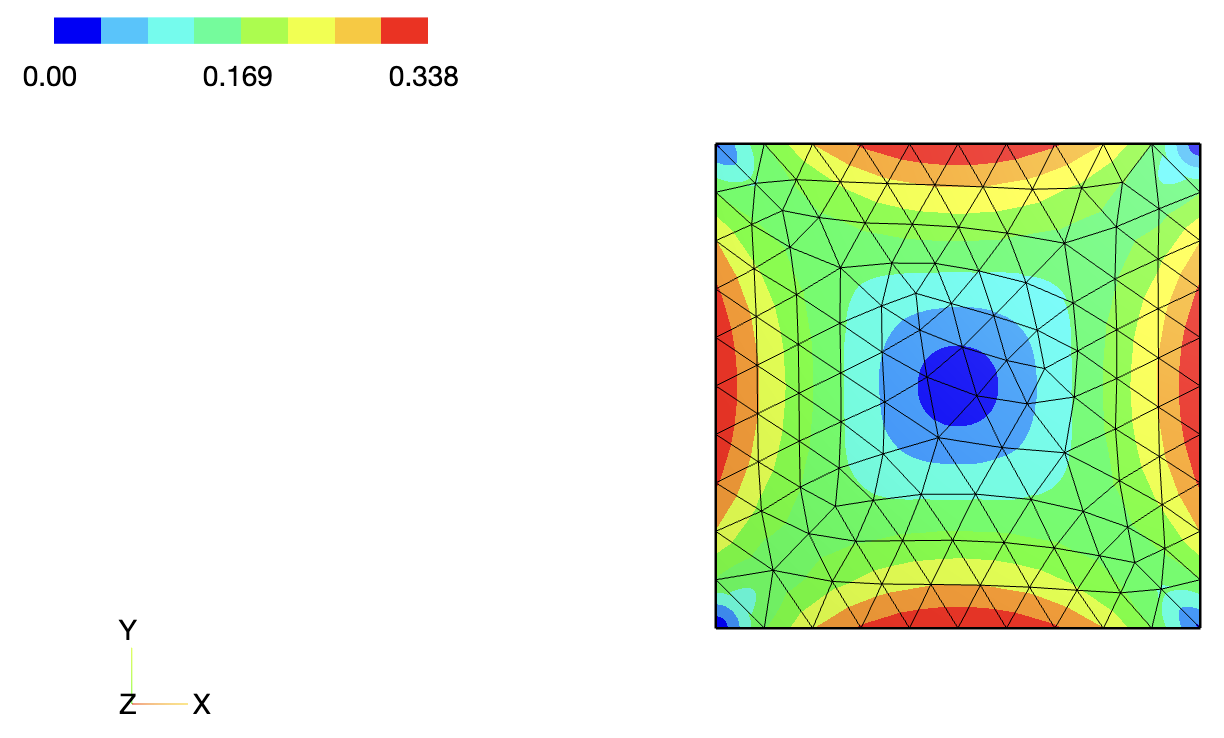}
\end{subfigure}
\caption{Poisson problem on unit square with $f=1$. Left: plot of $u_{h}$, Lagrange elements of second order. Right: plot of the magnitude of $\nabla u_{h}$. }
\label{fig:poisson-primal-ngsolve}
\end{figure}

We may consider another formulation. Introduce $\vec\sigma:=-\grad u$. Then the Poisson equation \eqref{eqn:poisson} has an equivalent form, which we refer to as the {\it mixed form}\index{mixed form}:
\begin{subeqnarray}\label{eqn:poisson-mixed}
\vec\sigma+\grad u&=0,\\
\div\vec \sigma &=f.  
\end{subeqnarray}

When we numerically solve \eqref{eqn:poisson-mixed}, we immediately encounter the problem of how to discretise $\vec \sigma$ and $u$. In the following, we explore several combinations\footnote{When we computed the errors, we used the primal formulation with mesh size 0.01 and piecewise polynomials of degree 3 as the reference solution. }. In particular, we investigate combinations of the Lagrange finite element space $V_{h}$ and its vector version, piecewise polynomials of degree less than or equal to $k$, denoted by $\mathrm{DG}_{k}$, and the Raviart-Thomas space 
$$
\mathrm{RT}_{k}:=\{\sigma|_{T}=\vec{a}+b\vec x: \vec a\in [\mathcal{P}_{k}]^{2}, b\in \mathcal{P}_{k}, \, \forall T\in \mathcal{T}_{h}; \vec \sigma\in H(\div)\}. 
$$


\begin{figure}[htbp]
\centering
\begin{subfigure} 
  \centering
{ \includegraphics[width=.50\linewidth]{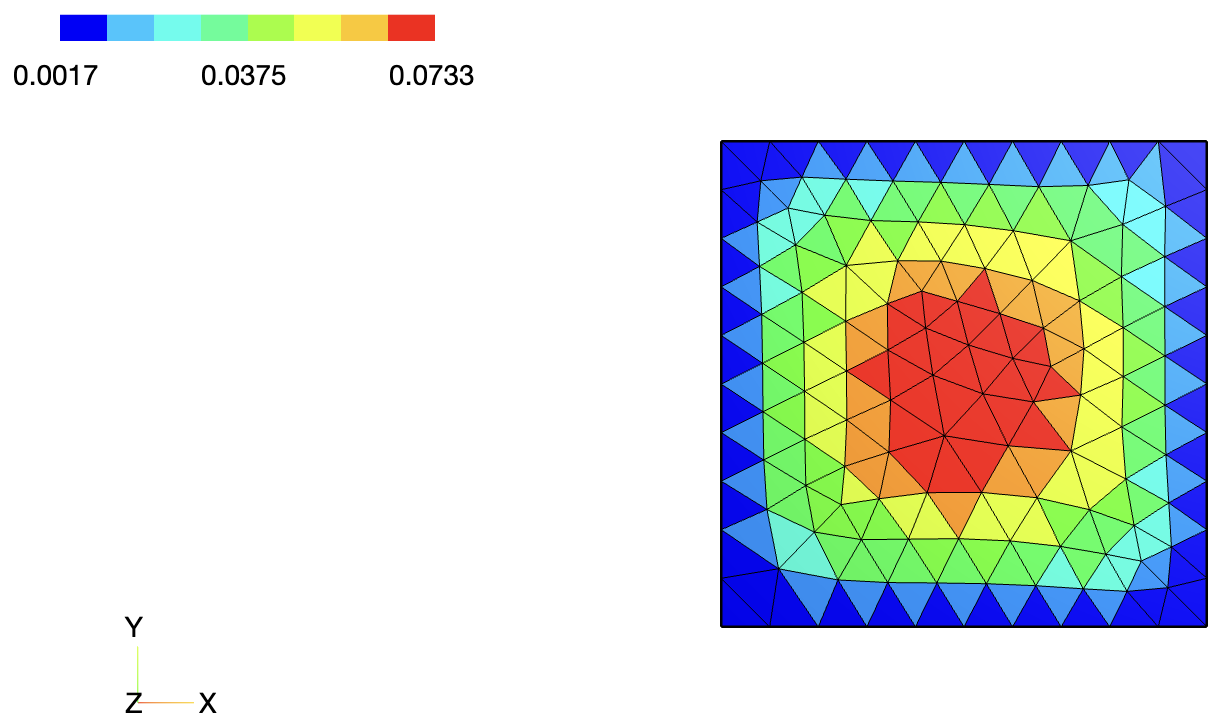}}
\end{subfigure}%
\begin{subfigure} 
  \centering
  \includegraphics[width=.48\linewidth]{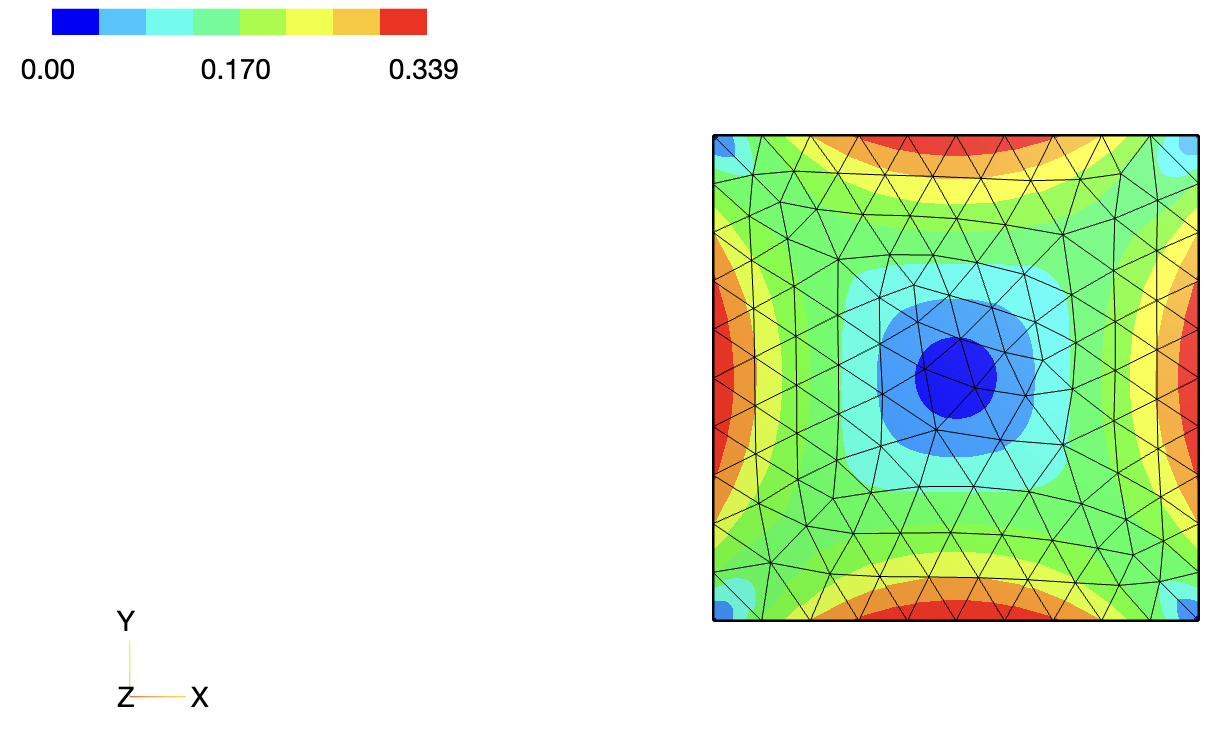}
\end{subfigure}
\caption{Mixed Poisson problem on unit square with $f=1$. Finite element pair: $u_{h}\in \mathrm{RT}_{1}$, the first order Raviart-Thomas element; $\vec \sigma_{h}\in \mathrm{DG}_{0}$, piecewise constants.  Left: plot of $u_{h}$. Right: plot of the magnitude of $\vec \sigma_{h}$. Error $u$: 0.0038, error $\vec\sigma$: 0.0017. Mesh size = 0.1.}
\label{fig:poisson-mixed-ngsolve-1}
\end{figure}

\begin{figure}[htbp]
\centering
\begin{subfigure} 
  \centering
{ \includegraphics[width=.50\linewidth]{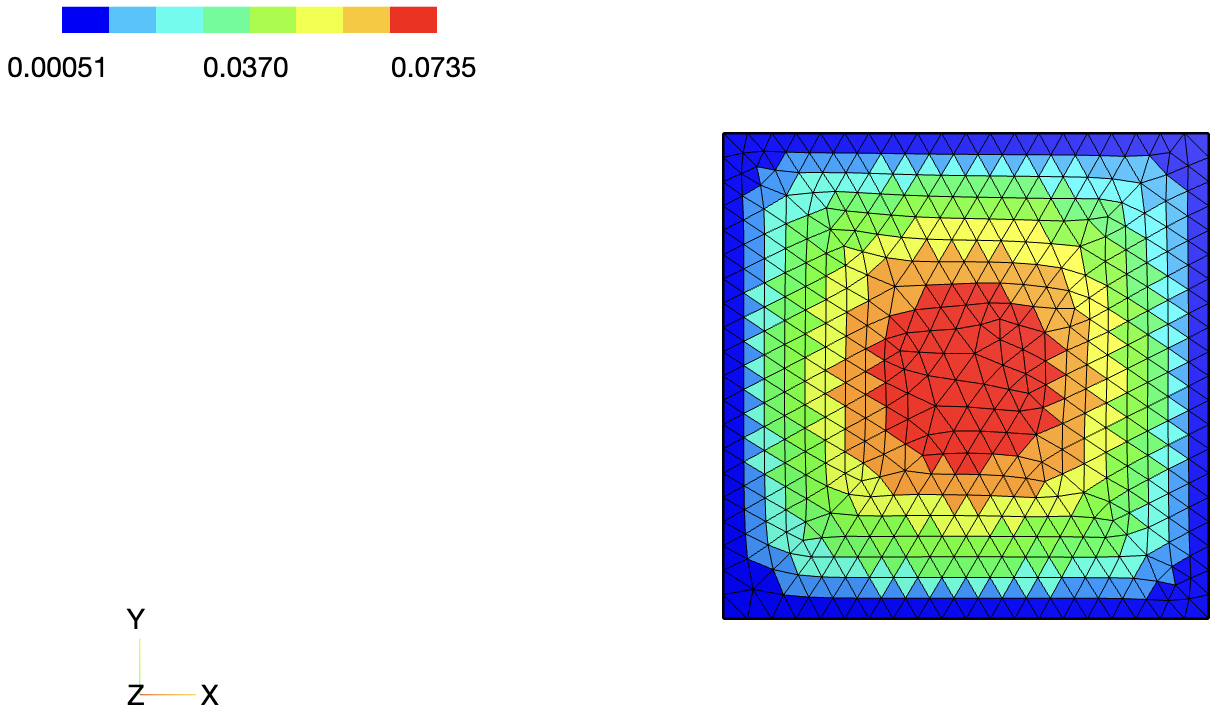}}
\end{subfigure}%
\begin{subfigure} 
  \centering
  \includegraphics[width=.48\linewidth]{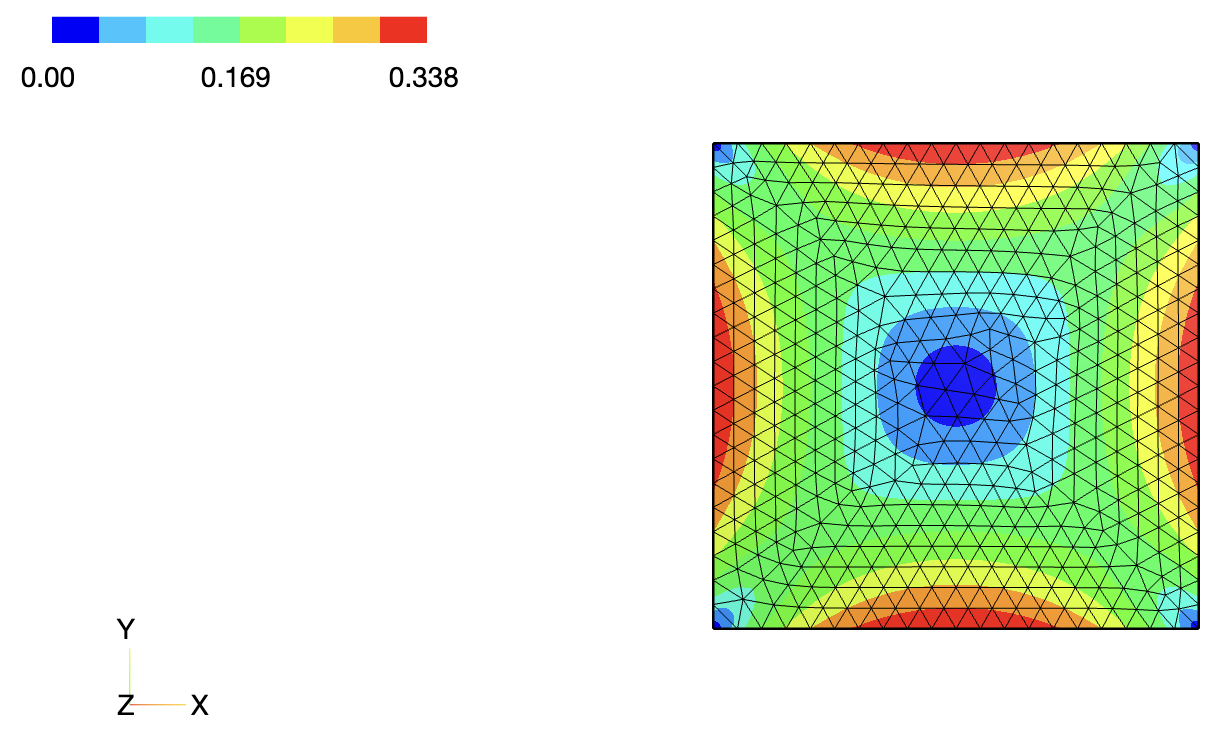}
\end{subfigure}
\caption{Mixed Poisson problem on unit square with $f=1$. Finite element pair: $u_{h}\in \mathrm{RT}_{1}$, the first order Raviart--Thomas element; $\vec\sigma_{h}\in \mathrm{DG}_{0}$, piecewise constants.  Left: plot of $u_{h}$. Right: plot of the magnitude of $\vec\sigma_{h}$. Error $u$: 0.0019, error $\vec\sigma$: 0.0005. Mesh size = 0.05.}
\label{fig:poisson-mixed-ngsolve-1-refined}
\end{figure}

\begin{figure}[htbp]
\centering
\begin{subfigure} 
  \centering
{ \includegraphics[width=.50\linewidth]{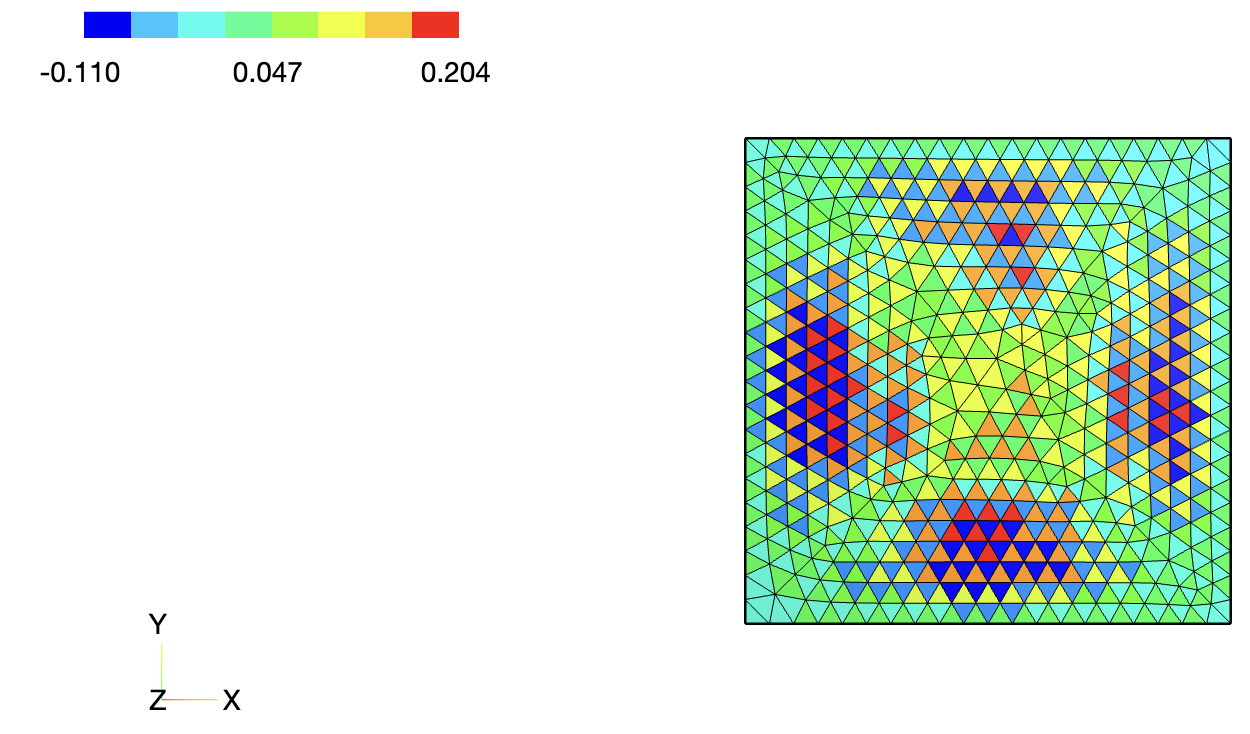}}
\end{subfigure}%
\begin{subfigure} 
  \centering
  \includegraphics[width=.48\linewidth]{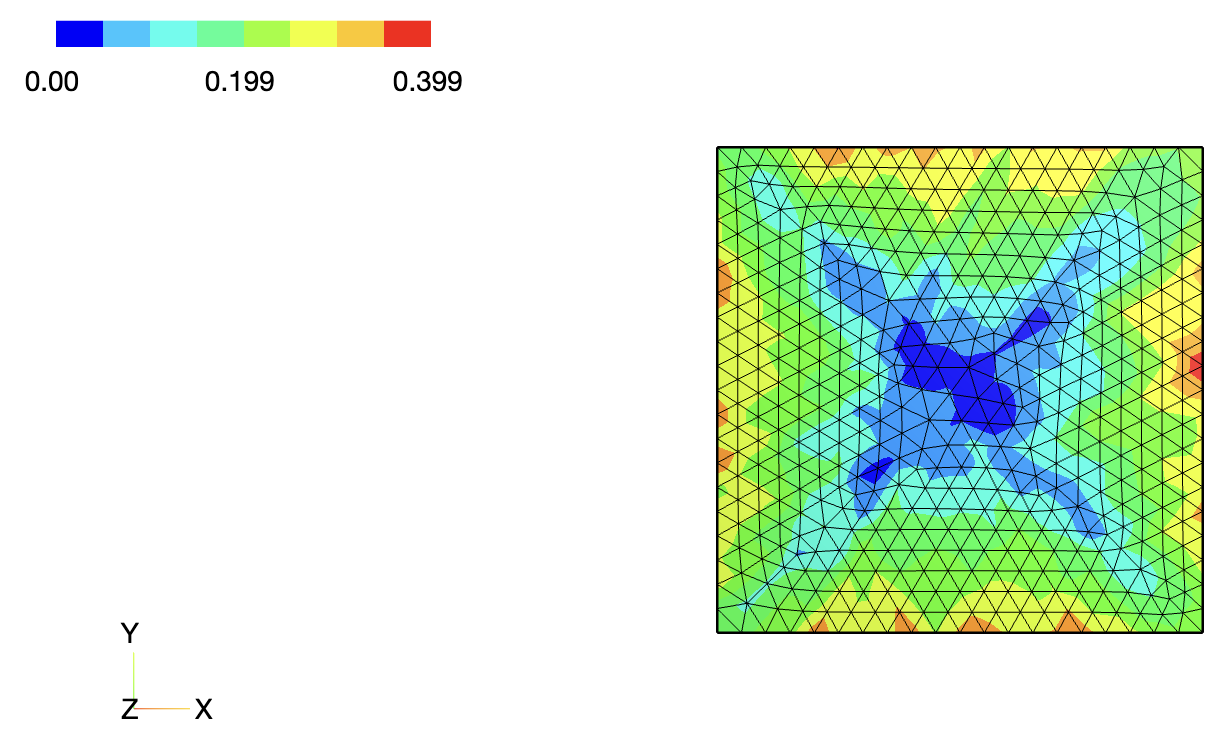}
\end{subfigure}
\caption{Mixed Poisson problem on unit square with $f=1$. Finite element pair: $u_{h}\in C^{0}\mathcal{P}^{1}$, the vector version of the Lagrange spaces; $\vec\sigma_{h}\in \mathrm{DG}_{0}$, piecewise constants.  Left: plot of $u_{h}$. Right: plot of the magnitude of $\vec\sigma_{h}$. Mesh size = 0.05.}
\label{fig:poisson-mixed-ngsolve-C0P1-DG0}
\end{figure}
One may explore other possibilities. For example, one can discretise each component of $\vec \sigma$ by the Lagrange element of degree 1, and $u$ by the Lagrange element of degree 2. The solver fails. In fact, the resulting matrix is singular in this case.

These numerical results demonstrate that only proper combinations of finite element spaces lead to correct convergence.
 
\end{example}

\begin{example}[incompressible flows]
Consider the Stokes system in its weak form:
\begin{subequations}\label{stokes}
\begin{align}
\int_{\Omega} \nabla \bm{u} \cdot \nabla \bm{v} \, dx - \int_{\Omega} p \nabla \cdot \bm{v} \, dx &= \int_{\Omega} \bm{f} \cdot \bm{v} \, dx,  \quad \forall \bm{v},  \\
\int_{\Omega} \nabla \cdot \bm{u} \, q \, dx &= 0,  \quad \forall q, 
\end{align}
\end{subequations}
where $\bm{u}$ is the velocity field (a vector field) and $p$ is the pressure (a scalar function).

We first consider discretisations of the velocity and pressure. In two dimensions, under mild mesh assumptions, if one uses continuous piecewise $\mathcal{P}_{4}$ polynomials\footnote{Hereafter, $\mathcal{P}_{r}$ denotes the space of polynomials of degree less than or equal to $r$.} for velocity and discontinuous piecewise $\mathcal{P}_{3}$ polynomials for pressure, the scheme converges correctly (see Figure~\ref{fig:stokes-solution}, left). However, if velocity is discretised by $\mathcal{P}_{2}$ and pressure by $\mathcal{P}_{1}$, the scheme fails. Interestingly, if one applies the Alfeld split to each triangle (subdividing it into three sub-triangles), the $\mathcal{P}_{2}$–$\mathcal{P}_{1}$ scheme recovers convergence (see Figure~\ref{fig:stokes-solution}, right). This phenomenon shows that, in this example, convergence of the scheme is closely tied to the mesh structure. The computations are carried out using NGSolve\footnote{https://ngsolve.org}\cite{schoberl2014c++}.

\begin{figure}[ht]
    \centering
        \begin{minipage}[b]{0.48\textwidth}
        \centering
        \includegraphics[width=0.5\textwidth]{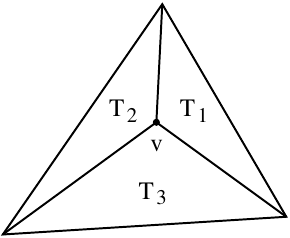}
        \caption{Illustration of the Alfeld split.}
        \label{fig:CT}
    \end{minipage}
\end{figure}

\begin{figure}[ht]
    \centering
    \begin{minipage}[b]{0.48\textwidth}
        \centering
                \includegraphics[width=\textwidth]{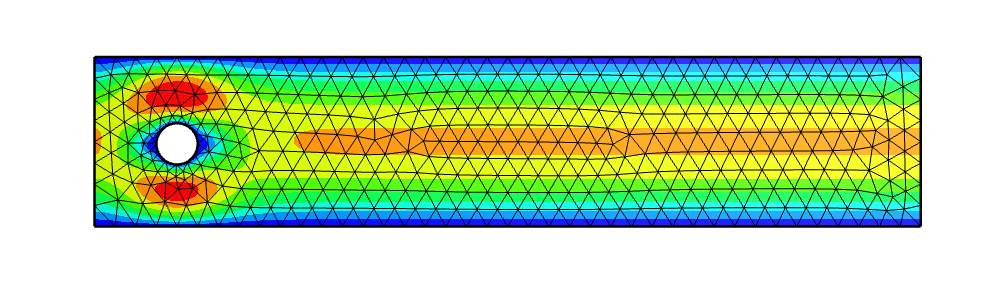}
    \end{minipage}
    \hfill
    \begin{minipage}[b]{0.48\textwidth}
        \centering
        \includegraphics[width=\textwidth]{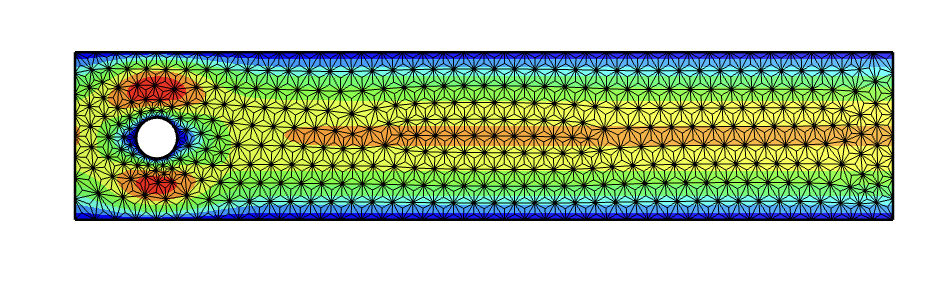}
    \end{minipage}
    \caption{Numerical results for the Stokes problem with different discretisations. Left:  $\mathcal{P}_{4}$–$\mathcal{P}_{3}$ scheme; Right: $\mathcal{P}_{2}$–$\mathcal{P}_{1}$ scheme with Alfeld split.}
    \label{fig:stokes-solution}
\end{figure}

This example demonstrates that in the Stokes problem, the choice of discrete velocity and pressure spaces is crucial.  
To interpret the results, we note that, at the continuous level, the well-posedness of the Stokes problem \eqref{stokes} is guaranteed by the inf-sup condition:
\[
\inf_{q \in L^{2}} \sup_{\bm{v} \in \bm{H}^{1}} 
\frac{\int \div \bm{v} \, q \, dx}{\|\bm{v}\|_{H^{1}} \, \|q\|_{L^{2}}} \geq \gamma > 0.
\]
This condition requires that the velocity and pressure spaces satisfy certain relations: the divergence operator maps the velocity space $\bm{H}^1$ onto the pressure space $L^2$, with a norm control that ensures stability.  
It can be expressed in terms of the exactness and norm control of the following de Rham complex at $L^2$:
\begin{equation}\label{eqn:continuous-stokes}
\begin{tikzcd}[row sep=tiny]
0 \arrow{r} & H^{2} \arrow{r}{\curl} & \bm H^{1} \arrow{r}{\div} & L^{2} \arrow{r} & 0. \\
&& \text{velocity} & \text{pressure} &
\end{tikzcd}
\end{equation}
Thus, the \emph{complex encodes the relation between different variables in a multiphysics system}. 

At the discrete level, stability is guaranteed by a discrete inf-sup condition:
\begin{equation}\label{discrete-inf-sup-stokes}
\inf_{q_h \in Q_h} \sup_{\bm{v}_h \in \bm{V}_h} 
\frac{\int \div \bm{v}_h \, q_h \, dx}{\|\bm{v}_h\|_{H^{1}} \, \|q_h\|_{L^{2}}} \geq \gamma > 0.
\end{equation}
This can be expressed through a discrete de Rham complex:
\begin{equation}\label{eqn:discrete-stokes}
\begin{tikzcd}[row sep=tiny]
0 \arrow{r} & W_h \arrow{r}{\curl} & \bm{V}_h \arrow{r}{\div} & Q_h \arrow{r} & 0. \\
&& \text{velocity} & \text{pressure} &
\end{tikzcd}
\end{equation}
Our aim is to construct a velocity space $\bm{V}_h \subset \bm{H}^1$ and a pressure space $Q_h \subset L^2$ such that $\div \bm{V}_h = Q_h$. Although this algebraic condition alone does not imply \eqref{discrete-inf-sup-stokes}, it is the first and most essential step toward establishing the discrete inf-sup condition. Designing such Stokes velocity–pressure pairs has long been a challenging problem. The differential complex provides a new perspective: from \eqref{eqn:continuous-stokes} and \eqref{eqn:discrete-stokes} we see that, in order to obtain the spaces $\bm{V}_h$ and $Q_h$, one essentially needs to construct a discrete complex consisting of finite element spaces. This fundamentally involves discretising the first space $H^2$, which is the main source of difficulty.  

In fact, discretising $H^2$ requires $C^1$ finite element or spline spaces. Constructing such spaces involves deep and delicate questions in algebraic geometry (for instance, the dimension of $C^1$ piecewise polynomial spaces on a triangulation \cite{lanini2025approximation,billera1988homology, billera1991dimenstion}). These issues propagate through the complexes \eqref{eqn:continuous-stokes} and \eqref{eqn:discrete-stokes} into the construction of Stokes elements, which explains why many open problems remain in higher-dimensional Stokes elements \cite{guzman2019scott}.

Returning to the numerical results: on an Alfeld split, one can define piecewise cubic $\mathcal{C}^1$ scalar splines \cite{lai2007spline} to discretise $W_h$ in \eqref{eqn:discrete-stokes}. By differentiation, this yields continuous $\mathcal{P}_2$ polynomials for the velocity space on the same split, paired with discontinuous $\mathcal{P}_1$ polynomials for the pressure space (see Figure~\ref{fig:stokes-complex-alfeld}). This finite element pair satisfies the key condition $\div \bm{V}_h = Q_h$ as well as the discrete inf-sup condition \eqref{discrete-inf-sup-stokes}. By contrast, such cubic $\mathcal{C}^1$ spline spaces do not exist on general triangulations. This explains why the scheme fails on general meshes, but converges to the correct solution on the Alfeld split.

\begin{figure}[ht]
    \centering
    \includegraphics[height=0.21\textheight]{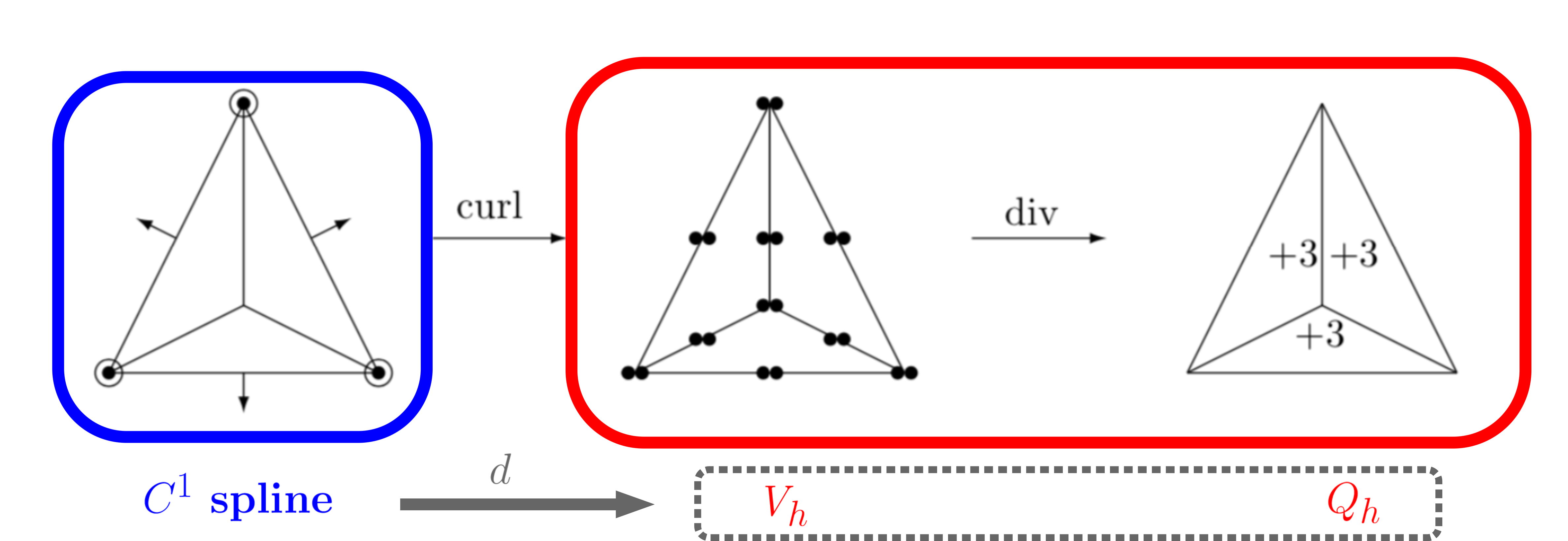}
    \caption{The Stokes complex on the Alfeld split. By differentiating a $C^1$ scalar spline space, one obtains the velocity and pressure spaces.}
    \label{fig:stokes-complex-alfeld}
\end{figure}

The above velocity–pressure pair on the Alfeld split was first proposed by Arnold and Qin \cite{arnold1992quadratic}, and later analysed from the viewpoint of complexes in \cite{christiansen2018generalized}. This construction illustrates a general strategy: to build Stokes finite elements by differentiating scalar spline spaces with high continuity (see Figure~\ref{fig:stokes-complex-alfeld}). In combination with existing results on spline spaces \cite{lai2007spline}, this approach allows us to reinterpret known results for the Stokes problem from the spline perspective, and to derive new constructions. In this way, it not only deepens our understanding of the discrete de Rham complexes, but also provides new theoretical foundations for stable and efficient discretisations of three-dimensional fluid problems.

 \end{example}




\medskip
{\noindent\bf Structure-preserving discretisation}
Another general philosophy in numerical mathematics is \emph{structure-preserving discretisation}. The importance of this subject is based on the fact that formulations that are equivalent at the continuous level can exhibit rather different numerical behaviours \cite{fengkang-fem,hairer2006geometric}. A striking example is geometric integrators \cite{hairer2006geometric,feng1987symplectic}: for the numerical solutions of classical mechanics with ordinary differential equations (ODEs), discretising the Newtonian formulation directly can lead to significant errors in the long-term evolution; meanwhile, discretising a Hamiltonian formulation to preserve the symplectic structure yields stable computations. Below, we demonstrate a related example demonstrating similar issues in computing the long-term behaviour of magnetohydrodynamics, where both \emph{temporal} and \emph{spatial} discretisations are crucial for obtaining physically correct solutions. 

There are many kinds of structures. A broad class of structure-preserving problems focus on preserving certain conserved quantities, such as mass, energy, helicity, or enstrophy of fluid systems \cite{gawlik2022finite,cotter2023compatible,arnold1999topological}. Another major class focuses on constraints. In many field theories (electromagnetism, general relativity, Yang-Mills fields, etc.), evolution equations are accompanied by algebraic and/or differential constraints, which are automatically satisfied by the continuous evolution. Preserving those constraints numerically is often crucial (e.g., in magnetohydrodynamics and numerical relativity) \cite{Brackbill.J;Barnes.D.1980a,Toth.G.2000a,alcubierre2008introduction,gourgoulhon20123}, because violating them may destroy essential mathematical properties and lead to numerical instability.   There are many additional examples, such as positivity preservation \cite{leonard1993positivity,hu2013positivity} or asymptotic preservation \cite{jin2022asymptotic}.



As we have seen, the existence and uniqueness of solutions can be viewed as certain cohomologies being trivial.  
Therefore, ensuring the well-posedness of numerical formulations can be viewed as a special type of {structure-preservation}, i.e., \emph{cohomology-preservation}. Moreover, nontrivial topologies of computational domains lead to nontrivial cohomology. Failure of dealing with such nontrivial cohomology properly also leads to spurious solutions \cite{Arnold.D;Falk.R;Winther.R.2010a,arnold2018finite}.  The key to all these cases is \emph{respecting the algebraic and differential relations among the variables}, often summarised as \emph{preserving differential complexes} such as \eqref{deRham}.

\begin{example}[long-term evolution]\label{example:relaxation}

A fundamental problem in plasma physics is to study how a magnetic field given at the initial time evolves in a long time (whether a stationary state exists and if so, what are the asymptotic properties etc.). This problem is sometimes referred to as magnetic relaxation. 
Below we provide a numerical example and compare a plain algorithm with an algorithm that preserves a conserved quantity \cite{he2025topology}. In the following, we solve a slightly simplified equation of the full MHD system, called the magneto-frictional system \cite{chodura19813d}:
\begin{subequations}\label{eqn:magneto-frictional-equations}
\begin{align}
    \label{eqn:magnetic-advection}\partial_t \bm{B} + \curl\bm{E}&=\bm{0},\\
    \label{eqn:electric-field}\bm{E}+ \bm{u}\times\bm{B} &= \bm{0}, \\
    \bm{j}& = \curl\bm{B}, \\
    \label{eqn:velocity}\bm{u} &= \tau \bm{j}\times\bm{B},
\end{align}
\end{subequations}
with initial data $\bm{B}|_{t=0} = \bm{B}_0$ satisfying the magnetic Gauss law $\div\bm B_0 = 0$.
 Here, the system is closed by the last equation  for given $\tau > 0$ which guarantees the energy decay
 $$
 \frac{1}{2}\frac{d}{dt}\|\bm B\|^{2}=-\tau\|\bm B\times \bm j\|^{2},
 $$
 whereas the full MHD system is closed by coupling with the Navier--Stokes equation. 
      
      We use the Hopf fibration \cite{smiet2017ideal} as initial data:
\begin{equation*}
    \bm B_{0} = \frac{4\sqrt{a}}{\pi(1+r^2)^3}(2y(y-xz),-2(x+yz),(-1+x^2+y^2-z^2)).
\end{equation*}
\smallskip
The Hopf fiberation is highly helical, as every single field line of this field is a perfect circle, and every single field line is linked with every other one \cite{smiet2017ideal}. We use $\tau = 10$, $dt = 1$ and $T=1000$.  Quadratic-invariant-preserving temporal schemes \cite{andrews2024high}  are used to preserve the helicity.  
 
The results were presented in \cite{he2025topology}. The figure on the left shows that history of the evolution by a scheme conserving the helicity $\int\bm A\cdot\bm B\, dx$, and the figure on the right shows a standard scheme that does not preserve this quantity.
\begin{figure}[htbp]

\centering
\begin{minipage}[t]{0.48\textwidth}
\centering
\includegraphics[width=8.3cm]{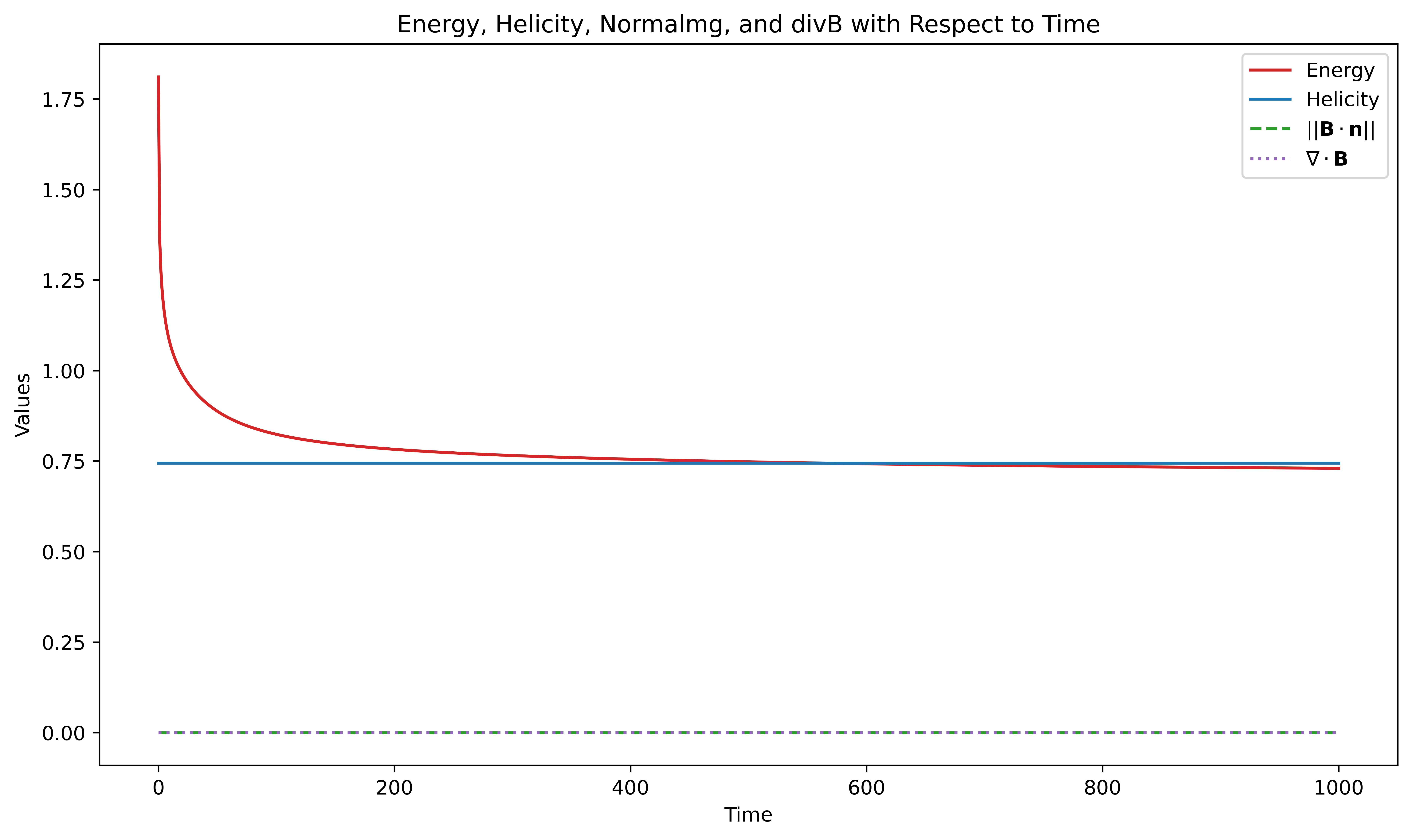}
\caption{Helicity-preserving scheme}
\end{minipage}
\begin{minipage}[t]{0.48\textwidth}
\centering
\includegraphics[width=8.3cm]{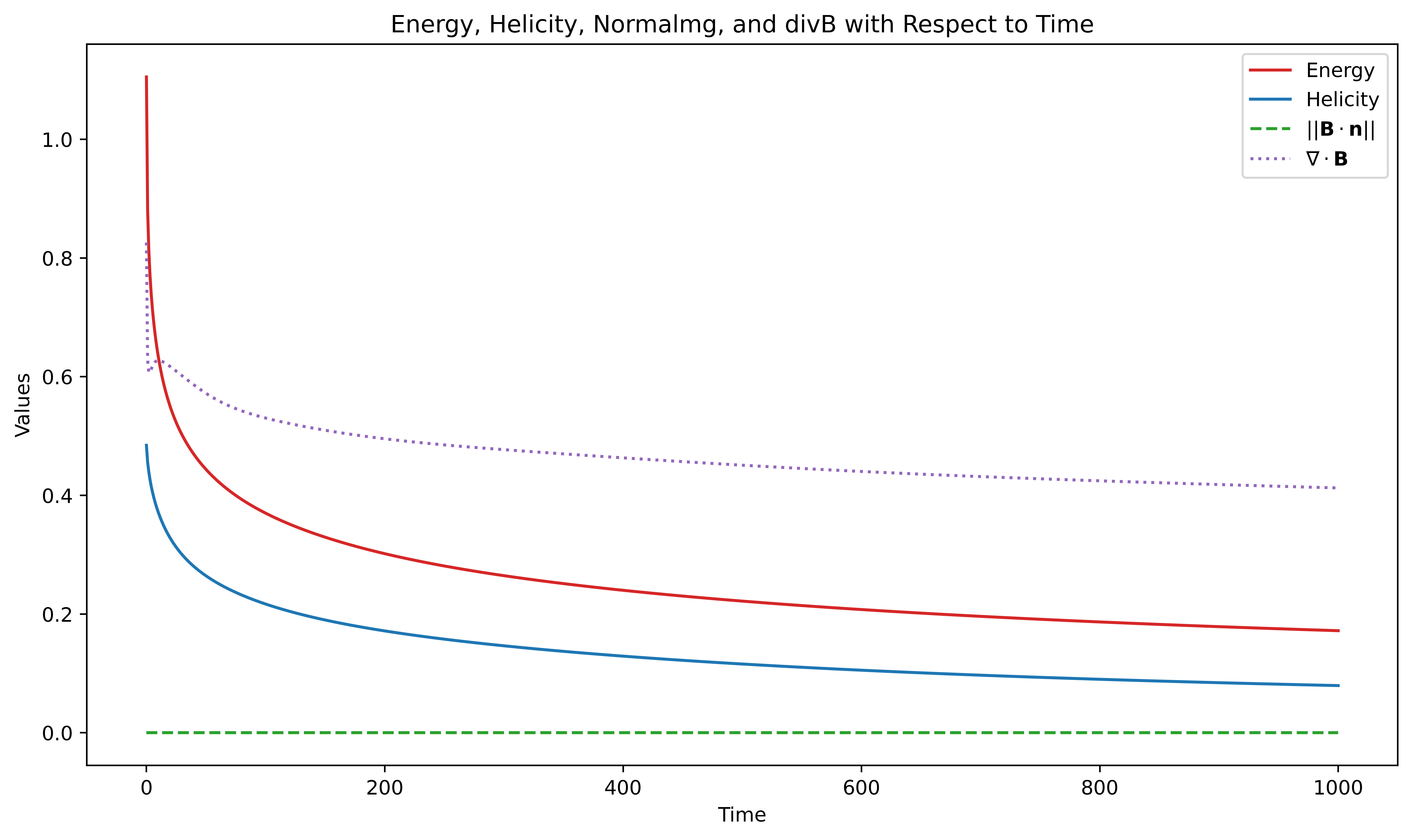}
\caption{Galerkin scheme based on Lagrange finite elements (non-preserving)}
\end{minipage}
\end{figure}
The two schemes show rather different results. This is because the scheme not preserving helicity fails to mimic a crucial topology mechanism for the relaxation problem. Details of this mechanism, as an example of the appearance of complexes and cohomology in fluid dynamics, will be discussed in Section \ref{sec:fluids}.

\end{example}

 \medskip
{\noindent \bf Discrete theories v.s. discretisation.}  
We return to another motivation for developing discrete notations: discrete theories. 

Numerous efforts have sought to reconcile the incompatibility between general relativity and quantum theory, producing various notions of discrete field theories.  For theoretical or computational purposes, various models have been investigated in discrete mechanics \cite{lee1985discrete,lee1987difference}, lattice gauge theory \cite{dalmonte2016lattice,rothe2012lattice}, and quantum field theory and discrete gravity \cite{dupuis2012discrete,immirzi1997quantum} etc. Regge calculus \cite{regge1961general} is one example.
 Other important motivation emerges from discrete differential geometry (DDG) and computer graphics, where one develops a discrete theory which mimics fundamental principles of the continuous version. The main message, as indicated in \cite{bobenko2008discrete}, is to \emph{discretise the whole theory, not just the equations.}

In these areas, compatible discretisation and discrete patterns can be even more vital for at least two reasons. First, \emph{convergence to true solutions} may be secondary. For instance, in computer graphics, obeying physical laws (often certain conservation laws) is vital for realistic visualisation, whereas exact convergence might be less critical. In data sciences or complex systems such as biology \cite{jensen2020force}, first-principle models are often unavailable; so one must rely on coarser structures (e.g., topological connectivity of nodes) rather than precise geometric distances. In such cases, modelling and algorithms built on discrete structures that enforce discrete conservation laws are crucial. For instance, discrete differential forms can be defined on cliques of a graph with applications to ranking, data representation, and deep learning, etc. \cite{lim2020hodge}.  Second,  if the nature of the world were discrete, then one should regard discrete models as fundamental objects and investigate their convergence in the continuum limit, rather than first having the continuum models and discretising them.

In numerical analysis, these two traditions also coexist. Traditional finite element methods can be generalised to arbitrary order with a rigorous theory that mirrors the continuous structures. However, these function-space-based methods can be more rigid (e.g., it is challenging to handle higher-order operators). On the other hand, lattice theories (finite difference with MAC/Yee grids, Discrete Exterior Calculus, discrete differential geometry, lattice gauge theory) encode topological or geometric structures in a purely discrete form, are flexible (sometimes extendable to graphs), and are physically intuitive; yet their convergence and many properties may be less understood.

\medskip
\noindent
{\bf This survey paper: structure-aware formulations.}
The first question in structure-preserving numerical computation is \emph{which structures} should be preserved. For specific problems, this typically boils down to understanding the key mathematical or physical principles and \emph{choosing suitable variables.} Concerning the fundamental role of differential complexes, we emphasise \emph{structure-aware formulations via differential complexes} in this paper. First, the resulting models and formulations enjoy the well-posedness encoded in complexes. Second, \emph{compatible and structure-preserving discretisation naturally follows by discretising the underlying complexes and preserving the cohomology.}

Each physical problem has its own structural features, encoded in different complexes. Thus, to make the above idea a practical programme, we discuss differential complexes and cohomology in various contexts. On the one hand, for specific applications, one investigates the differential structures therein and formulates the problem using complexes. On the other hand, one examines the construction and properties of complexes to guide modelling and computation. Sometimes surprising connections arise that merge the two directions. Many results from pure mathematical research turn out to coincide with models constructed from a rather different consideration. This demonstrates unified structures in sciences, engineering and mathematics, and the belief that {\it elegant mathematical structures are useful}. 

We briefly mention linear elasticity to demonstrate the above idea. Designing stable finite elements for the Hellinger-Reissner principle (in which both stress and displacement are used as variables) demands a complex that captures the divergence of a symmetric matrix field. This motivated Arnold, Falk, and Winther \cite{arnold2002mixed,Arnold2006a} to seek \emph{elasticity complexes} around 2000:
\begin{equation}\label{elasticity}
\begin{tikzcd}[row sep=tiny]
 C^{\infty}(\Omega; \mathbb{R}^{3})  \arrow{r}{\sym\grad} & C^{\infty}(\Omega; \mathbb{R}^{3\times 3}_{\sym})   \arrow{r}{\curl\circ \mathrm{T}\circ \curl} & C^{\infty}(\Omega; \mathbb{R}^{3\times 3}_{\sym})     \arrow{r}{\div}& C^{\infty}(\Omega; \mathbb{R}^{3}).\\
\mbox{displacement (vector)}&\mbox{strain (sym matrix)}&\mbox{stress (sym matrix)}&\mbox{load (vector)}
\end{tikzcd}
\end{equation}
The complex \eqref{elasticity} is often referred to as the Kröner complex \cite{kroner1960general,kroner1985incompatibility,kroner1981continuum} in mechanics or the Calabi complex \cite{calabi1961compact} in geometry. These efforts led to fruitful interactions between numerical analysts and differential geometers on the so-called Bernstein-Gelfand-Gelfand (BGG) construction, which is a mechanism for deriving various new complexes from the de~Rham complex \cite{arnold2021complexes}. This eventually led to a breakthrough in constructing finite elements for elasticity \cite{arnold2002mixed}.

Most early developments on BGG complexes in numerical analysis, however, focused on specific examples in linear elasticity. It remained unclear whether the entire elasticity complex (and other complexes) enjoy desired algebraic and analytic properties such as those in \cite{Arnold.D;Falk.R;Winther.R.2010a}. Part of the difficulty arose from bridging algebraic and analytic structures, i.e., establishing the BGG construction (a method to derive new complexes from the de~Rham complex) with Sobolev spaces. These challenges were addressed in \cite{arnold2021complexes}, where a systematic way to build new complexes with analytic foundations was established. It turns out that not only does a diverse \emph{zoo of complexes} emerge, but the \emph{machinery} itself is useful. Through the BGG construction, one can transport results proven for the de~Rham complex (electromagnetism) to many other complexes (elasticity, differential geometry, general relativity etc.). An example of this paradigm is the derivation of Poincaré operators. Carrying over these results from the de~Rham complex to the elasticity complex, one reproduces the same integral formula discovered by Cesarò and Volterra in 1906 and 1907 in elasticity and defect theory \cite{volterra1907equilibre,cesaro1906sulle} as a special case \cite{christiansen2020poincare,vcap2023bounded}. The BGG complexes and the so-called \emph{twisted complexes} (intermediate steps in deriving BGG complexes) encode rich physics, thus opening the door to \emph{continuum modelling via differential complexes and homological algebra} (see Section~\ref{sec:solid} below).





\medskip
\noindent
{\bf Scope and organization of this paper.}  
This survey does not attempt to give a comprehensive review of the application of differential complexes to any single scientific topic, which will demand extensive treatments on their own. Moreover, although structure-preserving discretisation, particularly finite element exterior calculus, serves as an important motivation for the development, we will not discuss numerics (except for some motivating examples). This is because the advances in finite element exterior calculus require a separate review.   Instead, in this paper, we aim to give a light overview of multiple topics to illustrate the many facets of differential complexes and cohomology, and their role and potential as fundamental tools for structure-aware formulations and structure-preserving computation.

The rest of the paper is organized as follows. Section~\ref{sec:complexes} introduces the notion of differential complexes and several important examples, including the de~Rham complex and other examples derived from it through the Bernstein-Gelfand-Gelfand (BGG) construction. These complexes provide the foundation for subsequent discussions in specific areas. Section~\ref{sec:analysis} discusses how the algebraic structures imply analytic properties. Section~\ref{sec:solid} addresses how differential complexes can be applied in modelling solid mechanics, covering variational principles for elasticity, continuum defects, microstructure, dimension reduction, and multidimensional problems. Although much of this material is classical, we highlight a differential perspective (twisted and BGG complexes, high-order forms, trace complexes, and \v{C}ech double complexes). Section~\ref{sec:fluids} focuses on fluid mechanics, particularly topological hydrodynamics. The differential structures encoded in the de~Rham complex allow us to define topological concepts such as helicity, knots, and topological mechanisms in fluid and magnetohydrodynamic flows that are crucial for questions in solar and astrophysics (e.g., the Parker hypothesis). Section~\ref{sec:einstein} discusses various formulations of the Einstein equations and their links to differential complexes. 
Section \ref{sec:graph} then focuses on a discrete theory, covering graph-based models.

\section{De~Rham and BGG complexes}\label{sec:complexes}

We use three examples of complexes to motivate this section -- the de~Rham complex, the elasticity complex, and the conformal complex. They represent three different classes, namely, topology, Riemannian geometry and conformal geometry, and yet fit in a unified picture. This section only focuses on smooth functions. 

As a warm up, we consider the {\it de~Rham complex in 1D}. Consider the domain $I=(0, 1)$. The 1D de~Rham complex reads
\begin{equation}\label{deRham1D}
\begin{tikzcd}
0 \arrow{r}{} & C^{\infty}(I) \arrow{r}{\frac{d}{dx}} & C^{\infty}(I)  \arrow{r}{} & 0.
\end{tikzcd}
\end{equation}
As before, the first and the last arrows are zero maps. The cohomology at the two $C^{\infty}$ spaces isomorphic to $ \mathbb{R}$ and 0, respectively, reflecting the fact that $\ker(\frac{d}{dx})$ consists of constants, and $\frac{d}{dx}$ maps $C^{\infty}$ {\it onto} $C^{\infty}$ (the inverse map is given by integral).

Let $\Omega$ be a bounded domain in $\mathbb{R}^{3}$. The de~Rham complex in 3D involves the $\grad$, $\curl$, and $\div$ operators: 
\begin{equation}\label{deRham:3D}
\begin{tikzcd}
0 \arrow{r}{} & C^{\infty}(\Omega) \arrow{r}{\grad}& C^{\infty}(\Omega)\otimes \mathbb{R}^{3} \arrow{r}{\curl} & C^{\infty}(\Omega)\otimes \mathbb{R}^{3} \arrow{r}{\div} & C^{\infty}(\Omega) \arrow{r}{} &  0.
\end{tikzcd}
\end{equation}
Here  $C^{\infty}(\Omega)\otimes \mathbb{R}^{3} $ denotes the space of vector-valued smooth functions. 
The complex property holds as $\curl\grad=0$ (gradient fields have no rotation) and $\div\curl =0$ (rotation fields have no source). As we shall see later, the exactness of the complex (except at index 0, where the cohomology is isomorphic to $\ker(\grad)$, consisting of a constant on each connected component of the domain) holds on domains with trivial topology, i.e., contractible domains\index{contractible domains}. This means that
  $\curl  u=0$ implies $u =\grad\phi$ for some $\phi$, and $\div v=0$ implies $ v =\curl \psi$ for some $ \psi$.

De Rham complexes can be formulated on general differential manifolds in any dimensions. The spaces consist of {\it differential forms}, and the operators are {\it exterior derivatives}. Differential forms are skew-symmetric tensor fields and exterior derivatives are the skew-symmetric components of full derivatives. The formalism was introduced by \'Elie Cartan \cite{cartan1899certaines} and now is referred to as {\it exterior calculus}. The skew-symmetry turns out to be the source of many magics. 
 The above examples \eqref{deRham1D} and \eqref{deRham:3D} can be viewed as this general construction written in coordinates.

Different problems involve different operators and variables. In elasticity, the displacement is a vector field (when a metric is in place, a vector field can be identified as a 1-form). The linearised deformation is $e:=\deff(u)=\sym\grad u$.  The {\it elasticity complex} extends vector fields and symmetric gradient: 
\begin{equation}\label{elasticity}
\begin{tikzcd}
 0\arrow{r} &    C^{\infty}\otimes \mathbb{V} \arrow[r, "\deff"] &C^{\infty}\otimes \mathbb{S} \arrow[r, "\inc"] & C^{\infty}\otimes \mathbb{S} \arrow[r, "\div"] &C^{\infty}\otimes \mathbb{V} \arrow[r] & 0.
\end{tikzcd}
\end{equation}
Here $\mathbb{V}:=\mathbb{R}^{3}$, $\mathbb{S}$ denotes the (finite dimensional) space of $3\times 3$ symmetric matrices, and $\inc u:=\curl \circ\operatorname{T}\circ\curl u$ takes column-wise $\curl$, transpose, and take column-wise $\curl$ again. Equivalently, $\inc u=\nabla\times u\times \nabla$, which denotes a column-wise $\curl$ of $u$ composed with a row-wise $\curl$ (the order is irrelavent). Finally, $\div$ is a column-wise divergence in our convention.  In the index notation, we have $(\deff u)_{ij}=1/2(\partial_{i}u_{j}+\partial_{j}u_{i})$, $(\inc g)^{ij}=\epsilon^{ikl}\epsilon^{jst}\partial_{k}\partial_{s}g_{lt}$, and $(\div v)_{i}=\partial^{j}u_{ij}$.  Once can readily check that \eqref{elasticity} is a complex, i.e., $\inc\circ\deff=0$ and $\div\circ \inc=0$. 

If $g$ is the metric of a Riemannian manifold, then  $\inc g:= \nabla\times g\times \nabla$ is the linearised Einstein tensor of $g$ around the Euclidean metric $g_{0}$ (up to a constant multiple). More precisely, the Einstein tensor of the perturbed metric $\tilde{g}:=g_{0}+tg$ is $t\inc g+o(t)$. In three space dimensions, the Einstein tensor carries the same information as the Ricci and Riemann curvature tensors.

 
  We will show that the cohomology of the elasticity complex \eqref{elasticity} is isomorphic to six copies of the de~Rham cohomology. Therefore \eqref{elasticity} is exact when the domain is contractible, except at index 0, where the cohomology is isomorphic to the space of infinitesimal rigid body motions   
\begin{equation}\label{def:RM}
\mathcal{RM}:=\{\vec a+\vec b\times \vec x: \vec a, \vec b\in \mathbb{R}^3\}=\ker (\deff),
\end{equation}
which is a six-dimensional space.

 The next example is the {\it conformal deformation complex}:
\begin{equation}\label{conformal-complex-smooth}
\begin{tikzcd} 
0\arrow{r}& C^{\infty}\otimes \mathbb{R}^3 \arrow{r}{\dev\deff}&C^{\infty} \otimes (\mathbb{S}\cap \mathbb{T}) \arrow{r}{\cot}&C^{\infty} \otimes (\mathbb{S}\cap \mathbb{T}) \arrow{r}{\div}&C^{\infty} \otimes \mathbb{R}^3 \arrow{r}& 0. 
\end{tikzcd}
\end{equation}
Here $\mathbb{T}:=\mathbb{R}^{3\times 3}_{\mathrm{traceless}}$ is the linear space of traceless matrices, 
 $\dev \ten w:=\ten w-\frac{1}{n}\tr (\ten w)\ten I$ is the deviator and therefore $\dev\deff\vec u$ is symmetric and traceless. Moreover,  $\cot:=\curl\circ \mathcal{S}^{-1}\circ \curl\circ \mathcal{S}^{-1}\circ \curl$, where $\mathcal{S}^{-1}\ten\sigma:=\ten\sigma^{T}-\tr(\ten\sigma)\ten I$, is the linearised Cotton-York operator (tensor), which plays the role of curvature in conformal geometry. In gravitational-wave models, the variables usually satisfy the transverse-traceless (TT) gauge conditions, i.e., the variable is symmetric, traceless and divergence-free. This gauge condition is precisely encoded in \eqref{conformal-complex-smooth}. 
 
 The three complexes \eqref{deRham:3D}, \eqref{elasticity}, and \eqref{conformal-complex-smooth} can be compared from several perspectives. Firstly, these complexes involve vectors, symmetric matrices and symmetric traceless matrices, respectively. Secondly, the second operators in the three complexes are of the first, second and third order, respectively, which can all be constructed as a composition of $\curl$; these three complexes are all formally self-adjoint (see Section \ref{sec:analysis} below).   
  Thirdly, these three complexes encode topology, Riemannian geometry and conformal geometry, respectively.

\subsection{De Rham complexes: algebraic topology}

The cohomology of the de Rham complex is closely related to the topology of a domain. In particular, three concepts are linked with each other: homology (topology of a domain), cohomology (a dual concept of homology), and the de~Rham cohomology. Showing the detail of this connection in this section will not only allow us to understand the cohomology of \eqref{deRham:3D}, but will also guide the design of discretisation. 

Topology studies whether one object can be deformed to another by {\it continuous} maps. For example,  Figure \ref{fig:different-topology} depicts two domains. One cannot deform to another continuously, as there is a hole in one domain but not in another. In topology, one develops invariants to classify such objects: if a quantity is proved to be the same among all objects that can be continuously deformed to each other, then it is called a topological invariant. If two objects have different values, then they are not topologically equivalent. However, note that having the same values of the invariant does not guarantee two objects to be equivalent. 
 \begin{figure}
\begin{center}
\begin{tikzpicture}
\draw[fill=gray] (0,0) circle (1cm);
\draw[->, dashed] (0,0) circle (0.7cm);
\draw[fill=gray] (3,0) circle (1cm);
\fill[white] (3,0) circle (0.5cm);
\draw[->, dashed] (3,0) circle (0.7cm);
\end{tikzpicture}
\end{center}
\caption{Two manifolds: a disk and an annulus. One cannot continuously deform to another due to the hole in the annulus. From a homology perspective, a loop around the hole is not the boundary of any patch; while any closed loop in the disk is a boundary. }
\label{fig:different-topology}
\end{figure}
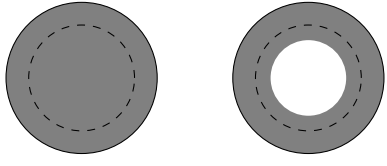

{\it Homology} is such a topological invariant. The idea of homology theory is to find out circles that are not the boundary of a disk, and generalisations to objects of any dimension (such as spheres that are not the boundary of a ball). This can be made precise by using simplicial complexes or the triangulation of a domain.  

\begin{definition}[Simplex]
Given $k+1$ points $x_{0}, x_{1}, \dots, x_{k}$ in $\mathbb{R}^{n}$ in general position (no $m+2$ points lie on a common $m$-dimensional plane), the convex hull 
\[
\sigma := [x_{0}, x_{1}, \dots, x_{k}]
\]
is called a $k$-simplex. 
\end{definition}

For example, a point is a 0-simplex ($k=0$); a line segment is a 1-simplex ($k=1$); a triangle is a 2-simplex ($k=2$); and a tetrahedron is a 3-simplex ($k=3$).

\begin{definition}[Simplicial complex]
A simplicial complex $\Sigma$ is a collection of simplices such that
\begin{enumerate}
\item if a simplex $\sigma \subset \Sigma$, then all of its faces (subsimplices) also belong to $\Sigma$; 
\item if two simplices $\sigma_{1}, \sigma_{2} \subset \Sigma$, then their intersection is either empty or a common face. 
\end{enumerate}
\end{definition}

An example of a simplicial complex is a triangulation, where all vertices, edges, faces, etc., are included. 

\begin{definition}[Oriented simplex]
An oriented $k$-simplex is a simplex together with an ordering of its vertices, denoted 
\[
[v_{0}, v_{1}, \dots, v_{k}].
\] 
Two orderings determine the same orientation if they differ by an even permutation; if they differ by an odd permutation, the orientations are opposite:
\[
[v_{\pi(0)}, v_{\pi(1)}, \dots, v_{\pi(k)}] = \operatorname{sgn}(\pi)\,[v_{0}, v_{1}, \dots, v_{k}].
\]
\end{definition}

We now define the associated chain complex
\begin{equation}\label{simplicial-complex}
\begin{tikzcd}
\cdots \arrow{r}{} & C_{i+1} \arrow{r}{\partial_{i+1}} & C_{i} \arrow{r}{\partial_{i}} & C_{i-1} \arrow{r}{} & \cdots,
\end{tikzcd}
\end{equation}
where each $C_{k}$ is the group of $k$-chains (formal linear combinations of oriented $k$-simplices), and $\partial_{k}$ is the boundary operator. The intuition that \eqref{simplicial-complex} is a complex comes from the fact that the boundary of a boundary vanishes (for instance, the boundary of a disk is a circle, and a circle has no boundary). We now make this precise.

\begin{definition}[$k$-chain]
Given a simplicial complex $\Sigma$, a $k$-chain is a finite formal linear combination of oriented $k$-simplices:
\[
\sum_{i} c_{i} \sigma_{i}, 
\qquad c_{i} \in \mathbb{R},
\]
where $\sigma_{i}:=[v_{0}^{i}, v_{1}^{i}, \dots, v_{k}^{i}].$
\end{definition}

\begin{definition}[Chain groups]
The set of all $k$-chains forms an Abelian group under addition, denoted by $C_{k}$, with
\[
\Big(\sum_{j} \alpha_{j}\sigma_{j}\Big) + \Big(\sum_{j} \beta_{j}\sigma_{j}\Big) 
= \sum_{j} (\alpha_{j}+\beta_{j})\sigma_{j}.
\]
\end{definition}

\begin{definition}[Boundary operator $\partial_{k}$]
The boundary operator is a homomorphism $\partial_{k}: C_{k}(\Sigma) \to C_{k-1}(\Sigma)$ defined on an oriented simplex by
\[
\partial_{k}[v_{0}, v_{1}, \dots, v_{k}]
= \sum_{i=0}^{k} (-1)^{i}[v_{0}, \dots, \widehat{v_{i}}, \dots, v_{k}],
\]
where the hat denotes omission of the vertex $v_{i}$. It extends linearly to all $k$-chains:
\[
\partial_{k}\!\left(\sum_{i} \alpha_{i}\sigma_{i}\right) 
= \sum_{i} \alpha_{i}\,\partial_{k}\sigma_{i}, 
\qquad \alpha_{i} \in \mathbb{R}.
\]
\end{definition}

The following result can be shown by a straightforward computation.
\begin{theorem}
Boundary of boundary vanishes, i.e., 
$\partial_{k-1}\circ \partial_{k}=0$.
\end{theorem}
For example, consider
$
T=[x_{0}, x_{1}, x_{2}]. 
$
Then 
$$
\partial T = [x_{1}, x_{2}]-[x_{0}, x_{2}]+ [x_{0}, x_{1}],
$$ 
 and
$$
\partial\partial T = (x_{2}-x_{1})-(x_{2}-x_{0})+(x_{1}-x_{0})=0.
$$
\begin{center}
  \begin{tikzpicture}[scale=0.7]
    \coordinate (1) at (0,0);
    \coordinate (2) at (4,0);
    \coordinate (3) at (2,3);

    \fill[gray] (1) -- (2) -- (3) -- cycle;

    \draw (1) -- (2) -- (3) -- cycle;

    \node[below left] at (1) {$x_0$};
    \node[below right] at (2) {$x_1$};
    \node[above] at (3) {$x_2$};
    
        \coordinate (A) at (6,0);
    \coordinate (B) at (10,0);
    \coordinate (C) at (8, 3);


    \draw[->] (A) -> (B) ;
      \draw[->] (B) -> (C) ;
      \draw[->] (C) -> (A) ;
    \node[below left] at (A) {$x_0$};
    \node[below right] at (B) {$x_1$};
    \node[above] at (C) {$x_2$};
    
\end{tikzpicture}

\end{center}
 
The simplicial homology is defined to be the quotient space $H_{i}:=\ker(\partial_{i})/\ran(\partial_{i+1})$.
The $k$-th {Betti number is defined to be the dimension of the $k$-th homology space $ b_{k}:=\dim H_{k}$, i.e.,  the dimension of $k$-dimensional closed loops that are not boundaries. 
Here are some examples of homology:
\begin{itemize}
\item disk  $D^{2}\subset \mathbb R^{2}$: $b_{0}=1$ (boundary of any point vanishes),  $b_{1}=0$, $b_{2}=0$. 
\item sphere $S^{2}$: $b_{0}=1$,  $b_{1}=0$, $b_{2}=1$ (boundary of a sphere vanishes). 
\end{itemize}

\begin{tcolorbox}[
    enhanced, breakable,
    toggle enlargement=evenpage,
    colback=red!5!white,colframe=red!75!black,
    title={A crash course on differential forms}]
{\noindent\bf Cartan's Formalism.}
Differential forms arise naturally when trying to generalise the concept of oriented area (and higher-dimensional volumes). For instance, consider the integral
\[
\int_{\Omega} f(x, y)\, dx\,dy.
\]
From the exterior calculus perspective, the infinitesimal area element \(dx\,dy\) can be interpreted as the wedge product \(dx \wedge dy\), which satisfies \(dx \wedge dy = -\,dy \wedge dx\). This skew-symmetry carries the orientation of the area element.

{\noindent\bf Exterior Calculus Perspective.}
A \emph{\(k\)-form} in \(n\)-dimensional space is a linear combination of terms of the form 
\[
dx^{j_1} \wedge dx^{j_2} \wedge \cdots \wedge dx^{j_k}, \quad j_1, \dots, j_k \le n,
\]
where each term is fully skew-symmetric:
\[
dx^{j_1}\wedge \cdots \wedge dx^{j_i} \wedge \cdots \wedge dx^{j_\ell} \wedge \cdots \wedge dx^{j_k}
= -\,dx^{j_1}\wedge \cdots \wedge dx^{j_\ell} \wedge \cdots \wedge dx^{j_i}\wedge \cdots \wedge dx^{j_k}.
\]
In particular, \(dx^i \wedge dx^i = 0\), and 
\[
dx^i \wedge dx^j = -\,dx^j \wedge dx^i.
\]
Because of this skew-symmetry, any \(k\)-form with \(k>n\) must vanish in \(\mathbb{R}^n\). 

For example, in \(\mathbb{R}^{3}\), the following sets form bases for the space of \(k\)-forms \(\Alt^{k}\mathbb{R}^{3}\):

\begin{center}
\begin{tabular}{|c|c|}
\hline
 \(k\) & Basis for \(\Alt^{k}\mathbb{R}^{3}\)\\\hline
 0 & \(1\)\\
 1 & \(dx^{1}, \; dx^{2}, \; dx^{3}\)\\
 2 & \(dx^{1}\wedge dx^{2}, \; dx^{2}\wedge dx^{3}, \; dx^{3}\wedge dx^{1}\)\\
 3 & \(dx^{1}\wedge dx^{2}\wedge dx^{3}\)\\
\hline
\end{tabular}
\end{center}

A differential $k$-form is an alternating $k$-form with function coefficients, i.e., any differential $k$-form $\omega$ can be written as a linear combination
$$
\omega=\sum_{j_{1}, \cdots, j_{k}}\sigma_{j_{1}, \cdots, j_{k}}(x)dx^{j_{1}}\wedge \cdots \wedge dx^{j_{k}}.
$$

 {\noindent\bf Proxies in 3D.}
A \(0\)-form is simply a scalar function. A \(1\)-form \(\omega\) of the form 
\[
\omega = w_{1} \, dx^{1} + w_{2} \, dx^{2} + w_{3} \, dx^{3}
\]
corresponds to the vector \((w_{1}, w_{2}, w_{3})\) in \(\mathbb{R}^3\).  
A \(2\)-form \(v\) given by
\[
v = v_{1}\,dx^{2}\wedge dx^{3} + v_{2}\,dx^{3}\wedge dx^{1} + v_{3}\,dx^{1}\wedge dx^{2}
\]
has the proxy vector \((v_{1}, v_{2}, v_{3})\). 
Finally, any \(3\)-form in \(\mathbb{R}^{3}\) must be a multiple of \(dx^{1}\wedge dx^{2}\wedge dx^{3}\), which corresponds to a single scalar.

We summarise this correspondence in the following table, where we slightly abuse notation by using the same symbol (\(\omega\) or \(v\)) to denote both the differential form and its vector (or scalar) proxy:

\begin{center}
\begin{tabular}{|c|c|c|}
\hline
 \(k\) & \quad Differential form \quad & \quad Proxy \quad\\\hline
 0 & \(f\) & \(f\) \\
 1 & \(w = w_{1}\,dx^{1} + w_{2}\,dx^{2} + w_{3}\,dx^{3}\) & \((w_{1},\, w_{2},\, w_{3})\) \\
 2 & \(v = v_{1}\,dx^{2}\wedge dx^{3} + v_{2}\,dx^{3}\wedge dx^{1} + v_{3}\,dx^{1}\wedge dx^{2}\) & \((v_{1},\, v_{2},\, v_{3})\) \\
 3 & \(u \, dx^{1}\wedge dx^{2}\wedge dx^{3}\) & \(u\) \\
\hline
\end{tabular}
\end{center}

{\noindent\bf Wedge Products in 3D.} The {\it wedge product} of a $k$-form and an $\ell$-form is a $(k+\ell)$-form. For 1-form bases, define $(f\,dx^{i})\wedge (g\,dx^{j})=fg\, dx^{i}\wedge dx^{j}$. Similar rules apply to general cases. 
Let \(u = u_{1}\,dx^{1} + u_{2}\,dx^{2} + u_{3}\,dx^{3}\) and \(v = v_{1}\,dx^{1} + v_{2}\,dx^{2} + v_{3}\,dx^{3}\) be \(1\)-forms in \(\Alt^{1}\mathbb{R}^3\). Also let 
\begin{equation}\label{eqn:1formwedge1form}
w = w_{1}\,dx^{2}\wedge dx^{3} + w_{2}\,dx^{3}\wedge dx^{1} + w_{3}\,dx^{1}\wedge dx^{2}
\end{equation}
be a \(2\)-form in \(\Alt^{2}\mathbb{R}^3\). Then:
\[
u \wedge v 
= 
(u_{1}v_{2} - u_{2}v_{1})\,dx^{1}\wedge dx^{2} 
+ 
(u_{2}v_{3} - u_{3}v_{2})\,dx^{2}\wedge dx^{3} 
+ 
(u_{3}v_{1} - u_{1}v_{3})\,dx^{3}\wedge dx^{1},
\]
and
\[
u \wedge w 
= 
(u_{1}w_{1} + u_{2}w_{2} + u_{3}w_{3}) \,dx^{1}\wedge dx^{2}\wedge dx^{3}.
\]
In \eqref{eqn:1formwedge1form}, the term $(u_{1}v_{2} - u_{2}v_{1})\,dx^{1}\wedge dx^{2} $ is because we have two terms $u_{1}dx^{1}\wedge v_{2}dx^{2}$ and $u_{2}dx^{2}\wedge v_{1}dx^{1}=-u_{2}v_{1}dx^{1}\wedge dx^{2}$. The computation leading to other terms is similar.
In vector proxy form, these translate to cross and dot products, respectively:
\[
\begin{aligned}
u \wedge v \quad &\longleftrightarrow \quad \mathbf{u} \times \mathbf{v},\\
u \wedge w \quad &\longleftrightarrow \quad \mathbf{u} \cdot \mathbf{w},
\end{aligned}
\]
where \(\mathbf{u}, \mathbf{v}, \mathbf{w}\) denote the respective vector proxies.

More generally, combining different degrees of forms yields the following wedge-product behavior (abusing notation as before):

\begin{center}
\begin{tabular}{|c|c|c|c|c|}
\hline
\multicolumn{5}{|c|}{$\omega \wedge \mu$}\\
\hline
 & $\ell=0$ & $\ell=1$ & $\ell=2$ & $\ell=3$\\
\hline
 $k=0$ & $\omega\mu$ & $\omega \boldsymbol{\mu}$ & $\omega \boldsymbol{\mu}$ & $\omega \mu$\\
 $k=1$ & $\boldsymbol{\omega}\mu$ & $\boldsymbol{\omega}\times \boldsymbol{\mu}$ & $\boldsymbol{\omega}\cdot \boldsymbol{\mu}$ & 0\\
 $k=2$ & $\boldsymbol{\omega}\mu$ & $\boldsymbol{\omega}\cdot \boldsymbol{\mu}$ & 0 & 0\\
 $k=3$ & $\omega\mu$ & 0 & 0 & 0\\
\hline
\end{tabular}
\end{center}

{\noindent\bf Exterior Derivatives.}
For differential forms, we define the exterior derivatives \(d^{k}: \Lambda^{k}(\Omega)\to \Lambda^{k+1}(\Omega)\). In coordinates, if
\[
\omega = \sum_{1\leq \sigma_{1}\leq \cdots\leq \sigma_{k}\leq n} f_{\sigma}\,dx^{\sigma_{1}}\wedge \cdots \wedge dx^{\sigma_{k}},
\]
then
\[
d^{k}\omega 
= 
\sum_{1\leq \sigma_{1}\leq \cdots\leq \sigma_{k}\leq n}
\frac{\partial f_{\sigma}}{\partial x^{j}}\,dx^{j}\wedge dx^{\sigma_{1}}\wedge \cdots \wedge dx^{\sigma_{k}}.
\]
In \(\mathbb{R}^3\), this definition matches up with the standard vector calculus operators when using the proxy identifications:

\begin{itemize}
\item For a \(0\)-form \(u\),
\[
d^{0}u = \frac{\partial u}{\partial x^{1}}\,dx^{1} + \frac{\partial u}{\partial x^{2}}\,dx^{2} + \frac{\partial u}{\partial x^{3}}\,dx^{3},
\]
which corresponds to the gradient \(\nabla u\) in vector calculus.

\item For a \(1\)-form 
\[
v = v_{1}\,dx^{1} + v_{2}\,dx^{2} + v_{3}\,dx^{3},
\]
\[
d^{1}v 
=
\frac{\partial v_{1}}{\partial x^{2}}\,dx^{2}\wedge dx^{1}
+
\frac{\partial v_{1}}{\partial x^{3}}\,dx^{3}\wedge dx^{1}
+
\frac{\partial v_{2}}{\partial x^{1}}\,dx^{1}\wedge dx^{2}
+
\frac{\partial v_{2}}{\partial x^{3}}\,dx^{3}\wedge dx^{2}
+
\frac{\partial v_{3}}{\partial x^{1}}\,dx^{1}\wedge dx^{3}
+
\frac{\partial v_{3}}{\partial x^{2}}\,dx^{2}\wedge dx^{3}.
\]
Using \(dx^{i}\wedge dx^{j}=-\,dx^{j}\wedge dx^{i}\) (and \(dx^{i}\wedge dx^{i}=0\)), we can rewrite this as
\[
d^{1}v 
=
\Bigl(\frac{\partial v_{2}}{\partial x^{1}} - \frac{\partial v_{1}}{\partial x^{2}}\Bigr)\,dx^{1}\wedge dx^{2}
+
\Bigl(\frac{\partial v_{3}}{\partial x^{2}} - \frac{\partial v_{2}}{\partial x^{3}}\Bigr)\,dx^{2}\wedge dx^{3}
+
\Bigl(\frac{\partial v_{1}}{\partial x^{3}} - \frac{\partial v_{3}}{\partial x^{1}}\Bigr)\,dx^{3}\wedge dx^{1},
\]
which corresponds to \(\nabla \times \mathbf{v}\) in the vector proxy.

\item It remains as an exercise to show that for a \(2\)-form in \(\mathbb{R}^3\), the exterior derivative \(d^{2}\) corresponds to the divergence operator \(\nabla \cdot\).
\end{itemize}

\end{tcolorbox}

Having introduced chains, boundaries, and homology, we now move to the dual picture. 
Instead of working directly with simplices and their boundaries, we can assign to each \(k\)-simplex a real value and track how these “dual values” change under a dual operator called the \emph{coboundary}. This leads to \emph{cochains} and \emph{cohomology}, which capture the same topological information as homology but often prove more convenient for certain theoretical and computational tasks.

\begin{definition}[$k$-cochain]
Let \(\Sigma\) be a simplicial complex. A \emph{\(k\)-cochain} is a function that assigns a real value to each \(k\)-simplex of \(\Sigma\). Formally,
\[
\varphi: \{\text{\(k\)-simplices in }\Sigma\} \,\to\, \mathbb{R}.
\]
The set of all such \(k\)-cochains is denoted by \(C^k(\Sigma)\). 
\end{definition}

\begin{definition}[coboundary operator]
The coboundary operator 
\[
\delta_k: C^k(\Sigma)\,\to\, C^{k+1}(\Sigma)
\]
is defined so that for each \( (k+1)\)-simplex \(\sigma\), 
\[
(\delta_k \varphi)(\sigma) = \varphi\bigl(\partial_{k+1}\sigma\bigr),
\]
where \(\partial_{k+1}\sigma\) is the boundary of \(\sigma\) considered as a \(k\)-chain. Concretely, if 
\(\sigma = [v_0, v_1, \dots, v_{k}, v_{k+1}]\) then
\[
(\delta_k \varphi)(\sigma)
=
\sum_{i=0}^{k+1} (-1)^i \,\varphi\bigl([v_0,\dots,\widehat{v_i},\dots,v_{k+1}]\bigr),
\]
matching   the definition of the boundary \(\partial_{k+1}\) reflected in the simplex that  \(\varphi\) acts on.
\end{definition}

\noindent
As a direct consequence of the properties of the boundary operator of chains, we have
\[
\delta_{k+1}\,\delta_{k} = 0, \quad\forall k.
\]
Hence we get a cochain complex
\[
\begin{tikzcd}
\cdots \ar[r] & C^k(\Sigma) \ar[r, "\delta_k"] & C^{k+1}(\Sigma) \ar[r, "\delta_{k+1}"] & C^{k+2}(\Sigma) \ar[r] & \cdots.
\end{tikzcd}
\]

\begin{definition}[cohomology groups]
The  {\(k\)-th cohomology group} is defined by
\[
H^k(\Sigma) = {\ker(\delta_k)}/{\ran(\delta_{k-1})}.
\]
Namely, \(k\)-cohomology classes capture those cochains whose coboundary vanishes (i.e.\ \(\ker \delta_k\)) but which themselves are not a coboundary of some \((k-1)\)-cochain.
\end{definition}

The homology, cohomology (of cochains), and the de~Rham cohomology are related to each other. The homology and cohomology are related through the Poincar\'e duality $H^k(\Omega) \cong H_{n-k}(\Omega)$ for compact manifolds without boundary (and the Poincar\'e--Lefschetz duality $H^k(\Omega, \partial \Omega) \cong H_{n-k}(\Omega)$ and $H^k(\Omega) \cong H_{n-k}(\Omega, \partial \Omega)$ for manifolds with boundary). Here $H^k(\Omega, \partial \Omega)$ is the so-called relativity cohomology, for which all the cochains lying on the boundary $\partial \Omega$ are treated as zero. Finally, cohomology and the de~Rham cohomology are related through the de~Rham map. More precisely, a differential $k$-form $\omega$ induces a $k$-cochain: given a $k$-simplex $\sigma$, the cochain maps it to the integral $\int_{\sigma}\omega$. The \emph{de Rham theorem} states that for a smooth manifold $\Omega$, 
$
H^k_{\mathrm{dR}}(\Omega; \mathbb{R})
\cong
H^k(\Omega; \mathbb{R})
$
and 
$
H^k_{\mathrm{dR}}(\Omega, \partial \Omega; \mathbb{R})
\cong
H^k(\Omega, \partial \Omega; \mathbb{R}).
$
Through the above results, we can directly relate de~Rham cohomology to the topology of a domain. Particularly, we consider two versions of de~Rham complexes
\begin{equation}\label{deRham-Cinfty}
\begin{tikzcd}
0 \ar[r] & C^{\infty}\Lambda^0 \ar[r, "d"] &C^{\infty}\Lambda^1 \ar[r, "d"] & \cdots \ar[r, "d"] & C^{\infty}\Lambda^n \ar[r] & 0,
\end{tikzcd}
\end{equation}
\begin{equation}\label{deRham-Cinfty0}
\begin{tikzcd}
0 \ar[r] & C_{0}^{\infty}\Lambda^0 \ar[r, "d"] &C_{0}^{\infty}\Lambda^1 \ar[r, "d"] & \cdots \ar[r, "d"] & C_{0}^{\infty}\Lambda^n \ar[r] & 0,
\end{tikzcd}
\end{equation}
where $C_{0}^{\infty}\Lambda^k$ denotes the space of compactly supported $k$-forms (in coordinates, the coefficients are in $C_{0}^{\infty}$).
The cohomology of \eqref{deRham-Cinfty} is isomorphic to $ H_{n-k}(\Omega, \partial \Omega)$; the cohomology of \eqref{deRham-Cinfty0} is isomorphic to $H_{n-k}(\Omega)$. This result can be seen in some special cases. Let $\Omega$ be a contractible domain (e.g., a ball). Then for \eqref{deRham-Cinfty}, the 0-th cohomology is isomorphic to $ H_{n}(\Omega, \partial \Omega)\cong \mathbb{R}$; and the $n$-th cohomology of \eqref{deRham-Cinfty0} is  $ H_{0}(\Omega)\cong \mathbb{R}$. The former reflects that the kernel of $\grad$ consists of constants, and the latter reflects the fact that the divergence of a compactly supported field has vanishing integral (thus $\div: [C_{0}^{\infty}(\Omega)]^{n}\to C_{0}^{\infty}$ has a one-dimensional cokernel).



 \subsection{Complexes from complexes}
 
 More complexes can be derived from the de~Rham complex. The simplest case would be deriving a complex with a second-order operator from two de~Rham complexes in 1D \eqref{deRham1D}. We connect the tail of the one complex to the head of another one:
    \begin{equation} \label{BGGdiagram-1D}
\begin{tikzcd}
0 \arrow{r} &C^{\infty} (I) \arrow{r}{\frac{d}{dx}} &C^{\infty} (I)\arrow{r}{} & 0\\
0 \arrow{r} &C^{\infty}(I)\arrow{r}{\frac{d}{dx}} \arrow[ur, "I"]&C^{\infty}(I) \arrow{r}{} & 0.
 \end{tikzcd}
\end{equation}
This derives
   \begin{equation}\label{cplx:1DBGG} 
\begin{tikzcd}
0 \arrow{r} &C^{\infty}(I)  \arrow{r}{\frac{d^{2}}{dx^{2}}} &C^{\infty}(I)\arrow{r}{} & 0.
 \end{tikzcd}
\end{equation}
We can make some observations even in this simple example. We observe that the cohomology of the derived complex \eqref{cplx:1DBGG} is isomorphic to the input -- the sum of the cohomology of the two rows in \eqref{BGGdiagram-1D}. First, $\ker(\frac{d^{2}}{dx^{2}})\cong \mathbb{R}^{2}=\mathbb{R}\oplus \mathbb{R}\cong\ker(\frac{d}{dx})\oplus \ker(\frac{d}{dx})$, which is the cohomology at the first $C^{\infty}(I)$ space in \eqref{cplx:1DBGG}. Second, $\frac{d}{dx}: C^{\infty}(I)\to C^{\infty}(I)$ in both rows of \eqref{BGGdiagram-1D} are surjective. This implies that $\frac{d^{2}}{dx^{2}}$ is also surjective. The second observation is that to generalise \eqref{BGGdiagram-1D} to more general function spaces, such as Sobolev spaces, restrictions will be added to such spaces. For example, a natural generalisation of \eqref{BGGdiagram-1D} to Sobolev spaces is 
    \begin{equation} \label{BGGdiagram-sobolev}
\begin{tikzcd}
0 \arrow{r} &H^{2} (I) \arrow{r}{\frac{d}{dx}} &H^{1} (I)\arrow{r}{} & 0\\
0 \arrow{r} &H^{1}(I)\arrow{r}{\frac{d}{dx}} \arrow[ur, "I"]&L^{2}(I) \arrow{r}{} & 0.
 \end{tikzcd}
\end{equation}
The second row has lower regularity than the first as the first space of the second row is connected to the last space in the first row, and the regularity of spaces decrease in a de~Rham complex.
 
A nontrivial example is the derivation of the elasticity complex \eqref{elasticity}. 

The first step is to
 {\it connect complexes}. We fit two vector-valued de~Rham complexes in a commuting diagram:
{
  \begin{equation} \label{elasticity-diagram}
\begin{tikzcd}
0 \arrow{r} &C^{\infty}\otimes \mathbb{R}^{3} \arrow{r}{\grad} &C^{\infty}\otimes\mathbb{R}^{3\times 3} \arrow{r}{\curl} &C^{\infty}\otimes\mathbb{R}^{3\times 3} \arrow{r}{\div} &C^{\infty}\otimes\mathbb{R}^{3}  \arrow{r}{} & 0\\
0 \arrow{r} &C^{\infty}\otimes\mathbb{R}^{3}\arrow{r}{\grad} \arrow[ur, "-\mskw"]&C^{\infty}\otimes\mathbb{R}^{3\times 3} \arrow{r}{\curl} \arrow[ur, "\mathcal{S}"]&C^{\infty}\otimes\mathbb{R}^{3\times 3} \arrow{r}{\div}\arrow[ur, "2\vskw"] &C^{\infty}\otimes \mathbb{R}^{3}\arrow{r}{} & 0
 \end{tikzcd}
\end{equation}
}
Here a vector-valued de~Rham complex in 3D means that we start with a vector-valued function space $C^{\infty}\otimes \mathbb{R}^{3}$ (a vector with three components, each component being a $C^{\infty}$ function). Then taking column-wise $\grad$, $\curl$ and $\div$, we obtain a complex with vectors, matrices, matrices, and vectors. The operators connecting the two rows are algebraic: $\mskw$ maps a vector to a skew-symmetric matrix in 3D and $\vskw$ is the inverse which takes the skew-symmetric part of a matrix and identify it with a vector:
$$
\mskw\left ( 
\begin{array}{c}
a\\b\\c
\end{array}
\right ):=\left ( 
\begin{array}{ccc}
0 & -c & b\\
c & 0 & -a\\
-b & a&0
\end{array}
\right ), \quad \vskw \sigma:=\mskw^{-1}\circ \skw \sigma.
$$
The operator in the middle is 
$\mathcal{S}\ten u:=\ten u^{T}-\tr(\ten u)\ten I$. Note that $\mskw$, $\mathcal{S}$ and $\vskw$ are surjective, bijective, and injective, respectively.

The second step is to {\it eliminate as much as possible in the diagram}. This means that we eliminate spaces or components connected by the connecting maps. In the above example, we decompose each matrix into its symmetric and skew-symmetric parts:
{
  \begin{equation*} 
\begin{tikzcd}
0 \arrow{r} &\mathbb{R}^{3} \arrow{r}{\grad} &\mathbb{S}+{\bf\color{blue}\mathbb{K}} \arrow{r}{\curl} &{\bf\color{red}\mathbb{R}^{3\times 3}  }\arrow{r}{\div} &{\bf\color{brown} \mathbb{R}^{3} } \arrow{r}{} & 0\\
0 \arrow{r} &{\bf\color{blue}\mathbb{R}^{3}}\arrow{r}{\grad} \arrow[ur, "-\mskw"]&{\bf\color{red}\mathbb{R}^{3\times 3}} \arrow{r}{\curl} \arrow[ur, "\mathcal{S}"]&\mathbb{S}+{\bf\color{brown}\mathbb{K} }\arrow{r}{\div}\arrow[ur, "2\vskw"] & \mathbb{R}^{3}\arrow{r}{} & 0.
 \end{tikzcd}
\end{equation*}
Here $\mathbb{S}$ denotes symmetric matrices and $\mathbb{K}$ denotes skew-symmetric matrices. 
To emphasise the algebraic structures, we have omitted $C^{\infty}\otimes$ in the notation, and thus $\mathbb{R}^{3}$ means $C^{\infty} \otimes \mathbb{R}^{3}$ etc. The second (nonzero) space on the top row is divided into $\mathbb{S}$ and $\mathbb{K}$. The $\mathbb{K}$ part matches $\mathbb{V}$ from the bottom row by the incoming $\mskw$. Then we eliminate $\mathbb{V}$ and $\mathbb{K}$ from the diagram. Similarly, the second last space in the bottom row splits into $\mathbb{S}$ and $\mathbb{K}$. The $\mathbb{K}$ part matches   $\mathbb{V}$ at the end of the top row. They are both eliminated from the diagram. 
In the middle, $\mathcal{S}$ is surjective. Therefore the two spaces connected by it are completed removed from the diagram. 
What we are left with is the following:
  \begin{equation*} 
\begin{tikzcd}
0  & \mathbb{R}^{3}  & \mathbb{S}  & &~  &~ \\
 ~&~&~ & \mathbb{S}   &  \mathbb{R}^{3}   & 0.
 \end{tikzcd}
\end{equation*}
}
Then the final step of reading out the derived complex is to {\it connect rows by a zig-zag}:
{
  \begin{equation*} 
\begin{tikzcd}[sep=large]
0 \arrow{r}& \mathbb{R}^{3}\arrow{r}{\deff} & \mathbb{S}  \arrow{r}{\curl}&~\arrow[dl, "\mathcal{S}^{-1}"] &~  &~ \\
 ~&~&~\arrow{r}{\curl}& \mathbb{S} \arrow{r}{\div} &  \mathbb{R}^{3} \arrow{r}{} & 0.
 \end{tikzcd}
\end{equation*}
}
Recall again that here $\mathbb{R}^{3}$ means $C^{\infty}\otimes \mathbb{R}^{3}$ etc. for short.

A major conclusion is that, under some conditions (which hold for all the examples in this paper), {\it the above algebraic process of eliminating and connecting is cohomology-preserving}: the cohomology of the output complex is isomorphic to the input. In the 1D example \eqref{cplx:1DBGG}, the cohomology is two copies of the de~Rham cohomology; in the 3D elasticity example \eqref{elasticity}, the cohomology is six copies of the de~Rham cohomology (each row in the diagram \eqref{elasticity-diagram} consists of three copies of the de~Rham complex, and therefore in total six copies). In particular, this means that the dimension of the cohomology of \eqref{elasticity} is six times the corresponding Betti number. The kernel of $\deff$ has dimension six (the explicit form \eqref{def:RM} gives a straightforward verification).

The 3D elasticity complex \eqref{elasticity} is an example of a family of (infinitely many) complexes. The diagram \eqref{elasticity-diagram} can be generalised with the following diagram, where $\Alt^{k, \ell}:=\Alt^{k}\otimes \Alt^{\ell}$ is the space of {\it alternating $\ell$-form-valued $k$-forms} (skew-symmetric multilinear maps which take in $k$ vectors and return an alternating $\ell$-form), and again, we omitted $C^{\infty}\otimes$ in each space:
  \begin{equation} \label{form-diagram}
\begin{tikzcd} 
0 \arrow{r} &   \mathrm{Alt}^{0, J-1}  \arrow{r}{d} & \mathrm{Alt}^{1, J-1}  \arrow{r}{d} & \cdots \arrow{r}{d} &   \mathrm{Alt}^{n, J-1} \arrow{r}{} & 0\\
0 \arrow{r} &  \mathrm{Alt}^{0, J}  \arrow{r}{d} \arrow[ur, "S^{0,J}"] & \mathrm{Alt}^{1, J}  \arrow{r}{d}  \arrow[ur, "S^{1,J}"] & \cdots \arrow{r}{d}  \arrow[ur, "S^{n-1,J}"] &   \mathrm{Alt}^{n, J} \arrow{r}{} & 0.
 \end{tikzcd}
\end{equation}
The connecting maps are naturally given by an alternating sum:
\begin{multline*}
S^{i, J}\mu (v_{0},\cdots,  v_{i})(w_{1}, \cdots, w_{J-1}):=\sum_{l=0}^{i}(-1)^{l} \mu (v_{0},\cdots, \widehat{v_{l}}, \cdots,  v_{i})(v^{l}, w_{1}, \cdots, w_{J-1}), \\ \forall v_{0}, \cdots, v_{i}, w_{1}, \cdots, w_{J-1} \in \mathbb{R}^{n}.
\end{multline*}
These maps satisfy the anti-commutativity $d\circ S=-S\circ d$.

In $n$D, all the $k$-forms for $k>n$ vanish. Therefore the diagram has finite numbers of rows and columns. We can write down all the cases in 3D:
$$
\begin{tikzcd}
{ 0} \arrow{r} &\mathrm{Alt}^{0, 0 } \arrow{r}{{d^{0}}} &\mathrm{Alt}^{1, 0}  \arrow{r}{d^{1}} &\mathrm{Alt}^{2, 0} \arrow{r}{d^{2}} & \mathrm{Alt}^{3, 0} \arrow{r}{} & {0}\\{
0} \arrow{r}&\mathrm{Alt}^{0, 1}\arrow{r}{d^{0}} \arrow[ur, "{S^{0, 1}}"]&\mathrm{Alt}^{1, 1}\arrow{r}{d^{1}} \arrow[ur, "{S^{1, 1}}"]&\mathrm{Alt}^{2, 1}\arrow{r}{d^{2}}\arrow[ur, "S^{2, 1}"] &\mathrm{Alt}^{3, 1}\arrow{r}{} &{ 0}\\
{0} \arrow{r} &\mathrm{Alt}^{0, 2}\arrow{r}{d^{0}} \arrow[ur, "S^{0, 2}"]&\mathrm{Alt}^{1, 2} \arrow{r}{d^{1}} \arrow[ur, "{S^{1, 2}}"]&\mathrm{Alt}^{2, 2}\arrow{r}{{d^{2}}}\arrow[ur, "S^{2, 2}"] & \mathrm{Alt}^{3, 2} \arrow{r}{} &{0}\\
{0} \arrow{r} &\mathrm{Alt}^{0, 3}\arrow{r}{d^{0}} \arrow[ur, "S^{0, 3}"]&\mathrm{Alt}^{1, 3} \arrow{r}{d^{1}} \arrow[ur, "{S^{1, 3}}"]&\mathrm{Alt}^{2, 3}\arrow{r}{{d^{2}}}\arrow[ur, "S^{2, 3}"] & \mathrm{Alt}^{3, 3} \arrow{r}{} &{0}.
 \end{tikzcd}
$$
Here each space $\Alt^{k, \ell}$ has two indices $k$ and $\ell$. In vector/matrix proxies, we can think of the $k$-index as rows and $\ell$ as columns of a matrix.  Then we obtain proxies of  $\Alt^{k, \ell}$ as the tensor product between scalars and vectors. For example, for  $\Alt^{0, 1}$, $\Alt^{0}$ corresponds to a scalar and $\Alt^{1}$ corresponds to the vector, leading to $\Alt^{k, \ell}$ as a vector; $\Alt^{2}$ corresponds to a vector, and therefore  $\Alt^{2, 2}$ gives a matrix. Connecting the first two rows, we obtain a Hessian operator, leading to the Hessian complex \cite{arnold2021complexes}. The second and third rows lead to the elasticity complex \eqref{elasticity-diagram}, and the last two rows lead to the $\div\div$ complex \cite{arnold2021complexes}.

\medskip

\subsection{Twisted complex: Riemann-Cartan geometry}

It worth diving deeper into the derivation of the BGG complexes for at least two reasons (in addition to the insights it provides in understanding the cohomology of complexes). First, an intermediate step in the derivation, called the {\it twisted de~Rham complexes}, have important and a bit surprising applications in their own right. Second, the derivation leads to a general {\it machinery} which can be used in more general contexts: {\it we can carry over results for the de~Rham complex (electromagnetic fields etc.) to objects modelled by the BGG or the twisted complexes (strain and stress tensors in continuum mechanics, metric and curvature tensors in differential geometry etc.) by chasing the diagrams. }

A {BGG diagram} consists of complexes connected by algebraic operators $S^{\bs}$ in a (anti)commuting diagram ($dS=-Sd$):
 \begin{equation}\label{BGG-diagram}
\begin{tikzcd}[column sep=2.1cm, row sep=0.8cm]
{\cdots} \arrow{r} &V^{k-2}  \arrow{r}{d^{k-2}} &V^{k-1} \arrow{r}{d^{k-1}} &V^{k} \arrow{r}{d^{k}} & V^{k+1} \arrow{r}{} &\cdots\\
\cdots \arrow{r}& W^{k-2}\arrow{r}{d^{k-2}} \arrow[ur, "{S^{k-2}}"]&{ W^{k-1} } \arrow{r}{{d^{k-1}}} \arrow[ur, "{S^{k-1}}"]&{  W^{k}} \arrow{r}{{\div}}\arrow[ur, "{S^{k}}"] &{W^{k+1} }\arrow{r}{} &\cdots
 \end{tikzcd}
\end{equation}
In typical applications, there exists an index $J$ such that the first several $S^{k}$ operators are injective for $k\leq J$, and the last several are surjective for $k\geq J$, with a bijective $S^{J}$ in the middle. 

We can read out two kinds of complexes from the BGG diagram \eqref{BGG-diagram}:
the twisted complex
{ \begin{equation}\label{cplx:twisted}
  \begin{tikzcd}[ampersand replacement=\&, column sep=2.1cm]
\cdots\arrow{r}\&
\left ( \begin{array}{c}
V^{k-1} \\
W^{k-1}
 \end{array}\right )
 \arrow{r}{
 \begin{pmatrix} d^{k-1} & -S^{k-1} \\ 0 & d^{k-1} \end{pmatrix}
 }\&  \left ( \begin{array}{c}
V^{k}  \\
W^{k}
 \end{array}\right ) \arrow{r}{
 \begin{pmatrix} d^{k} & -S^{k} \\ 0 & d^{k} \end{pmatrix}
 } \&\left ( \begin{array}{c}
V^{k+1}   \\
W^{k+1}
 \end{array}\right )  \arrow{r}{} \&\cdots,
\end{tikzcd}
\end{equation}}
and the following BGG complex  is obtained by eliminating components connected by $S^{\bs}$:
\begin{equation} \label{BGG-seq}
\begin{tikzcd} 
\cdots\ar[r]   & \ker(S_{\dagger}^{J-2}) \ar[r,"Q\circ d"] & \ker(S_{\dagger}^{J-1})   \ar[r,"Q\circ d"] &  \ker(S_{\dagger}^{J}) \ar[r,"d"] & ~ \ar[ld,"\mathcal S^{-1}"] \\
 & &  &   \ar[r,"d"] &   \ker(S^{J+1})\ar[r,"d"] &  \ker(S^{J+2})  \ar[r,"d"] &   \cdots,
\end{tikzcd}
\end{equation}
where $S_{\dagger}^{k}: V^{k}\to W^{k-1}$ is the adjoint of $S^{k-1}: W^{k-1}\to V^{k}$ (for example, $\vskw$ is the adjoint of $\mskw$), and $Q$ is the projection to $\ker(S_{\dagger})$ (for example, if $S$ is $\mskw$, then $S_{\dagger}$ is $\vskw$, and $Q$ is the symmetrisation of a matrix). Note that \eqref{cplx:twisted} is a complex, i.e., 
$$
\left( \begin{array}{cc}
d & -S\\
0 & d
 \end{array}\right ) \circ \left( \begin{array}{cc}
d & -S\\
0 & d
 \end{array}\right ) =0,
$$
since $dS=-Sd$. The anti-commutativity also implies that $d^{k}$ maps $ \ker(S^{k})$ to  $ \ker(S^{k+1})$ in \eqref{BGG-seq}.


 The derivation of the 1D BGG complex and the 3D elasticity complex is summarised in the following diagrams:
   \begin{equation*}
  \adjustbox{scale=1,center}{%
\begin{tikzcd}[ampersand replacement=\&, column sep=1in]
0\arrow{r}{}\& \left (
\begin{array}{c}
{C^{\infty}}\\
{C^{\infty}} 
\end{array}
\right )\arrow{d}{\mathcal{F}^{0}}
\arrow{r}{\left (
\begin{array}{cc}
\frac{d}{dx} & 0\\
0&\frac{d}{dx}
\end{array}
\right ) } \&\left (
\begin{array}{c}
{C^{\infty}}  \\
{C^{\infty}} 
\end{array}
\right )\arrow{d}{\mathcal{F}^{1}}\arrow{r}{} \& 0\\
0\arrow{r}{}\& \left (
\begin{array}{c}
{C^{\infty}}\\
{C^{\infty}} 
\end{array}
\right )\arrow{d}{\pi^{0}}
\arrow{r}{\left (
\begin{array}{cc}
\frac{d}{dx} & -I \\
0&\frac{d}{dx}
\end{array}
\right ) } \&\left (
\begin{array}{c}
{C^{\infty}}  \\
{C^{\infty}} 
\end{array}
\right )\arrow{d}{\pi^{1}}\arrow{r}{}\& 0
\\
0 \arrow{r}\&C^{\infty}  \arrow{r}{\frac{d}{dx}}\&~\arrow[dl, "I"] \&~   \\
 ~\&~\arrow{r}{\frac{d}{dx}}\&C^{\infty} \arrow{r}{} \& 0.
 \end{tikzcd}}
\end{equation*}
  \begin{equation}\label{twisted-diagram}
  \adjustbox{scale=0.75,center}{%
\begin{tikzcd}[ampersand replacement=\&, column sep=1in]
0\arrow{r}{}\&  \left (
\begin{array}{c}
{C^{\infty}\otimes \mathbb{V}}\\
{C^{\infty}\otimes \mathbb{V}}
\end{array}
\right )\arrow{d}{\mathcal{F}^{0}} \arrow{r}{\left (
\begin{array}{cc}
\grad &0\\
0&\grad
\end{array}
\right ) } \&\left (
\begin{array}{c}
{C^{\infty}\otimes \mathbb{M}}\\
{C^{\infty}\otimes \mathbb{M}} 
\end{array}
\right )\arrow{d}{\mathcal{F}^{1}}
\arrow{r}{\left (
\begin{array}{cc}
\curl & 0\\
0&\curl
\end{array}
\right ) } \&\left (
\begin{array}{c}
{C^{\infty}\otimes \mathbb{M}}  \\
{C^{\infty}\otimes \mathbb{M}} 
\end{array}
\right )\arrow{d}{\mathcal{F}^{2}}\arrow{r}{\left (
\begin{array}{cc}
\div & 0\\
0&\div
\end{array}
\right ) } \&\left (
\begin{array}{c}
{C^{\infty}\otimes \mathbb{V}}  \\
{C^{\infty}\otimes \mathbb{V}} 
\end{array}
\right )\arrow{d}{\mathcal{F}^{3}}\arrow{r}{}\& 0\\
0\arrow{r}{}\& \left (
\begin{array}{c}
{C^{\infty}\otimes \mathbb{V}}\\
{C^{\infty}\otimes \mathbb{V}}
\end{array}
\right )\arrow{d}{\pi^{0}}  \arrow{r}{\left (
\begin{array}{cc}
\grad & \mskw  \\
0&\grad
\end{array}
\right ) } \&\left (
\begin{array}{c}
{C^{\infty}\otimes \mathbb{M}}\\
{C^{\infty}\otimes \mathbb{M}} 
\end{array}
\right )\arrow{d}{\pi^{1}}
\arrow{r}{\left (
\begin{array}{cc}
\curl & -\mathcal{S} \\
0&\curl
\end{array}
\right ) } \&\left (
\begin{array}{c}
{C^{\infty}\otimes \mathbb{M}}  \\
{C^{\infty}\otimes \mathbb{M}} 
\end{array}
\right )\arrow{d}{\pi^{2}}\arrow{r}{\left (
\begin{array}{cc}
\div &-2\vskw  \\
0&\div
\end{array}
\right ) } \&\left (
\begin{array}{c}
{C^{\infty}\otimes \mathbb{V}}  \\
{C^{\infty}\otimes \mathbb{V}} 
\end{array}
\right )\arrow{d}{\pi^{3}}\arrow{r}{}\& 0
\\
0 \arrow{r}\&C^{\infty}\otimes \mathbb{V}\arrow{r}{\deff} \&C^{\infty}\otimes \mathbb{S}  \arrow{r}{\curl}\&~\arrow[dl, "\mathcal{S}^{-1}"] \&~  \&~ \\
 ~\&~\&~\arrow{r}{\curl}\&C^{\infty}\otimes \mathbb{S} \arrow{r}{\div} \& C^{\infty}\otimes \mathbb{V} \arrow{r}{} \& 0.
 \end{tikzcd}}
\end{equation}
Here the first row consists of a direct sum of de~Rham complexes. The second row is an example of the twisted complex
\eqref{cplx:twisted}.

The twisted version of the elasticity complex has a neat correspondence to Riemann--Cartan geometry \cite{christiansen2023extended}:
  \begin{equation*}
\begin{tikzcd}
 \mbox{embedding} \arrow{r}{\grad} &\mbox{coframes}\arrow{r}{\curl} &\mbox{torsion} \arrow{r}{\div}  &\cdots \\
 \mbox{rotation}\arrow{r}{\grad} \arrow[ur, "-\mskw"]&\mbox{connection} \arrow{r}{\curl} \arrow[ur, "\mathcal{S}"]&\mbox{curvature} \arrow{r}{\div}\arrow[ur, "2\vskw"] &\mbox{(co)vector (in Bianchi identity)},
 \end{tikzcd}
\end{equation*}
The twisted complex involves both curvature and torsion, while the elasticity (BGG) complex derived from the twisted complex only involves curvature. The BGG process is thus a cohomology-preserving elimination of torsion from Riemann--Cartan geometry. In Section \ref{sec:solid}, we will investigate the mechanics meaning of this elimination -- the first space in the bottom row representing a rotational degree of freedom in generalised continuum is eliminated to obtain the standard elasticity model.

\subsection{Conformal complex: conformal geometry}\label{sec:conformal}

Before discussing the construction, we first show the geometric meaning of the conformal deformation complex \eqref{conformal-complex-smooth}. 
 Given a metric $g$, a vector field $v\in \mathscr{X}(M)$ is called a {\it conformal Killing field} if $\mathcal{L}_{v}g=\varphi g$, where $\varphi$ is some scalar function. That is, the change of a metric  along a conforming Killing field is proportional to itself. By the formulas for Lie derivatives, $
 (\nabla v)_{ij}+(\nabla v)_{ji}=\varphi g_{ij}$. Taking the trace, we have $\varphi=\frac{2}{n}\div v$. Thus, $v$ is a conformal Killing field if and only if $\nabla v+\nabla^{T}v-\frac{2}{n}\div v g=0$. If $g$ is the Euclidean metric, i.e., $g=I$, then $\nabla v+\nabla^{T}v-\frac{1}{n}\div v I=2\dev\deff v$ is (two times) the traceless symmetric gradient. For the 1D case ($n=1$), $\dev\deff v$ is trivial as $\deff v$ is a scalar. For $n=2$, let $u=(u^{1}, u^{2})$. Then 
 $$
\dev\sym\grad u=\left (\begin{array}{cc}
\frac{1}{2}({\partial_{x}u_{1}-\partial_{y}u_{2}}) & \frac{1}{2}({\partial_{y}u_{1}+\partial_{x}u_{2}})\\ 
 \frac{1}{2}(\partial_{y}u_{1}+\partial_{x}u_{2}) &-\frac{1}{2}(\partial_{x}u_{1}-\partial_{y}u_{2})
\end{array}\right ),
$$
 which exactly consists of the components of the Cauchy-Riemann operator. This is a clear indication that the conformal Killing field is closely related to complex analysis in 2D.
 
 A Riemannian metric $g$ is conformally flat if it can be locally written as a conformal rescaling of the Euclidean metric, i.e., there exists $\phi(x)$ such that $g_{ij}=e^{\phi(x)}\delta_{ij}$. The Cotton tensor measures the deviation of a metric from being conformally flat. In terms of the Ricci curvature $R_{ij}$ and the scalar curvature $R$, the Cotton tensor is a third order tensor defined by 
 $$
 C_{ijk}=\nabla_{k}R_{ij}-\nabla_{j}R_{ik}+\frac{1}{2(n-1)}(\nabla_{j}Rg_{ik}-\nabla_{k}Rg_{ij}).
 $$
 Clearly, $ C_{ijk}$ is skew-symmetric with respect to the indices $j$ and $k$. Identifying a skew-symmetric 2-tensor with a vector in 3D (taking the Hodge dual), one obtains a 2-tensor
 $$
 C_{i}^{j}=\epsilon^{k\ell j}\nabla_{k}(R_{\ell i}-\frac{1}{4}Rg_{ij}). 
 $$
 In the linearised case, this is exactly the second operator $\cot$ in the conformal deformation complex. 
 
 Finally, $\div$ is a compatibility condition analogous to the differential Bianchi identity for the Riemann tensor.

The construction of the conformal deformation complex \eqref{conformal-complex-smooth} follows from the following diagram. 
    \begin{equation}\label{diagram:conformal} \adjustbox{scale=0.8,center}{%
\begin{tikzcd}
0 \arrow{r}&C^{\infty}\otimes \mathbb{V}\arrow{r}{\grad} &C^{\infty}\otimes \mathbb{M}  \arrow{r}{\curl} &C^{\infty}\otimes \mathbb{M} \arrow{r}{\div} & C^{\infty}\otimes \mathbb{V} \arrow{r}{} & 0\\
0 \arrow{r} &C^{\infty}\otimes (\mathbb{R}\oplus\mathbb{V})\arrow{r}{\grad}   \arrow[ur, "(\iota~-\mskw)"] &C^{\infty}\otimes (\mathbb{V}\oplus \mathbb{M})  \arrow{r}{\curl}   \arrow[ur, "(-\mskw~\mathcal{S})^{T}"]&C^{\infty}\otimes (\mathbb{V}\oplus \mathbb{M}) \arrow{r}{\div} \arrow[ur, "(I~2\vskw)^{T}"] & C^{\infty}\otimes (\mathbb{R}\oplus \mathbb{V})  \arrow{r}{} & 0\\
0 \arrow{r}&C^{\infty}\otimes \mathbb{V}\arrow{r}{\grad}  \arrow[ur, "(I~-\mskw)"]&C^{\infty}\otimes \mathbb{M}  \arrow{r}{\curl}  \arrow[ur, "(2\vskw~\mathcal{S})"]&C^{\infty}\otimes \mathbb{M} \arrow{r}{\div} \arrow[ur, "(\tr~ 2\vskw)"] & C^{\infty}\otimes \mathbb{V}  \arrow{r}{} & 0.
 \end{tikzcd}}
\end{equation}
\vbox{\vspace{-9.9cm}
\leftline{\hspace{2.9cm}\vspace{0.3cm}{
\begin{tikzpicture}
\draw[very thick, opacity=.75, ->, rounded corners] (6.2,1.3) -- ++(0.8, 0.0) -- ++(-0.8,-0.8) -- ++(0.8,0) -- ++(-0.8,-0.8)  -- ++(0.8,0)  ;
\draw[very thick, opacity=.75, ->] (3.4,1.5) -- ++(.4,0) ;
\draw[very thick, opacity=.75, ->] (9.4,-0.6) -- ++(.4,0) ;
\draw[very thick,opacity=.75, red] (2.1,1.3) ellipse (7mm and 3.5mm);
\draw[very thick,opacity=.75, red] (5.1,1.3) ellipse (7mm and 3.5mm);
\draw[very thick,opacity=.75, red] (8.1,-0.8) ellipse (7mm and 3.5mm);
\draw[very thick,opacity=.75, red] (11.1,-0.8) ellipse (7mm and 3.5mm);
\end{tikzpicture}}
}}
The diagram has three rows. The first and the last rows are vector-valued de~Rham complexes, while the row in the middle is a sum of a scalar and a vector version. To understand this, recall that we want to obtain a symmetric and traceless ($\mathbb{S}\cap \mathbb{T}$) matrix. We want to remove the trace part (scalar) and the skew-symmetric part (identified with a vector $\mathbb{R}^{3\times 3}_{\skw}\cong \mathbb{V}$) from the matrix in the second space of the first row. This is why a $(\mathbb{R}\oplus \mathbb{V})$-valued de~Rham complex is used for the second row. The $\mathbb{R}$ value eliminates the trace, and $\mathbb{V}$ eliminates the skew-symmetric part of the matrix. The construction follows a similar recipe as the cases with two rows, i.e., we eliminate components connected by the diagonal maps as much as possible. For example, in the following sequence
$$
\begin{tikzpicture}[>=stealth]

\node (A) at (0,0) {$C^{\infty}\otimes \mathbb {V}$};
\node (B) at (1.5,1.5) {$C^{\infty}\otimes (\mathbb {V}\oplus \mathbb M)$};
\node (C) at (3,3) {$C^{\infty}\otimes \mathbb {M}$};

\draw[->] (A) -- (B);
\draw[->] (B) -- (C);

\end{tikzpicture}
$$
$C^{\infty}\otimes (\mathbb {V}\oplus \mathbb M)$ splits into two parts: the $\mathbb V$ part eliminates with $C^{\infty}\otimes \mathbb {V}$ from downstairs, and the $\mathbb M$ part cancels with $C^{\infty}\otimes \mathbb {M}$ from upstairs. Therefore, this sequence is completely eliminated in the diagram. What remains after this elimination process is the conformal complex \eqref{conformal-complex-smooth}. Another way to think about this construction is to observe that each sequence in the diagonal direction is also a complex (which trivially holds for diagrams with two rows). Then each space has an algebraic Hodge decomposition according to this diagonal complex (see Section \ref{sec:hodge}). The final BGG complex consists of the ``harmonic forms'' (cohomology) of this decomposition.

Similarly, we can construct another complex with tracefree and symmetric matrices, referred to as the {\it conformal Hessian complex}. We start from the diagram
 \begin{equation*} 
 \adjustbox{scale=1,center}{%
\begin{tikzcd}
0 \arrow{r} &C^{\infty}\otimes \mathbb{R}  \arrow{r}{\grad} &C^{\infty}\otimes \mathbb{V} \arrow{r}{\curl}\arrow[dl, "I", shift left=1] &C^{\infty}\otimes \mathbb{V} \arrow{r}{\div}\arrow[dl, "\mskw", shift left=1] & C^{\infty}\otimes \mathbb{R} \arrow{r}{}\arrow[dl, "1/3\iota", shift left=1] & 0\\
0 \arrow{r}&C^{\infty}\otimes \mathbb{V}\arrow{r}{\grad} \arrow[u, "x\cdot "]\arrow[ur, "I"]&C^{\infty}\otimes \mathbb{M}  \arrow{r}{\curl} \arrow[ur, "2\vskw"]\arrow[u, "x\cdot "]\arrow[dl, "1/3\tr", shift left=1]&C^{\infty}\otimes \mathbb{M} \arrow{r}{\div}\arrow[ur, "\tr"] \arrow[u, "x\cdot "]\arrow[dl, "-2\vskw", shift left=1]&\arrow[u, "x\cdot "]C^{\infty}\otimes \mathbb{V} \arrow{r}{} \arrow[dl, "I", shift left=1]& 0\\
0 \arrow{r} &C^{\infty}\otimes \mathbb{R}\arrow[u, "x\otimes "]\arrow{r}{\grad} \arrow[ur, "\iota"]&C^{\infty}\otimes \mathbb{V}  \arrow[u, "x\otimes "]\arrow{r}{\curl} \arrow[ur, "-\mskw"]\arrow[u, "x\otimes "]&C^{\infty}\otimes \mathbb{V}\arrow[u, "x\otimes "] \arrow{r}{\div}\arrow[ur, "I"] & C^{\infty}\otimes \mathbb{R} \arrow[u, "x\otimes "]\arrow{r}{} & 0.
 \end{tikzcd}}
\end{equation*}     
\vbox{\vspace{-6.5cm}
\leftline{\hspace{-.2cm}\vspace{0.3cm}{
 \adjustbox{scale=1.2,center}{%
\begin{tikzpicture}
\draw[very thick, opacity=.75, ->, rounded corners] (3,1.5) -- ++(0.8, 0.0) -- ++(-1.5,-0.8) -- ++(1,0)  ;
\draw[very thick, opacity=.75, ->] (5.3,0.5) -- ++(.4,0) ;
\draw[very thick, opacity=.75, ->, rounded corners](7.5,0.5) -- ++(0.8, 0.0) -- ++(-1.5,-0.8) -- ++(1,0)  ;
\draw[very thick,opacity=.75, red] (2.,1.3) ellipse (7mm and 3.5mm);
\draw[very thick,opacity=.75, red] (4.2,0.3) ellipse (7mm and 3.5mm);
\draw[very thick,opacity=.75, red] (6.5,0.2) ellipse (7mm and 3.5mm);
\draw[very thick,opacity=.75, red] (8.8,-0.7) ellipse (7mm and 3.5mm);
\end{tikzpicture}}
}
}}
The elimination leads to the following {\it conformal Hessian complex}:
 \begin{equation} \label{conformal-hessian}
\begin{tikzcd}
0 \arrow{r} &C^{\infty}\otimes \mathbb{R}  \arrow{r}{\dev\hess} &C^{\infty}\otimes (\mathbb{S}\cap \mathbb{T}) \arrow{r}{\sym\curl} &C^{\infty}\otimes (\mathbb{S}\cap \mathbb{T}) \arrow{r}{\div\div} & C^{\infty}\otimes \mathbb{R} \arrow{r}{} & 0
 \end{tikzcd}
 \end{equation}

\subsection{The BGG machinery}

The connecting maps between the de~Rham complexes, the twisted complexes and the BGG complexes are available and explicit (such as those in \eqref{twisted-diagram}; see \cite[(67)]{arnold2021complexes}, \cite[(21), (30)]{vcap2023bgg}). The existence of these maps not only indicates that these complexes all have isomorphic cohomology, but they also provide a machinery for carrying over results for the de~Rham complexes (e.g., in electromagnetism) to results for other applications related to the twisted and BGG complexes. In this section, we show an example of using this idea to derive the Poincar\'e operators. 

Given a complex
$$
\begin{tikzcd}
\cdots \arrow[r] & V^{i-1} 
\arrow[r, "d^{i-1}"] 
\arrow[r, leftarrow, "P^{i}"', yshift=-0.5ex] & V^i 
\arrow[r, "d^{i}"] 
\arrow[r, leftarrow, "P^{i+1}"', yshift=-0.5ex] & V^{i+1} \arrow[r] & \cdots,
\end{tikzcd}
$$
the Poincar\'e operators   ${P^{k}}: V^{k}\mapsto V^{k-1}$ are defined to satisfy the null-homotopy property 
$$
{{d}^{k-1}}{P^{k}}+{P^{k+1}}{{d}^{k}}=I_{V^{k}},
$$
The existence of Poincar\'e operators immediately implies the exactness of the complex, as
$$
du=0 \quad \Longrightarrow \quad u=(dP+Pd)u=d(Pu).
$$
The Poincar\'e operators are widely used for PDE analysis \cite{costabel2010bogovskiui,guzman2021estimation} and construct finite element spaces \cite{hiptmair1999canonical,Arnold.D;Falk.R;Winther.R.2006a,christiansen2018generalized}.  

For the first slot of a de~Rham complex, suppose $\vec u=\grad \varphi$. Once we know $\vec u$, then $\varphi$ can be recovered by integrating along a curve. This is the 0-th Poincar\'e operator in the de~Rham complex. The construction for higher form degrees is also given by integration. 

A natural question related to elasticity is the following: suppose that the strain tensor $\ten e=\deff \vec v$, where $\vec v$ is the displacement. Can we recover $\vec v$ from $\ten e$? In fact, this question was answered by the formulas by Ces\`aro and Volterra in 1906 and 1907 \cite{cesaro1906sulle,volterra1907equilibre}. The formula of recovering $\ten e$ from $\vec v$ not only involves integration along curves, but also involves derivatives of $\ten e$. The Ces\`aro--Volterra formula is a crucial ingradient for intrinsic theories of elasticity \cite{ciarlet2009intrinsic,ciarlet2009cesaro,van2016frank}.

From a complex and BGG point of view, the question boils down to constructing Poincar\'e operators for the elasticity complex. The BGG machinery comes into play here \cite{christiansen2020poincare,vcap2023bounded}: since the Poincar\'e operators $P^{\bs}$ for the de~Rham complexes are known, one may follow the BGG diagrams \eqref{poincare-bgg} to obtain the Poincar\'e operators $ \mathscr{P}^{\bs}$ for the elasticity complex $(\Upsilon^{\bs}, \mathscr{D}^{\bs})$ as
 $$
 \mathscr{P}^{i+1}=\pi^{i}F^{i}P^{i+1}(F^{i+1})^{-1}A^{i+1}.
 $$
    \begin{minipage}{.5\textwidth}
        \centering
\begin{equation}\label{poincare-bgg}
\begin{tikzcd}[swap]
Y^{i}\arrow{d}{F} \arrow[r, "d^{i}", shift right=1] &\arrow{l}{P^{i+1}}Y^{i+1}\arrow{d}{F}\\
Y^{i} \arrow[d, "\pi^{i}", shift right=1] \arrow[r, "d_{V}^{i}", shift right=1]&\arrow{l}{P_{V}^{i+1}}Y^{i+1}\arrow[d, "\pi^{i+1}", shift right=1] \\
\Upsilon^{i}\arrow[r, "\mathscr{D}^{i}", shift right=1]\arrow{u}{A^{i}}&\arrow{l}{\mathscr{P}^{i+1}}\Upsilon^{i+1}\arrow{u}{A^{i+1}}
\end{tikzcd}
\end{equation}
\vbox{\vspace{-6.5cm}
\leftline{\hspace{2.1cm}\vspace{0.cm}{
\begin{tikzpicture}
\draw[very thick,opacity=.75, blue, <-]  (0.3,0.5) -- ++(0,2.4);
\draw[very thick,opacity=.75, blue, ->]  (3.5,0.5) -- ++(0,2.4);
\draw[very thick,opacity=.75, blue, <-]  (0.4,3.2) -- ++(3,0);
\draw[very thick,opacity=.75, red, <-]  (0.4,-.2) -- ++(3,0);
\end{tikzpicture}}
}}
\end{minipage}
    \begin{minipage}{.5\textwidth}
        \centering
$(Y^{\bs}, d^{\bs})$:   de~Rham complex 
  \vspace{+0,6cm}\\
$(Y^{\bs}, d_{V}^{\bs})$:   twisted complex   
  \vspace{+0,6cm}\\
  $(\Upsilon^{\bs}, \mathscr{D}^{\bs})$: BGG complex
\end{minipage}

This rather different homological algebraic approach recovers the classical Cas\`aro-Volterra formula
\begin{align*} 
\mathscr{P}^{1}(e)&=\int_{0}^{x}e(y)\cdot\,dy-\int_{0}^{x}(\mskw\int_{0}^{y}(e({z})\times \nabla)^{T}\cdot dz)\cdot dy
\end{align*}
as a special case.

\section{Analysis and PDEs}\label{sec:analysis}

In the previous section, we considered complexes built from smooth functions and vector fields. However, one can also construct similar complexes using more general function spaces, which allows algebraic properties (particularly the finite-dimensionality of cohomology) to yield various analytic consequences. 
 
On the interval \(I = (0,1)\), we have the one-dimensional de~Rham complex
\[
\begin{tikzcd}
0 \arrow{r} & H^k(I) \arrow{r}{\tfrac{d}{dx}} & H^{k-1}(I) \arrow{r} & 0,
\end{tikzcd}
\]
where \(H^k(I)\) is the Sobolev space of functions whose weak derivatives up to order \(k\) lie in \(L^2(I)\):
\[
H^k(I) := \bigl\{\,u \in L^2(I)\,:\,\tfrac{d^j u}{dx^j} \in L^2(I),\,j=1,\dots,k\bigr\}.
\]
We also define \(H_0^k(I)\) to be the completion of \(C_0^\infty(I)\) under the \(H^k\) norm, which can be viewed as \(H^k(I)\) functions that vanish (in the appropriate sense) at the boundary. Consequently, for \(k \ge 1\), we get another complex
\[
\begin{tikzcd}
0 \arrow{r} & H_0^k(I) \arrow{r}{\tfrac{d}{dx}} & H_0^{k-1}(I) \arrow{r} & 0.
\end{tikzcd}
\]
By convention, we set \(H^0(I) := L^2(I)\) and \(H_0^0(I) := L_0^2(I) := \bigl\{u \in L^2(I)\,:\,\int_I u\,dx = 0\bigr\}.\)

Similarly, in three dimensions on a domain \(\Omega \subset \mathbb{R}^3\), the de~Rham complex  has the Sobolev version
\begin{equation}\label{deRham:3D-sobolev}
\begin{tikzcd}
0 \arrow{r}{} & H^{q}(\Omega) \arrow{r}{\grad}& H^{q-1}(\Omega)\otimes \mathbb{R}^{3} \arrow{r}{\curl} & H^{q-2}(\Omega)\otimes \mathbb{R}^{3} \arrow{r}{\div} & H^{q-3}(\Omega) \arrow{r}{} &  0,
\end{tikzcd}
\end{equation}
where \(q\) is a real number. There are also variants with boundary conditions (e.g.\ taking closures of $C^{\infty}_{0}$ fields).

For the analysis of many problems in electromagnetism and other fields, another kind of regularity plays an important role:
\begin{equation}\label{deRham:3D-Hd}
\begin{tikzcd}
0 \arrow{r}{} & H^{1}(\Omega) \arrow{r}{\grad}& H(\curl, \Omega) \arrow{r}{\curl} & H(\div, \Omega)\arrow{r}{\div} & L^{2}(\Omega) \arrow{r}{} &  0,
\end{tikzcd}
\end{equation}
where
\[
H(\mathrm{curl},\Omega) := \bigl\{\,u \in L^2(\Omega)\otimes\mathbb{R}^3 : \nabla\times u \in L^2(\Omega)\otimes\mathbb{R}^3 \bigr\},
\]
and similarly for \(H(\mathrm{div},\Omega)\) and more general spaces \(H(d,\Omega)\), where $d$ is a linear differential operator.  Clearly, \eqref{deRham:3D-Hd} is only one example of many. For instance, the sequence
\begin{equation}\label{deRham:3D-stokes}
\begin{tikzcd}
0 \arrow{r}{} & H^{1}(\Omega) \arrow{r}{\grad}& H^{0,1}(\curl, \Omega) \arrow{r}{\curl} & H^1\otimes \mathbb{R}^{3}\arrow{r}{\div} & L^{2}(\Omega) \arrow{r}{} &  0,
\end{tikzcd}
\end{equation}
is also a complex, where $H^{0,1}(\curl, \Omega) :=\{u\in L^{2}(\Omega)\otimes \mathbb{R}^{3}: \curl u\in H^{1}\otimes \mathbb{R}^{3}\}$. 

A fundamental theorem concerning these Sobolev complexes was proved by Costabel and McIntosh~\cite{costabel2010bogovskiui}:

\begin{theorem}[Costabel--McIntosh]\label{thm:costabel-mcintosh}
Let \(\Omega\subset\mathbb{R}^3\) be a strong Lipschitz domain. Then for any real \(q\), the complex \eqref{deRham:3D-sobolev} has \emph{uniform cohomology representatives}:
\[
\ker\bigl(d^k,H^{q-k}\Lambda^k(\Omega)\bigr)
= d^k\bigl(H^{q-(k-1)}\Lambda^{k-1}(\Omega)\bigr)
+ \mathscr{H}^k,
\]
where \(\mathscr{H}^k \subset C^\infty\) is a finite-dimensional space independent of \(q\).
\end{theorem}

Similar statements hold when boundary conditions are imposed (i.e.\ taking the closure of \(C^\infty_0\)-forms in the appropriate Sobolev norms).

Theorem \ref{thm:costabel-mcintosh} on the cohomology of the Sobolev complex also implies the cohomology of other variants. For example, the same claim also holds for \eqref{deRham:3D-Hd} and \eqref{deRham:3D-stokes}.  The idea for showing results for complexes like \eqref{deRham:3D-Hd} and \eqref{deRham:3D-stokes} is to observe that, e.g., $H(\curl, \Omega)$ is in between $H^{1}(\Omega)\otimes \mathbb{R}^{3}$ and  $L^{2}(\Omega)\otimes \mathbb{R}^{3}$. These two spaces fit in two de~Rham complexes, one with $q=2$ and one with $q=1$. As Theorem \ref{thm:costabel-mcintosh} is independent of $q$ and works for both cases, we draw the same conclusion for spaces of $H(\curl, \Omega)$ type.

In the remainder of this section, we illustrate how the finite-dimensionality of these cohomology groups leads to a variety of analytic consequences for PDEs and their solutions. To illustrate this idea and emphasise the underlying algebraic structures, we omit many details and follow an informal style.  A rigorous discussion involving unbounded operators and Hilbert complexes can be found in \cite{arnold2018finite,Arnold.D;Falk.R;Winther.R.2006a,Arnold.D;Falk.R;Winther.R.2010a}.

\subsection{Adjoint operators and Hodge--Laplacian problems}

Associated with each complex, we can define Hodge--Laplacian problems. First, let $d^{\ast}_{k}$ be the formal $L^{2}$ adjoint of $d^{k-1}$. For example, from the integration by parts
$$
\int_{\Omega}\grad u \cdot  v\, dx=-\int_{\Omega}u\div v\,dx+\mbox{ boundary terms },
$$
we say that the formal adjoint of $\grad$ is $-\div$ (and thus the formal adjoint of $\div$ is $-\grad$). Similarly, from 
$$
\int_{\Omega}\curl v \cdot  w\, dx=\int_{\Omega}v\cdot \curl w\,dx+\mbox{ boundary terms },
$$
we say that the formal adjoint of $\curl$ is still $\curl$ ($\curl$ is formally self-adjoint). The adjoint operators also form a complex:
\begin{equation}\label{3D-deRham-adjoint}
\begin{tikzcd}[swap]
    0 &\arrow[l]  C^{\infty}(\Omega) &\arrow[l, "-\div"]  C^{\infty}(\Omega)\otimes  \mathbb{R}^{3} &\arrow[l, "\curl"]  C^{\infty}(\Omega)\otimes  \mathbb{R}^{3} &\arrow[l, "-\grad"] C^{\infty}(\Omega) &\arrow[l]   0.
\end{tikzcd}
\end{equation}
It is a coincidence for the 3D de~Rham complex that the adjoint complex has the same form, but this is not always the case (not even for the 2D de~Rham complex). The Hodge--Laplacian at $V^{k}$ is $d^{\ast}_{k}d^{k}+d^{k-1}d_{k-1}^{\ast}$. The {\it Hodge--Laplacian problem} refers to the PDE
$$
(d^{\ast}_{k}d^{k}+d^{k-1}d_{k-1}^{\ast})u=f, 
$$
for a given right hand side $f$.

 Concerning the example \eqref{deRham:3D}, $d$ and $d^{\ast}$ at $V^{0}=C^{\infty}(\Omega)$ are
$$
\begin{tikzcd}
0 \textcolor{red}{\arrow[r,  ->, shift left=1]\arrow[r,  <-, shift right=1]}& {C^{\infty}(\Omega)}  {\arrow[r,  ->,  "\grad", shift left=1]\arrow[r,  <-, swap, "-\div", shift right=1]} &C^{\infty}(\Omega;\R^3) \quad & 
C^{\infty}(\Omega;\R^3) \quad & C^{\infty}(\Omega)\quad &0.
 \end{tikzcd}
$$
Here $d^{\ast}$ is trivial (mapping everything to zero). The Hodge--Laplacian problem at $V^{0}$ is thus
\begin{equation}\label{HL-0}
-\div\grad u =f,
\end{equation}
which is the Poisson equation.

 Similarly, at $V^{1}=C^{\infty}(\Omega)\otimes \mathbb{R}^{3}$, the $d$ and $d^{\ast}$ operators are demonstrated as follows:
 $$
\begin{tikzcd}
0 &{C^{\infty}(\Omega)}  {\arrow[r,  ->,  "\grad", shift left=1]\arrow[r,  <-, swap, "-\div", shift right=1]} & {C^{\infty}(\Omega;\R^3)}  {\arrow[r,  ->,  "\curl", shift left=1]\arrow[r,  <-, swap, "\curl", shift right=1]} & 
C^{\infty}(\Omega;\R^3) \quad & C^{\infty}(\Omega)\quad &0,
 \end{tikzcd}
$$
and therefore the 
Hodge--Laplacian problem is 
\begin{equation}\label{HL-1}
-\grad\div v+\curl\curl v =f.
\end{equation}
If $\div v=0$, this is reduced to a curl-curl system $\curl\curl v=f$, which is related to the Maxwell equation. 

The cases of $V^{2}$ and $V^{3}$ are similar. One formally obtains the same $-\grad\div+\curl\curl$ system and the Poisson equation. 

 \begin{remark}\label{rmk:boundary-conditions}
The presentation above is somewhat formal. A more rigorous treatment uses the framework of \emph{Hilbert complexes} \cite{Arnold.D;Falk.R;Winther.R.2006a,bru1992hilbert}, in which one begins with a sequence of \(L^2\)  spaces and views each differential operator as an \emph{unbounded operator}:
\begin{equation}\label{deRham:L2}
\begin{tikzcd}
0 \arrow{r}{} & L^{2}(\Omega) \arrow{r}{\grad}& L^{2}(\Omega)\otimes \mathbb{R}^{3} \arrow{r}{\curl} &L^{2}(\Omega)\otimes \mathbb{R}^{3} \arrow{r}{\div} & L^{2}(\Omega) \arrow{r}{} &  0.
\end{tikzcd}
\end{equation}
Here the $\grad$, $\curl$, $\div$ operators do not map $L^{2}$ spaces to $L^{2}$, but they do so on certain subspaces of $L^{2}$. Specifying the domains of these unbounded operators, along with defining their adjoints in a precise manner, automatically encodes the appropriate boundary conditions for the variables. For instance, if we specify \eqref{deRham:3D-Hd} as the domain which does not include explicit boundary conditions, then its adjoint operator will. Consequently, the adjoint Hilbert complex for the 3D $L^{2}$ de~Rham sequence has the domain
\begin{equation}\label{3D-deRham-adjoint-Hd}
\begin{tikzcd}[swap]
0 
&\arrow[l] L^{2}(\Omega)
&\arrow[l, "-\div"] H_{0}(\div,\Omega)
&\arrow[l, "\curl"] H_{0}(\curl,\Omega)
&\arrow[l, "-\grad"] H^{1}_{0}(\Omega)
&\arrow[l] 0.
\end{tikzcd}
\end{equation}
Alternatively, one may specify \eqref{3D-deRham-adjoint-Hd} as the domain complex. Then its adjoint complex \eqref{deRham:L2} has \eqref{deRham:3D-Hd} as the domain. 
\end{remark}

\begin{remark}
We have the vector identity
$$
(\curl\curl-\grad\div)\vec u=-\Delta \vec u,
$$
where \(\Delta\) denotes the scalar Laplacian applied componentwise to \(u \in \mathbb{R}^3\). Hence, the Hodge--Laplacian on 1-forms or 2-forms in three dimensions is often referred to as the \emph{vector Laplacian}. From a de~Rham complex point of view, however, working with $\curl\curl-\grad\div$ can differ significantly from treating each component as a separate scalar Laplacian. In particular, the boundary conditions that emerge in the Hilbert complex setting are not the same as the usual Dirichlet condition \(u|_{\partial \Omega}=0\) for scalar problems.
\end{remark}

\subsection{Hodge decomposition}\label{sec:hodge}

Any three-dimensional flow can be decomposed into the sum of a rational part ($\curl$ of a field) plus a potential part (gradient of a function) plus another part representing the topology of the domain. The idea can be precisely stated in terms of complexes.


Given a complex 
$$
\begin{tikzcd}
\cdots \arrow{r} & V^{k-1} {\arrow[r,  ->,  "d^{k-1}", shift left=1]\arrow[r,  <-, swap, "d_{k}^{\ast}", shift right=1]} & V^{k}{\arrow[r,  ->,  "d^{k}", shift left=1]\arrow[r,  <-, swap, "d_{k+1}^{\ast}", shift right=1]} & V^{k-1}\arrow{r}&\cdots,
\end{tikzcd}
$$
we have the Hodge decomposition at $V^{k}$ as the sum of  a part coming from the left (range of $d^{k-1}$ on its domain)  plus a part coming from the right (range of $d^{\ast}_{k+1}$ on its domain) plus a harmonic part orthogonal to both, i.e., 
\begin{equation*}
V^{k}=\ran(d^{k-1})\oplus \ran(d^{\ast}_{k+1}) \oplus \mathscr{H}^{k}.
\end{equation*}
The above decomposition is based on the fact that $\ran(d^{k-1})$ is orthogonal to the kernel of $d^{\ast}_{k}$ (as $(d^{k-1}v, w)=(v, d_{k}^{\ast}w)=0$, for any $v$ and $w\in \ker(d_{k}^{\ast})$). Furthermore, the kernel of $d^{\ast}_{k}$ can be decomposed as the range of $d^{\ast}_{k+1}$ plus the harmonic part $\mathscr{H}^{k}:=\ker (d^{k})\cap \ker(d^{\ast}_{k})$. Alternatively, one may use a similar argument based on the fact that the range of $d^{\ast}_{k+1}$ is orthogonal to the kernel of $d^{k}$.

The argument above requires that the formal adjoint matches the rigorous adjoint of the corresponding unbounded operators, which is precisely how boundary conditions become incorporated (see Remark~\ref{rmk:boundary-conditions}).  
Typically, one first specifies a Hilbert complex of base spaces (often \(L^2\)). The domains of \(d\) and \(d^*\) are then chosen so that boundary terms arising from integration by parts vanish. In other words, the requirement that
\[
(d^{k-1}v,\,w)\;=\;(v,\,d^*_k w)
\]
be well-defined forces either \(v\) or \(w\) to satisfy certain boundary conditions. Consequently, those boundary conditions emerge naturally from the adjoint operators in the Hilbert complex framework.

\begin{example}[Hodge decomposition of the de~Rham complex in 3D]
Consider the $L^{2}$ de Rham complex \eqref{deRham:L2}. We specify \eqref{deRham:3D-Hd} as the {\it domain complex}. The adjoint operators also form a complex which formally has the same form as \eqref{deRham:L2}. The domain complex of this adjoint complex is \eqref{3D-deRham-adjoint-Hd}. Here we did not specify boundary conditions in \eqref{deRham:3D-Hd}. Thus its adjoint \eqref{3D-deRham-adjoint-Hd} should incorporate certain vanishing boundary conditions. 

For the ``1-forms'' in \eqref{deRham:L2},
we have 
\begin{equation}\label{hodge-decomposition-example}
L^{2}\otimes \mathbb{R}^{3}=\ran(\grad)\oplus \ran(\curl^{\ast}) \oplus \mathscr{H}= \grad H^{1}(\Omega)\oplus \curl H_{0}(\curl, \Omega)\oplus \mathscr{H},
\end{equation}
where $\mathscr{H}=\ker (\curl, H(\curl, \Omega))\cap \ker (\div, H_{0}(\div, \Omega))$. Here $\ran(\grad)$ is from the domain complex of $d^{\bs}$ \eqref{deRham:3D-Hd}, and $\ran(\curl^{\ast})$ is from the domain of the adjoint complex \eqref{3D-deRham-adjoint-Hd}.
The decomposition also gives 
$$
\ker (\curl, H(\curl, \Omega))=\grad H^{1}(\Omega)\oplus \mathscr{H},
$$
and 
$$
\ker (\div, H_{0}(\div, \Omega))=\curl H_{0}(\curl, \Omega)\oplus \mathscr{H}.
$$

The Hodge decomposition of \eqref{deRham:3D-Hd} at ``2-forms'' is 
$$
L^{2}\otimes \mathbb{R}^{3}=\ran(\curl)\oplus \ran(\div^{\ast}) \oplus \mathscr{H}=\curl H(\curl, \Omega)\oplus \grad H_{0}^{1}(\Omega) \oplus \mathscr{H}_{0},
$$
where $\mathscr{H}_{0}= \ker (\div, H(\div, \Omega))\cap \ker (\curl, H_{0}(\curl, \Omega))$. The decomposition gives 
$$
\ker (\curl, H_{0}(\curl, \Omega))=\grad H_{0}^{1}(\Omega)\oplus \mathscr{H}_{0},
$$
and 
$$
\ker (\div, H(\div, \Omega))=\curl H(\curl, \Omega)\oplus \mathscr{H}_{0}.
$$
\end{example}

\subsection{Poincar\'e-Korn inequalities}

On a bounded Lipschitz domain, the following inequalities hold:
\begin{itemize}
\item {Poincaré inequalities:}
\begin{equation}\label{poincare-inequality}
\|u\|_{H^{1}}\le C\|\nabla u\|_{L^{2}},\quad u\in H^{1},\;u\perp \ker(\nabla).
\end{equation}
A straightforward generalisation to a vector-valued function \(\mathbf{v}\) implies
\[
\|\mathbf{v}\|_{H^{1}}\le C\|\nabla \mathbf{v}\|_{L^{2}}.
\]
The matrix \(\nabla \mathbf{v}\in \mathbb{R}^{3\times 3}\) has nine components, so controlling \(\|\mathbf{v}\|_{H^{1}}\) effectively uses all nine.

\item {Korn inequality:}
\[
\|\mathbf{u}\|_{H^{1}}\le C \|{\sym}\,\grad \mathbf{u}\|_{L^{2}},\quad \mathbf{u}\perp \ker({\sym}\grad).
\]
Here, \(\mathbf{u}\) is controlled by the symmetric part of its gradient \({\sym}\grad \mathbf{u}\). In 3D, \({\sym}\grad \mathbf{u}\) has six independent components.

\item {Conformal Korn inequality:}
\[
\|\mathbf{u}\|_{H^{1}}\le C \|{\dev}\,{\sym}\grad  \mathbf{u}\|_{L^{2}},\quad \mathbf{u}\perp \ker({\dev}\,{\sym}\grad ).
\]
Recall that \(\dev\ten w = \ten w - \tfrac{1}{n}\tr(\ten w)\ten I\) is the deviator of a matrix \(w\). In 3D, this means a vector field \(\mathbf{u}\) is controlled by only five components of its gradient matrix.
\end{itemize}

The conformal Korn inequality holds in \(n\)-dimensions only for \(n\ge3\). As discussed in Section~\ref{sec:conformal}, in 2D the operator \({\dev}\,{\sym}\grad \) corresponds to the Cauchy--Riemann operator. For applications of the conformal Korn inequality and an explanation of its failure in 2D, see \cite{dain2006generalized}.

The Poincaré inequality is classical. The Korn inequality plays a foundational role in linear elasticity theory, and several proofs can be found in \cite{ciarlet2013linear}. One such proof uses the Lions lemma (analogous to \eqref{poincare-inequality}, but with an \(L^{2}\)-norm on the left and an \(H^{-1}\)-norm on the right), plus the fact that second derivatives commute (\(\partial_{i}\partial_{j}=\partial_{j}\partial_{i}\)).

Observe that the Poincaré, Korn, and conformal Korn inequalities all arise from the fact that the operators \(\mathcal{D} = \grad,{\sym}\grad,{\dev} {\sym}\grad\) have closed range. If \(\mathcal{D}:\ker(\mathcal{D})^\perp\to \mathrm{ran}(\mathcal{D})\) is a linear operator between Banach spaces and its range is closed, then the Banach isomorphism theorem implies \(\mathcal{D}^{-1}\) is bounded on \(\mathrm{ran}(\mathcal{D})\), giving
\[
\|u\|_{H^{1}}\le C\,\|\mathcal{D}u\|_{L^{2}}
\quad\text{for}\quad
u\in \ker(\mathcal{D})^\perp.
\]
Thus a unified proof for such inequalities boils down to showing closedness of \(\mathrm{ran}(\mathcal{D})\). Moreover, in the context of complexes, if there is exactness, then
\[
\mathrm{ran}(\mathcal{D}) \;=\;\ker(\mathcal{D}^{+}),
\]
for some operator \(\mathcal{D}^{+}\). Since \(\ker(\mathcal{D}^{+})\) is clearly closed, \(\mathrm{ran}(\mathcal{D})\) is also closed. This extends to the case of nontrivial but \emph{finite-dimensional} cohomologies, because finite-dimensional spaces are compact \cite{hormander1994analysis}.

From this perspective, we see:
 \begin{itemize}
\item
 The Poincar\'e inequality is due to $\curl\circ\grad=0$ and $\ran(\grad)=\ker(\curl)$  in the de~Rham complex, reflecting topological constraints.
\item The Korn inequality follows from  $(\curl\circ \mathrm{T}\circ \curl) \circ (\sym\grad) =0$ and $\ran(\sym\grad)=\ker(\curl\circ \mathrm{T}\circ \curl )$, linking it to the elasticity complex and Riemannian geometry.
\item The conformal Korn inequality in 3D is due to
$(\curl\circ S^{-1}\circ \curl\circ S^{-1}\circ \curl) \circ (\dev\sym\grad) =0$ and $\ran(\sym\grad)=\ker(\curl\circ \mathrm{T}\circ \curl )$, connecting it to the Cotton--York tensor and conformal geometry. 
\end{itemize}

Here, failure of the 2D conformal Korn inequality corresponds to the non-closed range of the Cauchy--Riemann operator.

A fix from the perspective of differential complexes is to consider an analogy of the 3D diagram \eqref{diagram:conformal}
\[
\begin{tikzcd}
0 \arrow{r}
&  {H^{q}\otimes \mathbb{V}}
\arrow{r}{\grad}
& {H^{q-1}\otimes \mathbb{M}}
\arrow{r}{\mathrm{rot}}
& H^{q-2}\otimes \mathbb{V}
\arrow{r} & 0
\\
0 \arrow{r}
& H^{q-1}\otimes (\mathbb{R}\times \mathbb{R})
\arrow{r}{\grad}
\arrow[ur, "S^{0,1}"]
& H^{q-2}\otimes (\mathbb{V}\times \mathbb{V})
\arrow{r}{\mathrm{rot}}
\arrow[ur, "S^{1,1}"]
& H^{q-3}\otimes (\mathbb{R}\times \mathbb{R})
\arrow{r} & 0
\\
0 \arrow{r}
& H^{q-2}\otimes \mathbb{V}
\arrow{r}{\grad}
\arrow[ur, "S^{0,2}"]
& {H^{q-3}\otimes \mathbb{M}}
\arrow{r}{\mathrm{rot}}
\arrow[ur, "S^{1,2}"]
&  {H^{q-4}\otimes \mathbb{V}}
\arrow{r} & 0
\end{tikzcd}
\]
\vbox{\vspace{-6.5cm}
\leftline{\hspace{-0.2cm}\vspace{-0.cm}{
 \adjustbox{scale=1.3,center}{%
\begin{tikzpicture}
\draw[very thick, opacity=.75, ->, rounded corners] (3,1.2)  -- ++(1.2, 0.0) -- ++(-1.3,-1.) -- ++(1.2,0)  -- ++(-1.3,-1.) -- ++(1.2, 0.0)  ;
\draw[very thick, opacity=.75, ->, rounded corners] (3,1.4)  -- ++(1.2, 0.0)   ;
\draw[very thick, opacity=.75, ->, rounded corners] (6.3,1.2)  -- ++(1.2, 0.0) -- ++(-1.3,-1.) -- ++(1.2,0)  -- ++(-1.3,-1.) -- ++(1.2, 0.0)  ;
\draw[very thick, opacity=.75, ->, rounded corners] (6.3,-1.)  -- ++(1.1, 0.0)   ;
\draw[very thick,opacity=.75, red] (2.,1.3) ellipse (7mm and 3.5mm);
\draw[very thick,opacity=.75, red] (5.1,1.3) ellipse (7mm and 3.5mm);
\draw[very thick,opacity=.75, red] (5.1,-.7) ellipse (7mm and 3.5mm);
\draw[very thick,opacity=.75, red] (8.1,-.7) ellipse (7mm and 3.5mm);
\end{tikzpicture}}
}
}}
in which we eliminate the skew-symmetric and trace components of the matrix from the two scalars in the first space in the second row. 
The diagram involves \(S^{0,1}=(\iota, \mskw)\) ($\iota$ mapping to the diagonal component and $\mskw$ mapping to the skew-symmetric components; the operator acts on two components and adds up them, i.e., $(\iota, \mskw)(\alpha, \beta)=\iota\alpha+\mskw\beta$) and other operators that complete the commuting diagram. 

Using a similar elimination process as in \ref{sec:conformal}, we see that the space $H^{q-1}\otimes (\mathbb{R}\times \mathbb{R})$ cancels the trace and skew-symmetric parts of $H^{q-1}\otimes \mathbb{M}$, leading to a trace-free and symmetric matrix. The difference from the 3D examples in \ref{sec:conformal} is that the ``1-forms'' of the last row, i.e., ${H^{q-3}\otimes \mathbb{M}}$ cannot be completely eliminated from the diagram (only the trace and skew-symmetric components are eliminated by $H^{q-3}\otimes (\mathbb{R}\times \mathbb{R})$ that $S^{1, 2}$ connects to). Therefore the tracefree and symmetric components of $H^{q-3}\otimes (\mathbb{R}\times \mathbb{R})$ also remain in the derived complex
\[
\begin{tikzcd}
0 \arrow{r}
& H^{q}\otimes \mathbb{V}
\arrow{r}{D^{0}}
& 
\left(
\begin{array}{c}
H^{q-1}\otimes(\mathbb{S}\cap \mathbb{T})\\[4pt]
H^{q-3}\otimes(\mathbb{S}\cap \mathbb{T})
\end{array}
\right)
\arrow{r}{D^{1}}
& H^{q-3}
\arrow{r}
& 0.
\end{tikzcd}
\]
Here \(D^{0}\) maps \(u\) into two components: one via $\dev\deff\grad$ and one via $ \grad S_{\dagger}^{0, 2}\grad S_{\dagger}^{0, 1}\grad $, where $S_{\dagger}^{0, 2}$ and $S_{\dagger}^{0, 1}$ are the adjoint of $S^{0, 2}$ and $S^{0, 1}$ (with respect to the Frobenius norm), respectively. Similarly, ${D}^1$ maps two components to a scalar function: the first by a third order $\rot$ operator and the second by a $\rot$. 

We thus fix the {\it conformal Korn inequality in 2D} by adding a third order derivative term:
 \begin{align*}
\|u\|_{3}\leq C(\|\dev\deff\grad u\|_2&+\| \grad S_{\dagger}^{0, 1}\grad S^{0, 2}_{\dagger}\grad u\|),\\
& \quad \forall u\perp [\ker(\dev\deff\grad)\cap \ker(\grad S_{\dagger}^{0, 1}\grad S_{\dagger}^{0, 2}\grad )].
\end{align*}

In 3D, the Poincaré inequality effectively uses all nine entries of \(\nabla \mathbf{u}\), Korn uses six, and the 3D conformal Korn uses five. An open question is to find the minimal set of linear functionals \(\ell_{i}\) such that the generalised Korn inequality
\[
\|\mathbf{u}\|_{H^{1}}
\le C \Bigl(\sum_{i=1}^{N}\|\ell_{i}(\nabla \mathbf{u})\|_{L^{2}} + \|\mathbf{u}\|_{L^{2}}\Bigr)
\]
holds for some finite \(N\). For further discussion on this topic, see \cite{chipot2021inequalities}.

\subsection{Other applications in linear and nonlinear analysis}

The Poincaré–Korn inequalities illustrate how cohomological structures can lead to analytic results. Many other theorems in linear and nonlinear analysis follow the same principle.

For instance, in linear analysis we encounter \emph{regular decompositions}. Instead of the orthogonal Hodge decompositions for \(L^2\) spaces (as in \eqref{hodge-decomposition-example}), one obtains analogous decompositions in more regular spaces. A typical example is
\[
H(\mathrm{curl},\Omega)=\nabla H^1(\Omega)+\bigl[\,H^1(\Omega)\bigr]^3,
\]
and, more generally, for the de~Rham complex,
\[
H(d,\Omega)=dH^1(\Omega)+\bigl[\,H^1(\Omega)\bigr]^n,
\]
with similar statements holding for other complexes.  

Similarly, the \(\mathrm{div}\)--\(\mathrm{curl}\) lemma from compensated compactness theory essentially relies on the underlying algebraic structures in complexes. Its generalisations to a \(d\)--\(d^*\) lemma appear in \cite{pauly2019global,arnold2021complexes}.

\section{Solid mechanics}\label{sec:solid}

This section focuses on the connections between continuum mechanics (particularly solid mechanics) and differential complexes. Topology interacts with elasticity in several ways. Compatibility conditions (and related concepts such as stress functions) are well known to be closely related to the topological properties of the elastic body. See \cite{yavari2020applications} for a discussion of topological obstructions to compatibility in both linear and nonlinear elasticity. Furthermore, the failure of compatibility conditions can serve as an indication of defects (e.g., dislocations and disclinations) in the material. This idea has been extensively developed in the continuum mechanics literature, from pioneers such as Kröner \cite{kroner1985incompatibility,kroner1981continuum,kroner1963dislocation,kroner1995dislocations} and Nye \cite{nye1953some} to more geometric treatments in \cite{yavari2012riemann,yavari2013riemann}. Nevertheless, we present this topic here from the perspective of differential complexes and the BGG machinery.
In fact, differential complexes encode additional concepts in solid mechanics. Here are some examples.
\begin{itemize} \item \textit{Primal vs.\ dual variational principles.}
In the compatible theory, primal formulations correspond to the Hodge–Laplacian of 0-forms (where “0-forms” refers to the space in the complex at index 0, and similarly for other indices). Dual formulations (e.g., the Hellinger–Reissner principle) correspond to $n$-forms. Intrinsic formulations \cite{ciarlet2009intrinsic} use compatible strain tensors that are closed 1-forms ($\ten e$ satisfying $\inc \ten e=0$). Meanwhile, the Hu–Washizu principle involves three spaces in the complex.

\item \textit{Standard vs.\ microstructure models.}
BGG complexes model classical (displacement-based) theories, whereas \emph{twisted} complexes also include microscopic structures (e.g., local rotations). The BGG machinery introduces a cohomology-preserving projection that removes torsion components to yield standard models (for instance, eliminating pointwise rotations from Cosserat models to recover linear elasticity).

\item \textit{Rigidity and existence.}
Rigidity and existence (sometimes known as the fundamental theorem of Riemannian geometry \cite{ciarlet2013linear}) can be phrased in terms of the exactness of complexes. This aligns with the broader perspective discussed in the Introduction, where existence and uniqueness of solutions are viewed via homological methods. Several existence and uniqueness results in geometry and mechanics can be formulated using complexes and their exactness. Formally, consider an exact sequence
$$
\begin{tikzcd}
R \arrow{r}{f_{0}}&A \arrow{r}{f_{1}}& B   \arrow{r}{f_{2}}&C,
\end{tikzcd}
$$
any object $u:=f_{1}(A)$ satisfies a {\it compatibility} condition $f_{2}(u)=0$;  for any compatible object $v$ satisfying  $f_{2}(v)=0$, there exists an object (stress function, potential) $w\in A$ such that $f_{1}(w)=v$, and $w$ is unique up to $f_{0}$ of objects in $R$ (rigidity).

\item \textit{Dimension reduction and (boundary) trace complexes.}
Trace complexes correspond to lower-dimensional (e.g., plate or beam) models. Hence, dimension reduction can be incorporated into the BGG framework. From an analytic perspective, such models often arise as 
$\Gamma$-limits of higher-dimensional theories (see \eqref{diagram:models} below). 

\item \textit{Multi-dimensional or mixed-dimensional models.}
Bodies of different dimensions can be coupled in ways that incorporate richer physics and additional interface variables. This leads to the construction of \v{C}ech double complexes \cite{boon2023mixed,boon2022hodge}. \end{itemize}

These connections not only provide new perspectives on existing models but also indicate the potential for investigating new ones. For instance, to construct a 2D Cosserat-type plate model with defects, one could consider the Hodge–Laplacian acting on 1-forms (\emph{defects}) in the trace complex (\emph{2D}) of the twisted complex (\emph{Cosserat}) on a surface (\emph{plates}). By selecting different Lie structures and complexes, one might obtain new continuum models with naturally encoded symmetries. This recalls the \textit{Erlangen programme}, wherein geometries are characterized by particular Lie structures.



\subsection{Classical elasticity and compatible continuum: a nonlinear complex} 

The (linear) elasticity complex \eqref{elasticity} can be viewed as the linearisation of a sequence involving nonlinear operators:
\begin{equation}\label{nonlinear-elasticity}
\begin{tikzcd}
0 \arrow{r}{} &
\text{displacement} \arrow{r}{\mathscr{D}^{0}} &
\text{strain}\arrow{r}{\mathscr{D}^{1}} &
\text{stress} \arrow{r}{\mathscr{D}^{2}} &
\text{load} \arrow{r}{} &
0.
\end{tikzcd}
\end{equation}
As we explain below, this sequence is also a complex in the sense that composing two consecutive  operators yields zero. Although cohomology is not defined for complexes with nonlinear operators (as the zero sets and the image of the nonlinear operators are not vector spaces), the notion of exactness still applies. Next, we explain \eqref{nonlinear-elasticity} in detail. This will also set up the background used in the rest of this section. 

In continuum theories, we consider a body initially occupying a region
\(\widehat{\Omega}\subset \mathbb{R}^{3}\), called the \emph{initial configuration}.
The body is deformed to a new region \(\Omega \subset \mathbb{R}^{3}\), called the \emph{current configuration}.
A point \(\widehat{\vec{x}} \in \widehat{\Omega}\) is mapped to \(\vec{x}\in \Omega\) via a diffeomorphism \(\vec{\varphi}\), i.e.\ \(\vec{x}:=\vec{\varphi}(\widehat{\vec{x}})\).
The \emph{displacement} is the vector from the original to the new point,
\[
\vec{u}(\widehat{\vec{x}})=\vec{\varphi}(\widehat{\vec{x}})-\widehat{\vec{x}}
=
(\vec{\varphi}-\ten{I})(\widehat{\vec{x}}).
\]
The \emph{deformation gradient} is
\[
\ten{F}=\widehat{\nabla}\vec{\varphi}
=\frac{\partial \vec{\varphi}}{\partial \widehat{\vec{x}}}.
\]
We assume the initial configuration is equipped with the Euclidean metric \(\ten{g}_{0} := \ten{I}\).
Then the deformation induces the new metric
\[
\ten{G}
:=
(\widehat{\nabla}\vec{\varphi}) \cdot \ten{I}\cdot (\vec{\varphi}\,\widehat{\nabla})
=
(\widehat{\nabla}\vec{\varphi})\cdot(\vec{\varphi}\,\widehat{\nabla}),
\]
known as the (right) Cauchy–Green deformation tensor.  
Here, \(\widehat{\nabla}\vec{\varphi}\) denotes the \emph{column-wise} gradient (the gradient operator acts on each component of \(\vec{\varphi}\) to form columns), while \(\vec{\varphi}\,\widehat{\nabla}\) denotes the {row-wise} gradient.  
The tensor \(\ten{G}\) measures distances on the deformed body: if two points are separated by \(\Delta \widehat{\vec{x}}\) before deformation, their initial distance squared is \(\Delta \widehat{\vec{x}}^{T}\cdot \ten{I} \cdot \Delta \widehat{\vec{x}}\), while their post-deformation distance squared is \(\Delta \widehat{\vec{x}}^{T}\cdot \ten{G}\cdot \Delta \widehat{\vec{x}}\).  
In differential-geometric terms, \(\ten{G} = (\vec{\varphi}^{-1})^{\ast}\ten{I}\).  
We define the \emph{change of metric} or \emph{strain} as
\[
\ten{e}
=
(\widehat{\nabla}\vec{\varphi})\cdot(\vec{\varphi}\,\widehat{\nabla})
-\ten{I},
\]
so that \(\ten{e}\) encodes how the Euclidean metric changes under deformation.
\begin{figure}
\includegraphics[width=0.6\textwidth]{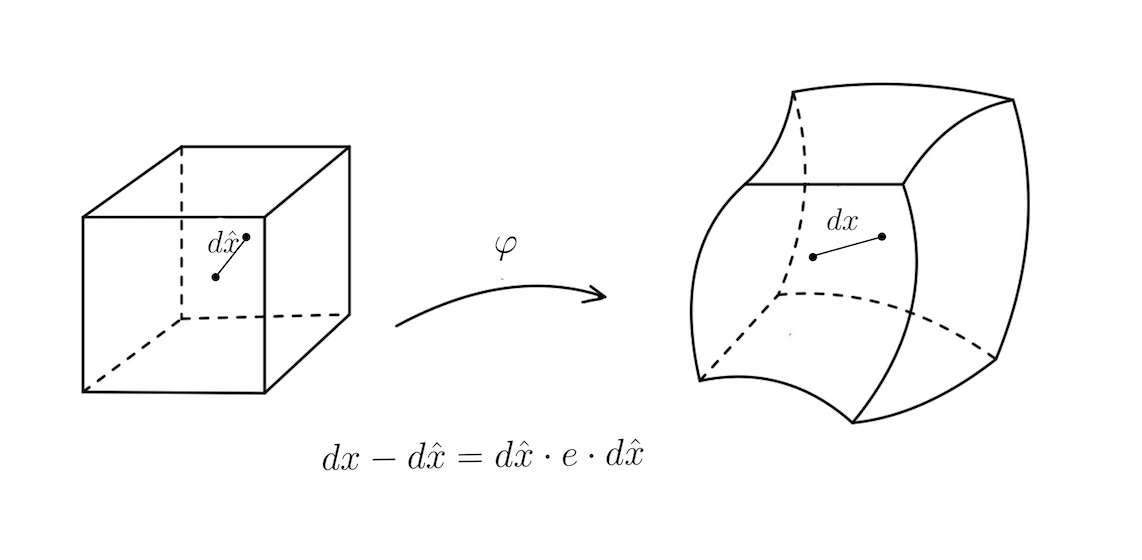}
\caption{Embedding (deformation) of a body induces a change of metric, i.e., strain. }
\end{figure}

We can also express \(\ten{G}\) directly in terms of \(\vec{u}\).  Setting \(\vec{\varphi}(\widehat{\vec{x}}) = \widehat{\vec{x}} + \vec{u}(\widehat{\vec{x}})\), we introduce
\[
\mathscr{D}^{0}(\vec{u})
=
\bigl[\widehat{\nabla}(\vec{u}(\widehat{\vec{x}})+\widehat{\vec{x}})\bigr]
\,\cdot\,
\bigl[(\vec{u}(\widehat{\vec{x}})+\widehat{\vec{x}})\,\widehat{\nabla}\bigr]
-\ten{I}=\widehat{\nabla}\vec{u}\cdot \vec{u}\widehat{\nabla}+\widehat{\nabla}\vec{u}+ \vec{u}\widehat{\nabla}.
\]
Under a \emph{small deformation}, i.e.\ \(\vec{\varphi}(\widehat{\vec{x}})
= \widehat{\vec{x}} + \varepsilon\,\vec{v}(\widehat{\vec{x}}) + \mathfrak{o}(\varepsilon)\),
we obtain
$$
(\widehat{\nabla}\vec{\varphi})\cdot(\vec{\varphi}\,\widehat{\nabla})-\ten I=\varepsilon(\widehat{\nabla}\vec{v} + \vec{v}\,\widehat{\nabla})+{o}(\varepsilon).
$$
Hence, the first-order term
\[
\mathscr{D}^{0}_{\mathrm{lin}}(\vec{u})
=
\widehat{\nabla}\vec{u} +\vec{u}\,\widehat{\nabla}
\]
is the standard \emph{linearised strain}.  In the small-deformation regime, we usually identify \(\widehat{\nabla}\) with \(\nabla\) and write \(\nabla\vec{u}+\vec{u}\nabla\).

The next operator \(\mathscr{D}^{1}\) sends a (nonlinear) strain \(\ten{e}\) (i.e., a metric difference) to its associated \emph{Einstein tensor}.\footnote{%
  In 3D, vanishing of the Einstein tensor implies vanishing of the Riemann curvature, hence “flatness.”
}
Its linearisation is \(\mathscr{D}^{1}_{\mathrm{lin}} = 2\,\inc\), where \(\inc\) is the \emph{incompatibility operator} (sometimes called the Saint–Venant operator).  Given \(\ten{e} = (\widehat{\nabla}\vec{\varphi})\cdot(\vec{\varphi}\,\widehat{\nabla}) - \ten{I}\), we have
\(\mathscr{D}^{1}(\ten{e})=0\) precisely because the pullback of a Euclidean metric is flat.  Thus,
\(\mathscr{D}^{1}(\ten{e})=0\) is the \emph{compatibility condition} for strain.  In the small-strain limit, \(\mathscr{D}^{1}_{\mathrm{lin}}(\ten{e})=0\) becomes the well-known \emph{Saint–Venant compatibility condition}.

\medskip

In elasticity, the Cauchy stress \(\ten{\sigma}\) is a symmetric matrix field whose action \(\ten{\sigma}\cdot \vec{n}_{\partial \mathcal{O}}\) on the boundary of a microelement \(\mathcal{O}\) gives the internal force.  Physically, \(\ten{\sigma}\) can be determined from \(\ten{e}\) via a constitutive law (e.g., Hooke’s law).  Mathematically, stress and strain are dual to each other via the energy inner product \(\int \ten{\sigma} : \ten{e}\,d{x}\).  Here, the divergence \(\nabla \cdot \ten{\sigma}\) corresponds to the force; hence, \(\nabla \cdot \ten{\sigma}=0\) describes force-free fields.  A \emph{stress function} (such as the Beltrami stress function) provides a potential for stress fields satisfying equilibrium.  In 3D, one classically has \(\ten{\sigma} = \nabla \times \ten{\varphi} \times \nabla\) for some symmetric tensor \(\ten{\varphi}\).  This is one of the exactness properties of the elasticity complex \eqref{elasticity}.

\smallskip

Classical elasticity admits several variational formulations: the \emph{displacement method} (using 0-forms) minimising the energy functional
$$
\inf_{\vec u\in [H^{1}]^{n}}\frac{1}{2}\|\deff \vec u\|_{A}-(\vec f, \vec u), 
$$
 the \emph{Hellinger–Reissner principle} seeking critical point $\ten \sigma\in H(\div, \mathbb{S})$, $\vec u\in [L^{2}]^{n}$ of the saddle point functional
 $$
\inf_{\vec u\in [H^{1}]^{n}}\sup_{\ten \sigma\in H(\div, \mathbb{S})}\frac{1}{2}\|\ten \sigma\|^{2}_{C^{-1}}+(\div\ten \sigma, \vec u)-(\vec f, \vec u), 
$$
  and the \emph{intrinsic elasticity model} minimising the energy functional
  $$
\inf_{\ten e\in H(\inc, \mathbb{S}): \inc \ten e=0}\frac{1}{2}\|\ten e\|_{A}-(\vec f, \mathcal{F}(\ten e)), 
$$
  where $\mathcal{F}(\ten e)$ is the Cec\`aro-Volterra operator satisfying $\deff\mathcal{F}(\ten e)=\ten e$. The \emph{Hu–Washizu principle} introduces a three-field variational formulation in which displacement $\vec u$, strain $\ten e$, and stress $\ten \sigma$ (0-, 1-, and 2-forms) are treated as independent fields. A typical saddle-point form is
   $$
\inf_{\vec u\in [H^{1}]^{n}, \ten e\in L^{2}\otimes \mathbb{S}}\sup_{\ten \sigma\in H(\div, \mathbb{S})}\frac{1}{2}\|\ten e\|^{2}_{C}-(\ten \sigma, \ten e)+(\div\ten \sigma, \vec u)-(\vec f, \vec u), 
$$
  

\smallskip

As we have seen, many key notions in elasticity, such as displacement, strain, stress, equilibrium, compatibility, fit naturally into the (nonlinear) diagram \eqref{nonlinear-elasticity} and its linearised version \eqref{elasticity}.  In the discussion above, we highlighted the first part (displacement and strain) and the last part (stress and load), noting the duality between strain and stress.  However, we have not yet discussed \(\mathscr{D}^{1}\) (or its linearised analog \(\inc\)) beyond observing the compatibility condition \(\inc\,\ten{e}=0\).  In the next subsection, we turn to the idea that violations of compatibility can model \emph{continuum defects}, so that the strain and stress spaces in \eqref{nonlinear-elasticity} (possibly not coming from a displacement) represent defect fields in a natural way.

\medskip
{\noindent\bf Exactness of the nonlinear complex.}
Although cohomology is not defined for the nonlinear complex, we can also talk about exactness, i.e., the zero set of a map is the image of the previous operator. 

The nonlinear elasticity complex in $\mathbb{R}^{3}$ \eqref{nonlinear-elasticity} can be reformulated as follows using geometric terms:
$$
\begin{tikzcd}[column sep=1.4cm]
\mbox{rigid body motion} \arrow{r}{\subset} & \mbox{map $\mathbb{R}^{3}$ to $ \mathbb{R}^{3}$} \arrow{r}{\varphi\mapsto \varphi^{\ast}g_{0}-g_{0}} & \mbox{change of metric} \arrow{r}{\mathrm{Ricci}} & \mbox{curvature}  \arrow{r}{\mathrm{Bianchi}} & \mathrm{(co)vector}
\end{tikzcd}
$$
The exactness of the nonlinear complex corresponds to various classical results \cite{hu2022nonlinear}.
 The exactness at ``map from $\mathbb{R}^{3}$ to $\mathbb{R}^{3}$'' corresponds to the rigidity result: If a map does not change the Euclidean metric, then it is a rigid body motion. The exactness at the ``change of metric'' corresponds to the ``fundamental theorem of Riemannian geometry'' \cite{ciarlet2013linear}: If a metric is flat, then it can be locally written as (a pullback of) the Euclidean metric. The exactness at curvature (a tempting question -- whether a tensor is the curvature of some metric if it satisfies certain Bianchi identities) is not well defined, as to formulate the Bianchi identity, one already needs a metric. See \cite{rendall1989insufficiency} for detailed discussions.

\subsection{Higher-order forms: defects and incompatibility}\label{sec:defects}

Defects in materials refer to imperfections that disrupt the regular arrangement of atoms or molecules within a solid. These defects can be classified into several types, including point defects, dislocations, and disclinations, each affecting the material’s properties in different ways. Point defects, such as vacancies and interstitials, occur when atoms are missing or misplaced in the crystal lattice, leading to changes in material properties like conductivity or diffusion. Dislocations are line defects that represent the misalignment of atomic planes and are central to plastic deformation, influencing a material’s strength. Disclinations, on the other hand, are rotational defects that occur in the lattice structure. 
These defects are crucial in understanding the mechanical, electrical, and thermal behaviours of materials, as well as their response to external stresses, and they play a significant role in the design of advanced materials with tailored properties. In a continuum theory, defects are often described by the violation of compatibility conditions. There is a vast literature on this topic, see, e.g., \cite{kroner1985incompatibility,anderson2017theory,kroner1981continuum,steinmann2015geometrical,yavari2013riemann,yavari2012riemann}. We by no means provide a comprehensive review, but we aim to show that (BGG) complexes precisely encode this idea of modelling and the algebraic structures therein.


 

In a continuum theory, defects are often described by the violation of compatibility conditions. We first summarize some key concepts and relations in classical elasticity and continuum defect theory in the diagrams below:
  \begin{equation}\label{twisted-diagram}
\begin{tikzcd}
 \mbox{displacement} \arrow{r}{\grad} &\mbox{distortion}\arrow{r}{\curl} &\mbox{dislocation density tensor} \arrow{r}{\div}  &\cdots \\
 \mbox{rotation}\arrow{r}{\grad} \arrow[ur, "-\mskw"]&\mbox{wryness tensor, Nye tensor} \arrow{r}{\curl} \arrow[ur, "\mathcal{S}"]&\mbox{disclination density, stress} \arrow{r}{\div}\arrow[ur, "2\vskw"] &\mbox{load},
 \end{tikzcd}
\end{equation}
  \begin{equation}\label{elasticity-complex-variables}
\begin{tikzcd}
 \mbox{displacement}  \arrow{r}{\deff} & \mbox{strain} \arrow{r}{\inc} & \mbox{internal stress, stress}    \arrow{r}{\div} & \mbox{load}. 
 \end{tikzcd}
\end{equation}
The notations are summarized in the following diagrams:
  \begin{equation}\label{elasticity-diagram-variables}
\begin{tikzcd}
 \vec{u} \arrow{r}{\grad} &\ten  \beta ~ (\ten  e=\ten  \beta-\ten  \omega)\arrow{r}{\curl} &\ten \alpha \arrow{r}{\div}  &\cdots \\
  \vec{\varpi}\arrow{r}{\grad} \arrow[ur, "-\mskw"]& \ten  \kappa \arrow{r}{\curl} \arrow[ur, "\mathcal{S}"]& \ten \theta, \ten  \sigma \arrow{r}{\div}\arrow[ur, "2\vskw"] &\vec{f},
 \end{tikzcd}
\end{equation}
  \begin{equation}\label{elasticity-complex-variables}
\begin{tikzcd}
 \vec{u} \arrow{r}{\deff} &\ten  e\arrow{r}{\inc} &\ten \eta, \ten  \sigma   \arrow{r}{\div} & \vec{f}. 
 \end{tikzcd}
\end{equation}

Now we explain the diagrams in the linearised theory. We refer to \cite{steinmann2015geometrical,yavari2013riemann} for the definitions below.

 Let $\vec\varphi: \Omega\to \vec\varphi (\Omega)$ be an isomorphism with $\Omega\subset \mathbb{R}^{3}$, and  $\vec{u}(\vec x):=\vec\varphi(\vec x)- \vec x$ be the displacement of a continua. In the classical elasticity theory, the \emph{distortion tensor} $ \ten\beta$ is defined as
$
\ten  \beta:=\vec{u}\nabla,
$
satisfying the compatibility condition $\ten \beta\times \nabla=0$. 
The {distortion tensor} can be decomposed into a symmetric part $\ten  e$ and a skew symmetric part $\ten  \omega$, i.e., $\ten  \beta=\ten  e+\ten \omega$.
Here $\ten  e$ is the classical \emph{strain tensor} and $\ten \omega$ is the \emph{rotation tensor}. Specifically, we have
$$
\ten  e=\deff \vec{u}:= \frac{1}{2}(\nabla\vec{u}+\vec{u}\nabla),\quad \mbox{and} \quad \ten  \omega=\skw( \vec{u}\nabla)= \frac{1}{2}(\vec{u}\nabla-\nabla\vec{u}).
$$
The rotation tensor  $\ten  \omega$ can be obtained as the matrix version of the axial vector $\vec{u}\times \nabla$, i.e., 
$
\ten \omega= \mskw (\vec{u}\times \nabla).
$
We denote the axial vector of $\ten \omega$ by $
\vec{\varpi}:=\vskw{\ten \omega}:=(\mskw)^{-1}\ten \omega=\vec{u}\times \nabla.
$ 
  The \emph{wryness tensors} $\ten \Xi$ (the name is often used in the context of Cosserat and micropolar continuum, see, e.g., \cite{pietraszkiewicz2009natural}) is defined as 
\begin{equation}\label{def:Xi}
\ten  \kappa:=\vec{\varpi}\nabla=\nabla\times \vec{u}\nabla,
\end{equation}
 and satisfies the compatibility condition 
\begin{align}\label{compatibility-Xi}
\ten  \kappa\times \nabla=0.
\end{align}
 The wryness tensors tensor is trace-free: $\tr \ten  \kappa=0$, as $\tr(\nabla\times \vec{u}\nabla)=\nabla\cdot \nabla\times \vec{u}=0$.

By definition, we have $\nabla\times \ten {e}=\frac{1}{2}\ten  \kappa$. With this identity, the Saint-Venant compatibility condition for $\ten  e$, $\nabla\times \ten e\times \nabla=0$, is equivalent to the compatibility condition \eqref{compatibility-Xi} for $\ten  \kappa$. 

The above definitions in the context of classical elasticity describe a compatible deformation. When defects are present, some compatibility conditions may be violated. The general philosophy in defect theory is to use the violation of compatibility conditions to model the defects. The violation of compatibility conditions, viewed as defect (dislocation, disclination etc.) density, is thus encoded in higher-order forms in a differential complex.
We review this idea in the rest of this section.

Dislocations are detected by {\it Burgers circuits}: one draws a closed circuit in the reference configuration. If the deformation is compatible, 
then the deformed circuit remains closed. However, this is not necessarily the case if dislocations are present. The gap is called the {\it Burgers vector}. On the mesoscopic level, the Burgers circuit consists of a loop of atoms, and the gap (the Burgers vector) describes the misplacements of the atoms. 

In the presence of defects, it is assumed that the distortion tensor $\ten \beta$ is decomposed into an elastic part plus a plastic part: $\ten \beta=\ten\beta^{e}+\ten\beta^{p}$ (additive for small deformation). The elastic part would vanish if all loads were removed. It stores elastic energy and is governed by some elastic stress--strain law (e.g., Hookean or hyperelastic). The plastic part reflects permanent molecular-level rearrangements -- slip, twinning, dislocation motion in crystals, etc. Removing the load does not necessarily make $\ten\beta^{p}$ vanish. 
This portion evolves according to plastic flow rules (yield conditions, hardening laws, etc.).

Under the above decomposition, $\ten\beta$ still satisfies the compatibility condition $\ten \beta\times\nabla=0$. However, the elasticity part  $\ten \beta^{e}\times\nabla=-\ten \beta^{p}\times\nabla$ may be nonzero. 
The Burgers vector is now obtained as 
$$
\vec b=\int_{\mathcal{C}} \ten \beta^{p}\, d\vec x=-\int_{\mathcal{C}} \ten \beta^{e}\, d\vec x=\int_{\mathcal{S}} (\ten \beta^{p}\times \nabla)\, d\ten S,
$$
where $\mathcal{S}$ is an area enclosed by $\mathcal{C}$.  The \emph{dislocation density tensor} is then defined by 
$$
\ten  \alpha:=\ten \beta^{p}\times \nabla=-\ten \beta^{e}\times \nabla,
$$
which represents the distribution of dislocations in the material.
Clearly, the dislocation density is divergence-free (as a source-free field): $\ten \alpha\cdot\nabla=0$. This corresponds to the fact that dislocations do not have a net source. 
Then the \emph{incompatibility tensor} $\ten  \eta$ involved in the Saint-Venant condition is
$$
\ten  \eta= \nabla\times\ten  \alpha.
$$
Note that $\ten \eta$ is automatically symmetric as $\inc$ maps any skew-symmetric matrix to zero.


Furthermore, $\ten\omega$ decomposes into elastic and plastic parts $\ten\omega=\ten \omega^{e}+\ten\omega^{p}$, with $\ten \omega^{e}=\skw \ten\beta^{e}$. For a macroscopically stress-free configuration, the elastic distortion consists of only the elastic rotation, i.e., $\ten \beta^{e}=\ten \omega^{e}$, because the symmetric part, i.e., the linearised strain $\sym\ten\beta^{e}$ is proportional to the elastic stress, which vanishes by assumption. Thus in the case of stress-free configuration,
 the Nye tensor \cite{nye1953some}
$\ten \kappa^{e}:=\vec \varpi^{e}\nabla$ (here $\vec \varpi^{e}:=\vskw\ten \omega^{e}$ is the axial vector of the elastic rotation tensor $\ten \omega^{e}$) and the dislocation density $\vec \alpha=\ten \beta^{e}\times \nabla$ satisfies the Nye formula: $\ten \alpha=\ten (\ten\kappa^{e})^{T}-\tr(\ten \kappa^{e})I$ . This is exactly a commuting identity of the diagram \eqref{elasticity-diagram-variables}.

The discussions above complies with the general philosophy of using the violation of compatibility conditions to model defects. This idea can be continued, leading to high-order continuum defect theories.  In particular, $\ten \kappa$ obtained by $\vec u$ via \eqref{def:Xi} automatically satisfies the compatibility $\ten \kappa \times \nabla=0$. We can further assume that we are given $\ten \kappa$ that is not derived from $\vec u$ or $\vec \varpi$ and does not necessarily satisfy the compatibility condition.

Disclinations describe rotational mismatches in a crystalline or continuum structure, leading to angular closure failures. In disclination theory, the Frank vector measures this rotational closure failure: if we draw a loop in the continuum and parallel transport a vector along it back to the origin, that vector may be rotated. The change in angle is captured by the Frank vector, which aligns exactly with the geometric definition of Riemannian curvature. Consequently, curvature provides a natural geometric interpretation for disclinations.

In compatible theories, $\ten \kappa:= \vec w\nabla$ satisfies the compatibility condition $\ten \kappa\times \nabla=0$. In generalised continuum, $\ten \kappa=\ten \kappa^{e}+\ten \kappa^{p}$ has an elastic-plastic decomposition, and $\ten\theta:=\ten \kappa^{p}\times \nabla$ is defined as the disclination (second-order) tensor \cite[Remark 1.7]{steinmann2015geometrical}.  This exactly corresponds to the curvature tensor in \eqref{elasticity-diagram-variables}. 
}

\subsection{Twisted complexes: microstructure and micropolar models}

In 1D, 2D and 3D, the Hodge--Laplacian problems of the twisted de~Rham and BGG complexes also correspond to various beam, plate and elasticity models. As we shall see, the twisted complex encode extra degrees of freedom (unknowns) characterising microstructure of materials. The BGG machinery for deriving the BGG complexes from the twisted complexes thus can be interpreted as eliminating these degrees of freedom. 

An often-used assumption in beam models is that a line orthogonal to the midline remains orthogonal after deformation (see Figure \ref{fig:beam}).
\begin{figure}
 \includegraphics[width=7cm]{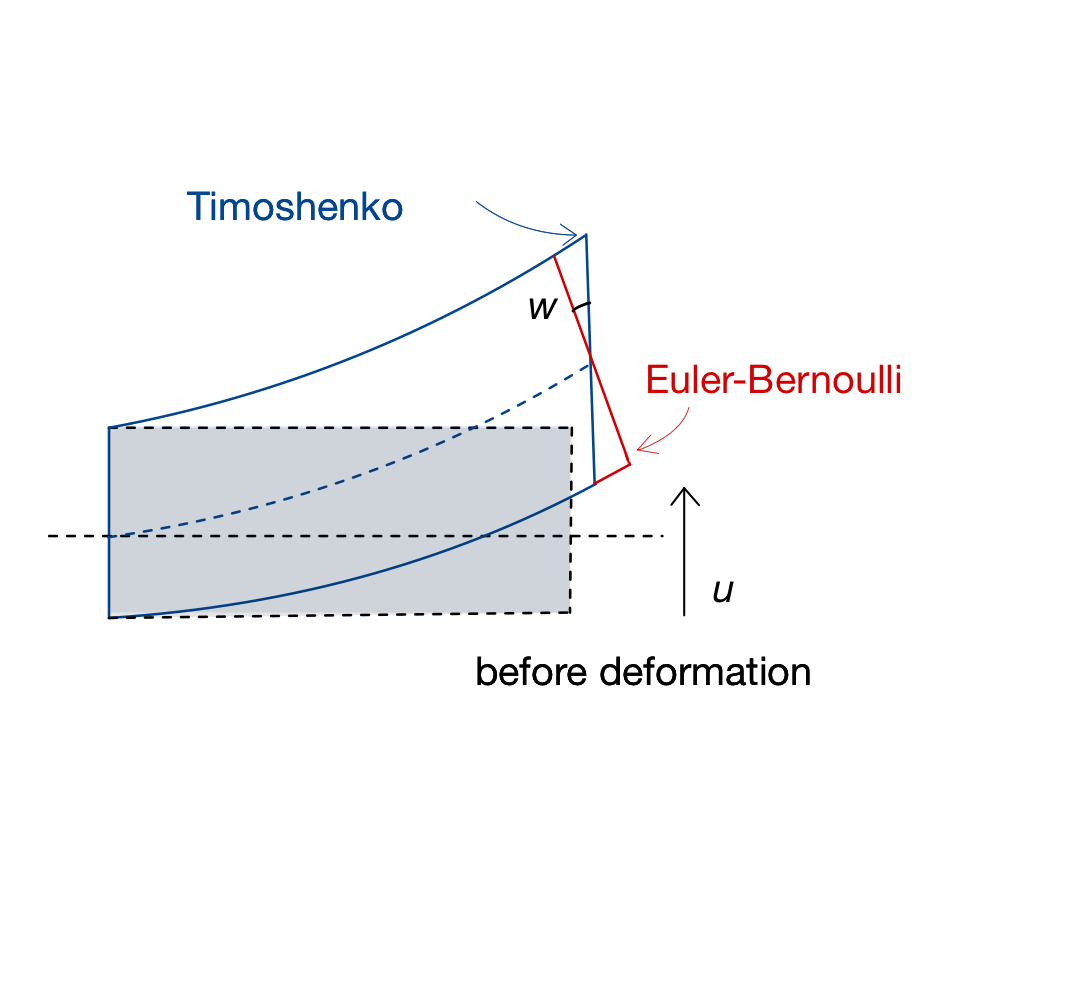}
 \vspace{-2cm}
 \caption{Beam models. In the Euler--Bernoulli beam, a line orthogonal to the mid-surface remains orthogonal after the deformation. Therefore, the vertical displacement $u$ is the only unknown. In the Timoshenko beam, the orthogonality condition is dropped. Therefore, the model involves an additional rotational variable $w$.}
  \label{fig:beam}
\end{figure}
 Under this assumption, the beam can be described by a single unknown, i.e., the vertical displacement. This leads to the Euler--Bernoulli beam theory, where the solution is obtained by minimising the energy functional 
\(\|\tfrac{d^2}{dx^2} w\|_{C}^{2}\), where $\|\cdot\|_{C}$ is a weighted $L^{2}$ inner product involving physical parameters of the material.
In the language of complexes, this energy functional corresponds to the 0th Hodge--Laplacian of the BGG complex in 1D:
\[
\begin{tikzcd}
0 \arrow{r} &H^{2}  \arrow{r}{\frac{d^{2}}{dx^{2}}} &L^{2} \arrow{r}{} & 0.
\end{tikzcd}
\]

If we drop the mid-surface orthogonality assumption, we introduce one more variable \(w\), representing the rotation angle. This yields the Timoshenko beam energy functional
\begin{equation}\label{eqn:timoshenko-energy}
\mu_{c}\bigl\|\tfrac{d}{dx} u - w\bigr\|_{C_{1}}^{2} + \bigl\|\tfrac{d}{dx}w\bigr\|_{C_{2}}^{2},
\end{equation}
which corresponds to the 0th Hodge--Laplacian of the twisted de~Rham complex in 1D:
\[
\begin{tikzcd}[ampersand replacement=\&]
0\arrow{r}\&
\begin{pmatrix}
H^{1} \\
H^{1}
\end{pmatrix}
 \arrow{r}{
 \begin{pmatrix} \frac{d}{dx} & -I \\ 0 & \frac{d}{dx} \end{pmatrix}
 }\&  
\begin{pmatrix}
L^{2}\\
L^{2}
\end{pmatrix} 
\arrow{r}{} 
\&0.
\end{tikzcd}
\]
Here, \(\mu_{c}\) is a constant characterizing the coupling between vertical displacement and rotation. When \(\mu_{c} = 0\), the system decouples; when \(\mu_{c} \to \infty\), we enforce \(w = \tfrac{d}{dx}\,u\), which reduces to the Euler--Bernoulli beam.

\medskip

A similar situation occurs in 2D. The Hessian complex
\begin{equation}\label{hessian-2D}
\begin{tikzcd}
0 \arrow{r}
 & H^{2}
 \arrow{r}{\hess}
 & H(\rot; \mathbb{S}) 
 \arrow{r}{\rot} 
 & L^{2}\otimes \mathbb{V}
 \arrow{r}{} 
 & 0
\end{tikzcd}
\end{equation}
and its twisted de~Rham counterpart
\begin{equation}\label{twisted-elasticity-2D}
\begin{tikzcd}
0\arrow{r}&
\begin{pmatrix}
H^{1} \\
H^{1}\otimes \mathbb{V}
\end{pmatrix}
 \arrow{r}{d_V^{0}}
 &  
\begin{pmatrix}
H(\rot) \\
H(\rot)\otimes \mathbb{V}
\end{pmatrix} 
 \arrow{r}{d_V^{1}}
 &   
\begin{pmatrix}
L^{2}\\
L^{2}\otimes \mathbb{V}
\end{pmatrix}
\arrow{r}{} &0,
\end{tikzcd}
\end{equation}
where
\[
d_V^{0} = \begin{pmatrix}
\grad & -I\\[6pt]
0 & \grad
\end{pmatrix},
\quad
d_V^{1} = \begin{pmatrix}
\rot & -\operatorname{sskw}\\[6pt]
0 & \rot
\end{pmatrix},
\]
respectively correspond to the Kirchhoff plate model and a Reissner--Mindlin plate model (see, e.g., \cite[p.~1279]{arnold1989uniformly}). The former's energy functional (the 0th Hodge--Laplacian) is \(\|\hess u\|_{B}^{2}\), while the latter is 
\(\mu_{c}\|\grad u - \vec{\phi}\|_{A}^{2} + \|\nabla \vec{\phi}\|_{B}^{2}\).
Again, \(\mu_{c}=0\) decouples \(u\) and \(\vec{\phi}\); as \(\mu_{c}\to \infty,\) we force \(\vec{\phi} = \nabla u\), and the Reissner--Mindlin solution converges to that of the Kirchhoff plate.
The Kirchhoff plate model can be also obtained from the Reissner--Mindlin plate as thickness tends to zero \cite{arnold2002range}.

Cosserat (or micropolar) models, proposed by the Cosserat brothers in the early 20th century~\cite{cosserat1909theorie}, extend classical continuum mechanics by allowing material points to possess orientational degrees of freedom in addition to translational ones. Thus, Cosserat models capture size effects and rotational phenomena in granular or otherwise heterogeneous media~\cite{vardoulakis2019cosserat}, whereas classical elasticity treats geometrically similar structures as having identical mechanical responses. The Cosserat model also inspired the notion of torsion in geometric contexts~\cite{scholz2019cartan}.
\begin{figure}
\includegraphics[width=0.6\textwidth]{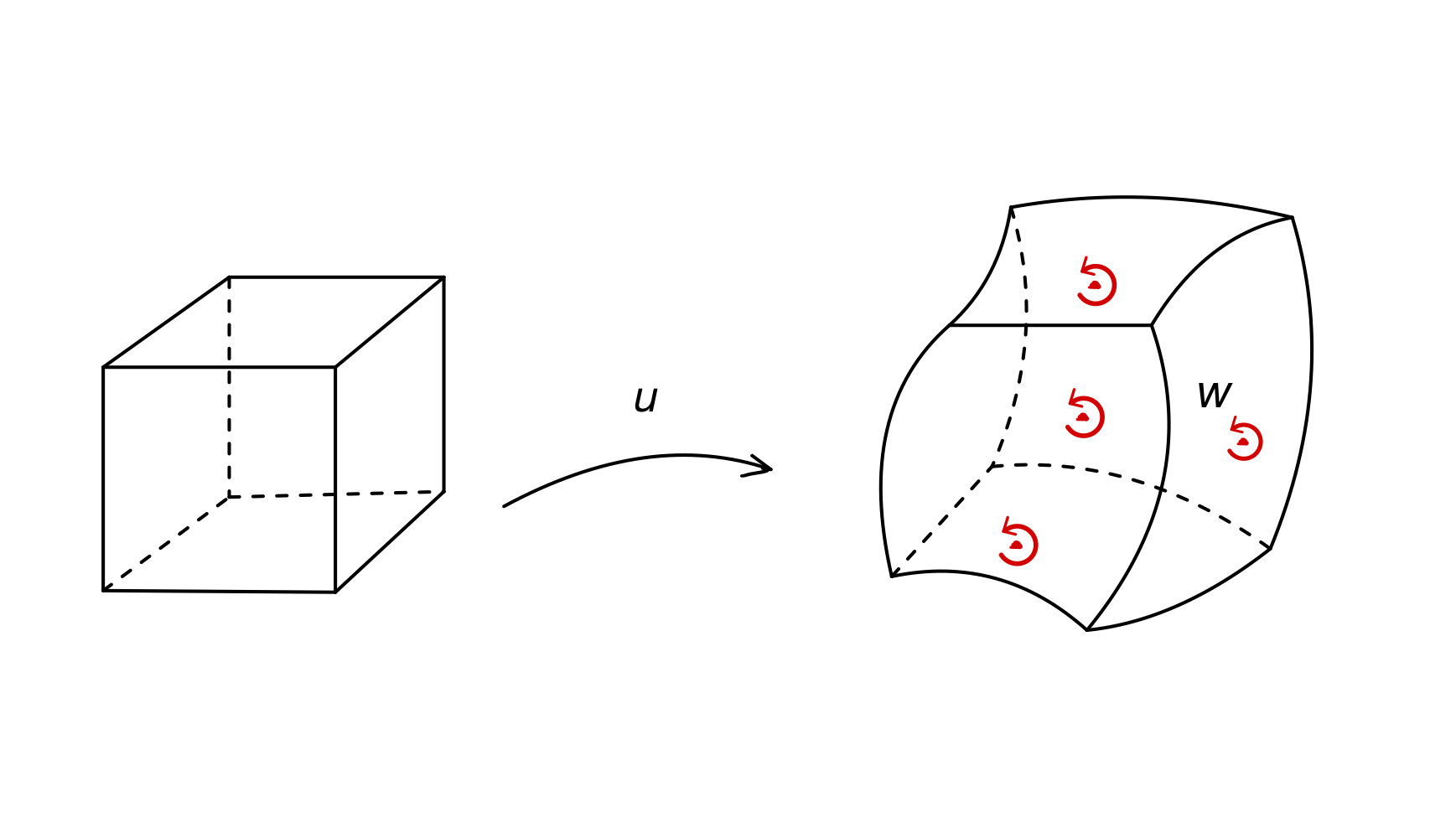}
\caption{Cosserat models introduce a pointwise rotation to continuum models. }
\end{figure}



\medskip

Next, we present a unified view of the linearised versions of these models via the twisted and BGG complexes. The linearised Cosserat elasticity model has the following energy functional, where \(\vec{u}\) is the displacement and \(\vec{\omega}\) is the pointwise rotation (treated as an axial vector):
\begin{flalign}\label{Cosserat-energy}
\begin{split}
\mathcal{E}^{\mathrm{Cosserat}}(\vec{u}, \vec{\omega})
&:= \int_{\Omega}\Bigl(\tfrac{1}{2}\|\grad \vec{u}-\mskw\,\vec{\omega}\|_{C_{1}}^{2}
+\tfrac{1}{2}\|\grad \vec{\omega}\|_{C_{2}}^{2}
-\langle \vec{f}_{\vec{u}}, \vec{u}\rangle
-\langle \vec{f}_{\vec{\omega}}, \vec{\omega}\rangle\Bigr)\,dx\\
&\quad=\int_{\Omega}\bigl(\tfrac{1}{2}\|\sym\grad \vec{u}\|_{\mathcal{C}}^{2}
+\mu_{c}\bigl\|\tfrac{1}{2}\curl \vec{u}-\vec{\omega}\bigr\|^{2}
+\tfrac{\gamma+\beta}{2}\|\sym\grad \vec{\omega}\|^{2}\\
&\qquad\qquad
+\tfrac{\gamma-\beta}{4}\|\curl \vec{\omega}\|^{2}
+\tfrac{\alpha}{2}\|\div \vec{\omega}\|^{2}\bigr)\,dx
-\int_{\Omega}\bigl(\langle \vec{f}_{\vec{u}},\vec{u}\rangle 
+\langle \vec{f}_{\vec{\omega}}, \vec{\omega}\rangle\bigr)\,dx,
\end{split}
\end{flalign}
with
\[
C_1(\varepsilon)=\mathcal{C}(\varepsilon)\,+\,\mu_c\,\skw(\varepsilon),  
\quad
\mathcal{C}(\varepsilon)=2\mu\,\sym(\varepsilon)+\lambda\,\mathrm{tr}(\varepsilon)\,I,
\]
\[
C_2(\varepsilon)=(\gamma+\beta)\sym(\varepsilon)+\alpha\,\mathrm{tr}(\varepsilon)\,I+(\gamma-\beta)\skw(\varepsilon),
\]
where \(\mu,\lambda\) are Lam\'e parameters, \(\mu_{c}\) is the Cosserat coupling constant, and \(\alpha,\beta,\gamma\) are additional micropolar moduli.

The coupling term \(\grad \vec{u}-\mskw \,\vec{\omega}\) arises from linearising the (multiplicative) action \(\exp(\mskw \vec{\omega})\in \mathrm{SO}(3)\) on the deformation.  
By separating the symmetric and skew-symmetric parts, \eqref{Cosserat-energy} can be rewritten as
\[
\mathcal{E}^{\mathrm{Cosserat}}(\vec{u}, \vec{\omega})
=\!\int_{\Omega}
\Bigl(\tfrac{1}{2}\|\sym\grad \vec{u}\|_{\mathcal{C}}^{2}
+\mu_{c}\bigl\|\tfrac{1}{2}\curl \vec{u}-\vec{\omega}\bigr\|^{2}
+\tfrac12\|\grad \vec{\omega}\|_{C_{2}}^{2}
-\langle \vec{f}_{\vec{u}}, \vec{u}\rangle 
-\langle \vec{f}_{\vec{\omega}}, \vec{\omega}\rangle
\Bigr)\,dx.
\]
When \(\mu_{c}=0\), the fields \(\vec{u}\) and \(\vec{\omega}\) decouple, so \(\vec{u}\) solves the standard elasticity problem. In the limiting case \(\mu_{c}\to \infty\), we impose \(\vec{\omega}=\tfrac12\,\curl \vec{u}\). Then the energy includes a ``mixed'' 4th--2nd order term \(\|\grad \curl \vec{u}\|_{C_{2}}^{2}\), describing a couple stress model~\cite{park2008variational}.

\medskip

This model can be recognised as the 0th Hodge--Laplacian of the twisted complex \eqref{twisted-elasticity-2D}. Precisely, for the 0-forms in \eqref{twisted-elasticity-2D}, the variable has two components: \(\vec{u}\) (displacement, top row) and \(\vec{\omega}\) (rotation, bottom row). The corresponding Hodge--Laplacian problem has the energy
\begin{align}\nonumber
\bigl\|\Bigl(\begin{smallmatrix}
\grad & \mskw\\[4pt]
0 & \grad
\end{smallmatrix}\Bigr)
\Bigl(\begin{smallmatrix}
\vec{u}\\
\vec{\omega}
\end{smallmatrix}\Bigr)\bigr\|_{C}^{2}
&=\|\sym(\grad \vec{u} + \mskw\,\vec{\omega})\|_{C_1}^2
+ \mu_{c} \|\skw(\grad \vec{u} + \mskw\,\vec{\omega})\|_{C_2}^{2}
+ \|\grad \vec{\omega}\|_{C_3}^{2}\\
&=\|\sym\grad \vec{u}\|_{C_1}^{2}
+ \mu_{c}\|\curl \vec{u} + \vec{\omega}\|_{C_2}^{2}
+ \|\grad \vec{\omega}\|_{C_3}^{2},
\label{cosserat-energy-linear}
\end{align}
which precisely corresponds to the linearised Cosserat model~\cite{neff2008linear}. For \(\mu \to 0\), \(\vec{\omega}\) decouples, reducing to classical linear elasticity; for \(\mu \to \infty\), we get \(\vec{\omega}=-\curl \vec{u}\) and hence a higher-order couple-stress type model~\cite{park2008variational}. The BGG machinery provides an algebraic viewpoint on eliminating or reintroducing the rotation variable.

Such micropolar models involve multiple scalings and constants, leading to challenges for parameter-robust numerical schemes. Notably, parameter-robust methods for Cosserat models should also handle the couple-stress regime.

\medskip

In summary, the twisted de~Rham and BGG complexes in 1D, 2D, and 3D correspond to various beam, plate, and elasticity models:
\[
\begin{array}{c|c|c}
\hline
 & \text{Twisted Complex} & \text{BGG Complex}\\
\hline
\text{1D} & \text{Timoshenko beam} & \text{Euler--Bernoulli beam}\\
\text{2D} & \text{Reissner--Mindlin plate} & \text{Kirchhoff--Love plate}\\
\text{3D} & \text{Cosserat elasticity} & \text{Classical elasticity}\\
\hline
\end{array}
\]
From this perspective, the cohomology-preserving projection \cite{arnold2021complexes} can be viewed as the process of eliminating rotational degrees of freedom from the Cosserat model to recover classical elasticity. The mappings that takes the BGG complex back to the twisted de~Rham complex reflect adding back the rotation field. An analogous interpretation holds for the Kirchhoff and Reissner--Mindlin plates. Moreover, both Reissner--Mindlin and Kirchhoff plate models can be obtained as dimension reductions of Cosserat and standard elasticity, respectively, with rigorous \(\Gamma\)-limit derivations in \cite{ciarlet1997mathematical,neff2010reissner}. These relationships are sketched in the following diagram:
\begin{equation}\label{diagram:models}
\begin{tikzcd}[row sep=large, column sep=huge]
\mathrm{Cosserat~elasticity} 
 \arrow{r}{\mathrm{BGG~(elasticity)}}
 \arrow[d, "{\mathrm{dimension~reduction}}"]
 & 
\mathrm{classical~elasticity}
\arrow[d, "{\mathrm{dimension~reduction}}"]
\\
\mathrm{(modified)~Reissner\!-\!Mindlin~plate}
\arrow{r}{\mathrm{BGG~(hessian)}}  
& 
\mathrm{Kirchhoff~plate}
\end{tikzcd}
\end{equation}

This cohomological viewpoint on the linear Cosserat model yields well-posed Hodge--Laplacian boundary-value problems and valuable tools for analysis and numerical approximation. Each BGG complex leads to a Hodge--Laplacian formulation (energy functional), and other complexes or diagrams may further generalise the Cosserat theory to other microstructures. 

\subsection{Mixed dimensional models: \v{C}ech complexes}

The \v{C}ech--de Rham double complex is another commonly used tool in algebraic topology (as well as in topological data analysis) \cite{bott2013differential}. In addition to the exterior derivatives in the de~Rham complex, the double complex introduces overlapping subdomains (where domains either partly or completely overlap, or only intersect at the interfaces) and an alternating sum of the restriction or trace operators on these subdomains. These ``jump operators'' $\delta^{\bs}$ led by the alternating sum also form a complex. Moreover, we can combine the exterior derivatives $d^{\bs}$ and the jump operators $\delta^{\bs}$ to form a new complex. The double complex and other complexes derived from it incorporate richer interface and multiphysics information. Below, we present basic definitions of \v{C}ech--de Rham double complexes and demonstrate some examples. 
 
Let \(\{U_j\}_j\) be an open cover of \(\Omega\), and denote
\[
U_{ij} := U_i \cap U_j, 
\quad
U_{ijk} := U_i \cap U_j \cap U_k,
\]
etc. Let \(\Lambda^{k}(U_{i_1 \cdots i_\ell})\) be the space of \(k\)-forms on \(U_{i_1} \cap \cdots \cap U_{i_\ell}\). We introduce
\[
\mathcal{W}^{k,s} := \bigoplus_{\,i_{0}, \dots, i_{s}} \Lambda^{k}\bigl(U_{\,i_{0}, \dots, i_{s}}\bigr),
\]
the direct sum of \(k\)-forms on all intersections involving \((s+1)\) indices. The exterior derivative \(d^k\) acts by
\[
d^k:\,\mathcal{W}^{k,s} \to \mathcal{W}^{k+1,s},
\]
since applying \(d\) does not change the domain of definition (only the degree of the form).

The Čech--de~Rham complex (sometimes called the Čech complex) adds a ``vertical differential'' \(\delta\) that goes in the other direction of a double complex. We define
\[
\delta:\,\mathcal{W}^{k,s} \to \mathcal{W}^{k,s+1}
\]
as follows. If \(\alpha \in \mathcal{W}^{k,s}\), write \(\alpha_{i_0,\dots,i_s}\) for the component of \(\alpha\) on
\(
U_{i_0}\cap \cdots \cap U_{i_s}.
\)
Then
\[
(\delta \alpha)_{\,i_0,\dots,i_{s+1}}
=
\sum_{j=0}^{s+1}
(-1)^j\,\alpha_{\,i_0,\dots,i_{j-1},\,i_{j+1},\dots,i_{s+1}}
\big|_{\,U_{\,i_0,\dots,i_{s+1}}}.
\]
The factor \((-1)^j\) indicates the usual alternating sum. 

\begin{example}
If \(\Omega\) is covered by \(U_1\) and \(U_2\) only, then \(\mathcal{W}^{k,1}\) has just one intersection \(U_{1,2}\). For \(\alpha \in \mathcal{W}^{k,0}\),
\[
(\delta \alpha)_{1,2} = (\alpha_1 - \alpha_2)\big|_{U_{1,2}},
\]
where \(\alpha_1\) and \(\alpha_2\) are the components of \(\alpha\) on \(U_1\) and \(U_2\). For three subdomains \(U_0, U_1, U_2\), one similarly obtains
\[
(\delta \alpha)_{0,1,2}
=
\alpha_{1, 2}\big|_{U_{0,1,2}}-\alpha_{0, 2}\big|_{U_{0,1,2}}+\alpha_{0, 1}\big|_{U_{0,1,2}}.
\]
\end{example}

It is straightforward to check that \(\delta\circ d = d\circ \delta\), and \(\delta^2 = 0\) just as with boundary operators. Hence we obtain a diagram:
\begin{equation}\label{double-complex}
\begin{tikzcd}
0 \arrow{r} &\mathcal{W}^{0, 0}\arrow{r}{d} \arrow{d}{\delta}& \mathcal{W}^{1, 0}\arrow{r}{d}\arrow{d}{\delta} &\mathcal{W}^{2, 0}\arrow{r}{d}\arrow{d}{\delta} &\mathcal{W}^{3, 0}\arrow{r}{} \arrow{d}{\delta}& \cdots\\
0 \arrow{r} &\mathcal{W}^{0, 1}\arrow{r}{d} \arrow{d}{\delta}& \mathcal{W}^{1, 1}\arrow{r}{d}\arrow{d} {\delta}&\mathcal{W}^{2, 1}\arrow{r}{d}\arrow{d}{\delta} &\mathcal{W}^{3, 1}\arrow{r}{}\arrow{d}{\delta} \arrow{}& \cdots\\
0 \arrow{r} &\mathcal{W}^{0, 2}\arrow{d}  \arrow{r}{d}& \mathcal{W}^{1, 2}  \arrow{r}{d} \arrow{d} &\mathcal{W}^{2, 2}\arrow{r}{d} \arrow{d} &\mathcal{W}^{3, 2}\arrow{d} \arrow{r}{}& \cdots\\
0 \arrow{r}&\cdots \arrow{r}{}& \cdots \arrow{r}{} & \cdots \arrow{r}{} &\cdots\arrow{r}{}&\cdots \\
\end{tikzcd}
\end{equation}

Define
\[
\mathcal{W}^{k}
:=
\bigoplus_{p+q=k}
\mathcal{W}^{p,q},
\quad
D^k := d + (-1)^k \delta.
\]
Since \(d\) and \(\delta\) commute and each squares to zero, one checks \(D^k \circ D^{k-1} = 0\). Thus
\[
\begin{tikzcd}
\cdots \arrow[r]{} 
& \mathcal{W}^{k-1} \arrow[r,"D^{k-1}"]
& \mathcal{W}^{k} \arrow[r,"D^k"]
& \mathcal{W}^{k+1} \arrow[r]{}
& \cdots
\end{tikzcd}
\]
is a complex, called the \emph{Čech--de~Rham complex}.

The above Čech--de~Rham complex also extends to degenerate structures in which different sets \(U_i, U_{ij},\dots\) can have varying dimensions. For instance, in a combinatorial elastic structure, beams (1D) may be coupled with shells (2D). Another example is a simplicial complex, where each \(U_i\) is a top-dimensional cell, and \(U_{ij}\) incorporates faces of codimension 1, etc. In such cases, the vertical operator \(\delta\) involves trace operators on the intersections, which makes the analytic theory of the Čech--de~Rham complex more delicate. We refer to \cite{boon2021functional} for further discussions in this direction.

Ignoring the analytic subtleties, in the rest of this section, we show a {\it formal} Hodge--Laplacian problem of the \v{C}ech complex to demonstrate connections between the \v{C}ech complex and mechanics models.  Consider the cover \(\{U_1, U_2\}\). The sub-diagram
\[
\begin{tikzcd}[column sep=2em]
\mathcal{W}^{0,0} \arrow[r,"d"] \arrow[d,"\delta"] & \mathcal{W}^{1,0}\\
\mathcal{W}^{0,1} & {}
\end{tikzcd}
\]
suggests an energy functional for \(u \in \mathcal{W}^{0,0}\):
\[
\|\nabla u\|^2_{U_1}
+ 
\|\nabla u\|^2_{U_2}
+ 
\|\,u|_{U_1} - u|_{U_2}\|^2_{U_{1}\cap U_{2}}
+ 
(f,\,u).
\]
Minimizing it formally yields an Euler--Lagrange system,
\begin{align*}
-\Delta (u|_{U_1}) + \mathbb{1}_{U_{1,2},U_1}(u|_{U_1} - u|_{U_2})
&=
(f,\,u|_{U_1}),\\
-\Delta (u|_{U_2}) - \mathbb{1}_{U_{1,2},U_2}(u|_{U_1} - u|_{U_2})
&=
(f,\,u|_{U_2}),
\end{align*}
together with boundary conditions. Here \(\mathbb{1}_{U_{1,2},U_i}\) is the characteristic function: \(\mathbb{1}_{U_{1,2},U_i}\) is a function defined on $U_{i}$, which equals to 1 on $U_{1, 2}$ and 0 elsewhere in $U_{i}$. These terms containing \(\mathbb{1}_{U_{1,2},U_i}\) incorporates contributions from the interactions between $U_{1}$ and $U_{2}$.
Such constructions are quite flexible regarding how the cover is chosen (e.g.\ one can even set \(U_1 = U_2 = \Omega\) to mimic a double-continuum model in porous media \cite[Sec.\ 3.2]{boon2022hodge}).

Mixed dimensional modelling has also been used in other applications such as brain’s waterscape \cite{rognes2022waterscales}, reinforced structures \cite{le2017mixed} and viscoelasticity and gels van der Waals heterostructures \cite{jariwala2017mixed}.


\begin{remark}
In his paper in the 1970s, Feng Kang \cite{fengkang} discussed finite elements on combinatorial manifolds with applications to combinatorial elastic structures. Some of the open questions listed in the paper include 
\begin{itemize}
\item the regularity of solutions to elliptic equations on combinatorial manifolds,
\item the generalisation of the de~Rham--Hodge theory to combinatorial manifolds.
\end{itemize}
These issues are still of significant interest, especially given modern mixed-dimensional and heterogeneous modeling applications. Particularly, developing a \v{C}ech-complex-based approach for modelling and computation of multidimensional models complies with the second question above.
\end{remark}

\section{Topological hydrodynamics and magnetohydrodynamics}\label{sec:fluids}

In this section, we review the connections between fluid mechanics and differential complex and cohomology. Specifically, we formulate the divergence-constraint of incompressible flows and advection-diffusion of fluids using differential forms; we discuss the notation of helicity and a topological mechanism underlying long-term evolution of these systems, explaining the striking difference of numerical behaviours in Example \ref{example:relaxation}.

\subsection{Divergence-constraints and advection-diffusion by differential forms}

We view fluid dynamics, particularly incompressible Navier--Stokes equations and magnetohydrodynamics, in terms of differential complexes and differential forms. The shift of point of view will provide new insights in designing structure-preserving numerical schemes.


A straightforward application of differential complexes in fluid mechanics is the incompressibility condition $\nabla\cdot\bm u=0$ in incompressible flow \eqref{NS}, which plays an important role in mathematical and numerical analysis. Written in a weak form, \eqref{NS} becomes: seek $\bm u: [0, T]\to \bm H_{0}^{1}(\Omega)$, such that for all $\bm v\in \bm H_{0}^{1}(\Omega)$, 
\begin{subequations}\label{NS-weak}
\begin{align}
(\bm u_{t}, \bm v)+((\bm u\cdot \nabla) \bm u, \bm v)+R_{e}^{-1}(\nabla  \bm u, \nabla \bm v)-( p, \nabla\cdot \bm v)&=(\bm f, \bm v), \\
(\nabla\cdot \bm u, q)&=0.
\end{align}
\end{subequations}
The divergence operator $\nabla\cdot: {\bm H}_{0}^{1}\to L^{2}/\mathbb{R}$ is onto. More precisely, the {Ladyzhenskaya-Babu\v{s}ka-Brezzi  (inf-sup) condition} holds
$$
\inf_{q\in L^{2}/\mathbb{R}}\sup_{\bm  v\in \bm H_{0}^{1}}\frac{(\nabla\cdot \bm v, q)}{\|\bm v\|_{H^{1}}\|q\|_{L^{2}}}\geq \gamma>0.
$$
The surjectivity and the norm control in the inf-sup condition are consequences of the fact that the de~Rham complex (with proper Sobolev spaces) is exact at $L^{2}/\mathbb{R}$. These results imply the well-posedness of the weak form. To see the role played by the algebraic structure, let's suppose that the divergence operator is not surjective, but has a nontrivial cokernel (which may happen on the discrete level if inappropriate finite element spaces for the velocity and pressure are chosen). That is, there exists a subspace $Q^{\perp}\subset L^{2}/\mathbb{R}$ such that $(\nabla\cdot \bm v, q)=0, \forall q\in Q^{\perp}$. Since $p$ in \eqref{NS-weak} is only involved in the term $( p, \nabla\cdot \bm v)$, adding any element of $Q^{\perp}$ to $p$ does not affect the equation. Therefore, $p$ is not unique unless $Q^{\perp}$ is trivial, i.e., the divergence operator is surjective.

The magnetohydrodynamics (MHD) system describes the macroscopic motion of conducting fluids with wide applications to plasma and astrophysics. The governing equation of incompressible MHD system is 
\begin{subeqnarray}\label{eqn:MHD}
\partial_t \bm{u} - \bm u \times (\nabla\times \bm u) - R_{e}^{-1}\Delta  \bm{u} - {s \bm{j} \times \bm{B}} + \nabla P &=& \bm{f}\quad 
\mbox{momentum equation},   \\ 
\bm{j} - \nabla \times \bm{B}  &=& \bm{0} \quad \mbox{Ampere's law}, \\
\partial_t\bm{B}  + \nabla \times \bm{E} &=& \bm{0}\quad \mbox{Faraday's law},  \\ 
R_m^{-1} \bm{j} -  \left ( \bm{E} + {{\bm{u}} \times \bm{B}} \right ) &=& \bm{0}\quad \mbox{Ohm's law},\\ 
\nabla \cdot \bm{B} &=& 0\quad \mbox{Gauss law}, \\
\nabla \cdot {\bm{u}} &=& 0,  
\end{subeqnarray}
with  initial conditions $\bm{u}(\bm{x},0) =  \bm{u}_{0}(\bm{x}), \bm{B}(\bm{x},0) =  \bm{B}_{0}(\bm{x})$, and 
 boundary conditions on $\partial \Omega$: $ \bm{u} = \bm{0}, \bm{B}\cdot \bm{n} = 0,  \bm{E}\times \bm{n} = \bm{0}$. Note that we have used the Lamb form of the convection term $\bm u \times (\nabla\times \bm u) $ and $P$ is the total pressure. Eliminating $\bm E$ and $\bm j$, we can derive the equation of $\bm B$:
 \begin{equation}\label{eqn:advection-diffusion-B}
 \partial_{t}\bm B-\nabla\times (\bm u\times \bm B)+R_{m}^{-1}\nabla\times \nabla\times \bm B=0.
 \end{equation}
The equation \eqref{eqn:advection-diffusion-B} is often called the advection-diffusion of the magnetic field. The term $-\nabla\times (\bm u\times \bm B)$ describes how the magnetic field is transported by the velocity field; the term $R_{m}^{-1}\nabla\times \nabla\times \bm B$ describes diffusion since $-\Delta \bm B=\nabla\times \nabla\times -\nabla\nabla\cdot \bm B$, and we have $\nabla\cdot\bm B=0$. The equation has a similar form as the vorticity equation
 \begin{equation}\label{eqn:advection-diffusion-w}
 \partial_{t}\bm \omega-\nabla\times (\bm u\times \bm \omega)+R_{m}^{-1}\nabla\times \nabla\times \bm \omega=0.
 \end{equation}
 The difference is that in \eqref{eqn:advection-diffusion-B}, $\bm u$ and $\bm B$ are independent (in the full MHD system, $\bm u$ is determined through the coupling with the Navier--Stokes equation); while in \eqref{eqn:advection-diffusion-w}, $\bm \omega=\nabla\times \bm u$.

In fluid mechanics and MHD, advection-diffusion equations of differential forms arise: seeking $k$-forms $\omega$ such that 
$$
\mathscr{L}_{u}\omega+\Delta_{\mathrm{HL}}\omega=f, 
$$
where $\mathscr{L}_{u}\omega$ is the Lie derivative (the flow derivatives or `fisherman derivatives' \cite{arnold1999topological}) of $\omega$ along a vector field $u$, and $\Delta_{\mathrm{HL}}$ is the Hodge--Laplacian.
 
 Given a vector field $u$ on a manifold $\Omega$, the {\it flow}\index{flow} generated by $u$ is a map $\Phi: \mathbb{R}\times \Omega\to \Omega$ such that $\partial_{t}\Phi(t, x)=u(\Phi(t, x), t)$ with $\Phi(0, x)=x$. 
 Let $\omega\in C^{\infty}\Lambda^{k}$ be a $k$-form on $\Omega$. The {\it Lie derivative}\index{Lie derivative} of $\omega$ along $u$ is defined by 
 \begin{equation}\label{def:LieDerivative}
 \mathsf{L}_{u}\omega=\left.\frac{d}{dt}\right|_{\tau=0}\Phi_{\tau}^{\ast}\omega=\lim_{\tau\rightarrow 0}\frac{\Phi_{\tau}^{\ast}\omega-\omega}{\tau}.
 \end{equation}
 For a 0-form (function) $\omega$, the above definition boils down to 
  $$
 \mathsf{L}_{u}\omega (x)=\left.\frac{d}{dt}\right|_{\tau=0}\omega=\lim_{\tau\rightarrow 0}\frac{\omega(\Phi_{\tau}x)-\omega(x)}{\tau}.
 $$
 This is just the derivative along the direction of $u$. 
 {\it Cartan's magic formula}\index{Cartan's magic formula}, or the {\it homotopy formula}, provides a connection between Lie derivatives and exterior derivatives, and gives new insights to the above seemingly involved formulas. 
 \begin{theorem}[Cartan's magic formula]
 Let $\beta$ be a smooth vector field on $\Omega$. Then we have 
 \begin{equation}\label{id:magic-formula}
\mathsf{L}_{\beta}=d\iota_{\beta}+\iota_{\beta}d,
 \end{equation}
 where $\iota_{\beta}$ is the contraction of a differential form by the vector field $ \beta$, i.e., $ \iota_{\beta}: C^{\infty}\Lambda^{k}\to C^{\infty}\Lambda^{k-1}$ is defined by
 $$
 \iota_{\beta}\omega(v_{1}, \cdots, v_{k-1}):=\omega(\beta, v_{1}, \cdots, v_{k-1}), \quad\forall \omega\in C^{\infty}\Lambda^{k}. 
 $$
 \end{theorem}
 Cartan's magic formula has the following proxies in 3D:
 \begin{tcolorbox}[colback=red!5!white,colframe=red!75!black]
$k=0$: \quad $\mathsf{L}_{\bm \beta}u=\bm\beta\cdot \grad u$,\\
$k=1$: \quad $\mathsf{L}_{\bm\beta}v=-\bm\beta\times \curl \bm v+\grad(\bm v\cdot \bm\beta)$,\\
$k=2$: \quad $\mathsf{L}_{\bm\beta}w= (\div  \bm w)\bm\beta+\curl(\bm w\times \bm\beta)$, \\
$k=3$: \quad $\mathsf{L}_{\bm\beta}g= \div (g\bm \beta)$. 
\end{tcolorbox}

The case $k=0$ corresponds to the standard scalar advection-diffusion equation; $k=2$ corresponds to the advection-diffusion of the magnetic field in MHD \eqref{eqn:advection-diffusion-B} or the vorticity \eqref{eqn:advection-diffusion-w}; and $k=1$ corresponds to the advection-diffusion of the magnetic potential. 
 
 Catan's magic formula indicates that there are at least two ways to compute the Lie derivative: one is by the definition \eqref{def:LieDerivative} and another is by the exterior derivatives \eqref{id:magic-formula}. In fact, they give two different approaches for numerical discretisation of the advection equations of differential forms: one is to use a finite difference to replace the limit in the definition \eqref{def:LieDerivative} (this involves solving the flow maps using an ODE solver), and another is to discretise differential operators appearing in the magic formula \eqref{id:magic-formula}. The former corresponds to Lagrangian (or semi-Lagrangian) methods \cite{heumann2012fully,pagliantini2016computational}, and the latter corresponds to Eulerian methods.

\subsection{Helicity and topological hydrodynamics}
As a second example of the appearance of de~Rham complexes in fluid dynamics, we consider the notion of helicity describing knots of fields and the topological mechanism encoded in helicity conservation.

Topological structures in fluid flows and other physical models, particularly vortex knots and linked vortex tubes, have drawn attention since the 19th century, when Lord Kelvin proposed that atoms might be modelled as knotted vortex rings \cite{kelvin1867vortex}. Later, Thomson discovered the electron (for which he received a Nobel Prize) and proposed his plum pudding model. The development vortex atomic theory was thus paused. However, this development inspired knot theory as a mathematical theory, and the idea of knotting later thrived in the study of fluid dynamics: knots reflect deep topological properties and conservation laws within fluid dynamics. 

A key quantity capturing the ``knottedness” or linking of flow fields is {\it helicity}, first formalised in the mid-20th century.  In the context of MHD,  Lodewijk Woltjer \cite{woltjer1958theorem} defined the magnetic helicity, known as Woltjer's invariant. Keith Moffatt \cite{moffatt1969degree} and others further established the concept of helicity as a topological measure in fluid flows, emphasising its role in describing knottedness or linkage in field lines. Figure \ref{fig:knots} shows two linking tubes of a divergence-free field. The helicity is proportional to the product of the fluxes of the tubes and a topological constant (the linking number). In general, helicity is the averaging asymptotic linking number, which can be derived as a continuum version of linked tubes \cite{arnold1999topological}. 
Helicity remains invariant under ideal (nondissipative) flow motions. This makes helicity a powerful tool for understanding the evolution of complex structures in both classical fluid dynamics and plasma physics, where the interplay between geometry, topology, and dynamics often governs the stability and behaviour of flows.

\begin{figure}
  \includegraphics[width=0.35\textwidth]{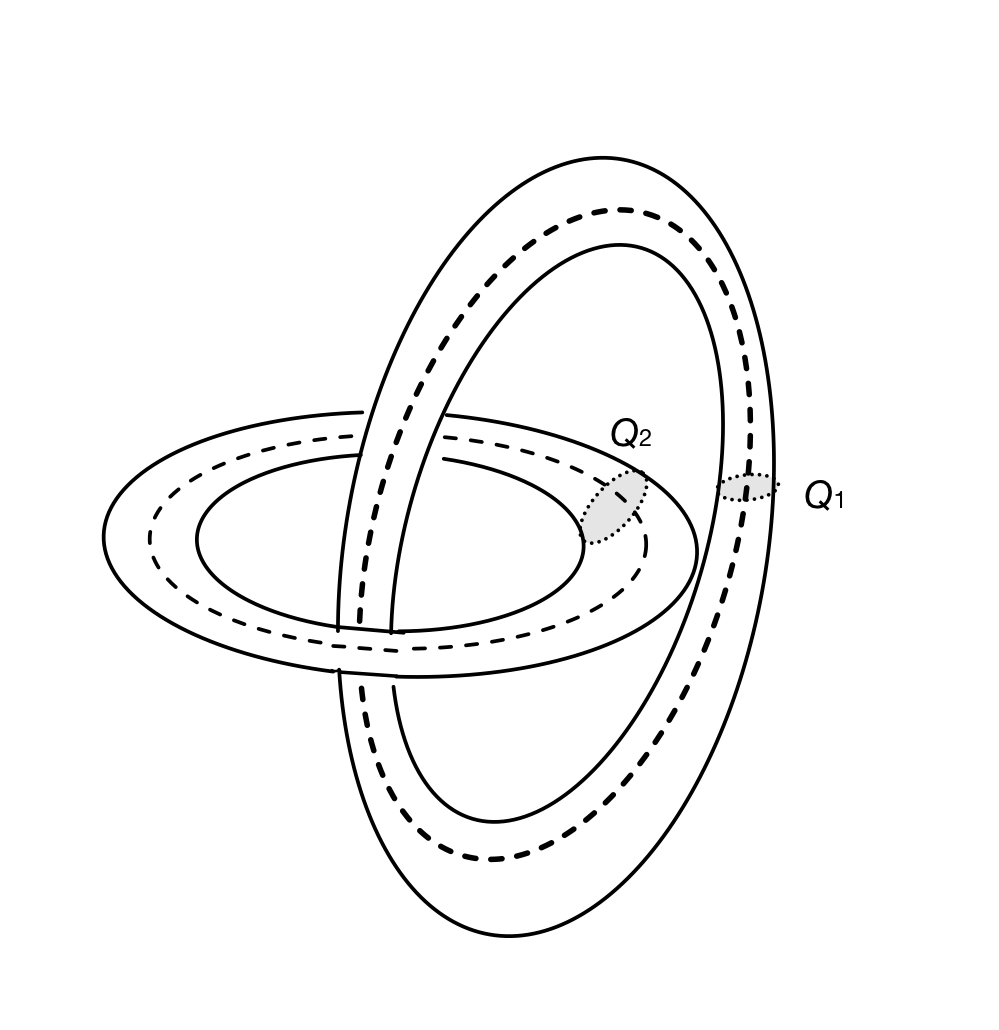} 
    \caption{Two tubes of a field $\bm \xi$ linked with each other. Let $Q_{1}$ and $Q_{2}$ be the fluxes of the two tubes, respectively. Then the helicity is 
    $\mathcal{H}_{\bm{\xi}}=2\ell Q_{1}\cdot Q_{2}$,  where $\ell$ is the Gauss linking number only depending on the topology of the knot ($\ell=1$ in the figure).  }
  \label{fig:knots}
  \end{figure}

Helicity can be defined for any divergence-free field in $\mathbb{R}^{3}$ with a potential, i.e., for $\bm \xi$ satisfying $\div \bm \xi=0$, the corresponding helicity $\mathcal{H}_{\bm \xi}$ is defined by
$$
\mathcal{H}_{\bm \xi}:=\int_{\Omega} \bm \xi\cdot \curl^{-1} \bm \xi\, dx,
$$ 
where $\Omega$ is a subdomain of $\mathbb{R}^{3}$ and $ \curl^{-1} \bm \xi$ is a potential of $\mathcal{H}_{\bm \xi}$ satisfying $\curl (\curl^{-1} \bm \xi)=\bm \xi$.  First, we note that $ \curl^{-1} \bm \xi$ is not unique (up to gradient fields and 1-form cohomology representatives), and the above definition does not depend on the specific choice of $ \curl^{-1} \bm \xi$ as long as $\bm \xi$ is tangent to the boundary of $\Omega$. To see this, let $\bm\eta=\curl^{-1}\bm\xi$ and $\bm \eta':=\bm \xi+\grad \varphi+\bm h$, where $\bm h$ is a harmonic form. Then 
$$
\int_{\Omega} \bm \eta'\cdot \bm \xi\,dx=\int_{\Omega}(\bm \eta+\grad \varphi+h)\cdot \bm \xi\, dx=\int_{\Omega} \bm \eta\cdot \bm \xi\,dx,
$$
where we have used the fact that $\int_{\Omega}\grad \varphi\cdot  \bm \xi\, dx=-\int_{\Omega}\varphi\cdot \div \bm \xi\, dx=0$ (integration by parts) and $\int_{\Omega}\bm h\cdot \bm\xi\,dx=0$ by the definition of harmonic forms. Second, the definition of helicity involves a potential, which only exists when harmonic 2-forms are trivial, though various generalisations exist \cite{arnold1999topological}. 

For a divergence-free field $\bm \xi$, its energy is defined to be $\mathscr{E}(\bm \xi):=\int_{\Omega}|\bm \xi|^{2}\, dx$. We have the following:
\begin{theorem}[Arnold inequality \cite{arnold1974asymptotic}]
The helicity provides a lower bound for the energy:
\begin{equation}\label{arnold-inequality}
|\int \curl^{-1}\bm \xi\cdot \bm \xi\, dx|\leq C\int  |\bm \xi|^2\, dx.
\end{equation}
\end{theorem}
\begin{proof}
The desired result is a consequence of the Cauchy-Schwarz inequality 
$$
|\int \curl^{-1}\bm \xi\cdot \bm \xi\, dx|\leq \|\curl^{-1}\bm \xi\|_{L^{2}}\|\bm \xi\|_{L^{2}},
$$ and the Poincar\'e inequality 
$$
\|\curl^{-1}\bm \xi\|_{L^{2}}\leq C\|\nabla\times (\curl^{-1}\bm \xi)\|_{L^{2}},
$$ 
which is again a corollary of finite dimensional cohomology; see Section \ref{sec:analysis}.
\end{proof}
The Arnold inequality is based on a simple proof. However, its consequence is significant. We view $\bm \xi$ as a 2-form, and the potential $\curl^{-1}\bm \xi$ is a 1-form. Then the inequality \eqref{arnold-inequality} is the coordinate form of the following inequality of forms:
$$
|\int \curl^{-1}\bm \xi\wedge \bm \xi\, dx|\leq C\int  \bm \xi\wedge \star \bm\xi\, dx.
$$
The left hand side does not involve metric and is thus topological. The right hand side involves the Hodge star and therefore a metric, which is geometric. The inequality further clarifies the fact that knots provide a topological barrier for the energy to decay.

In the context of fluids, the (fluid) helicity is defined to be $\int_{\Omega}\bm u\cdot \bm \omega\, dx$. 
 In MHD, the magnetic field $\bm B$ is divergence-free, and thus a magnetic helicity is defined:
 $$
 \mathcal{H}_{m}:=\int_{\Omega}\bm A \cdot \bm B\,dx,
 $$
 where the magnetic potential $\bm A$ satisfies $\nabla\times \bm A=\bm B$. Another quantity is the {\it cross helicity} 
 $$
 \mathcal{H}_{c}:=\int_{\Omega}\bm u \cdot \bm B\,dx.
 $$
 \begin{theorem}
 In the absence of magnetic diffusion (the magnetic Reynolds number $R_{m}=\infty$), the magnetic helicity is conserved 
 $$
 \frac{d}{dt} \mathcal{H}_{m}=0.
 $$
  If further for ideal fluids ($R_{e}=R_{m}=\infty$) and $\bm f$ is a gradient field, then the cross helicity is conserved 
  $$
  \frac{d}{dt}\mathcal{H}_{c}=0.
  $$
 \end{theorem}
The MHD system \eqref{eqn:MHD} satisfies energy dissipation or conservation:
 \begin{eqnarray*}
{1 \over 2}\frac{d}{dt} \|\bm{u}\|_0^2 + \frac{s}{2}  \frac{d}{dt}\|\bm{B}\|_0^2 + R_{e}^{-1} \|\nabla \bm{u}\|_0^2  + s R_m^{-1} \| \bm{j}\|_0^2= (\bm{f}, \bm{u}),
\end{eqnarray*}
and hence
\begin{eqnarray*}
\max_{0\leq t\leq T} \left(
    \|\bm{u}\|_0^2 + s \|\bm{B}\|_0^2 \right) 
    + R_{e}^{-1} \int_{0}^{T} \|\nabla
    \bm{u}\|_0^2\,\mathrm{d}\tau + 2 s R_m^{-1} \int_{0}^{T}
    \|\bm{j}\|_0^2 \,\mathrm{d}\tau \\
    \leq  ~ \|\bm{u}_{0}\|_0^{2} + s \|\bm{B}_{0}\|_0^{2} + 
    {R_{e}} \int_0^T\|\bm{f}\|_{-1}^{2}\,\mathrm{d}\tau. 
\end{eqnarray*}
\medskip
With $\bm f=0$ and $R_{m}^{-1}=0$, the total energy is non-increasing. However, some key information is {not clear}: 
\begin{itemize}
\item whether the total energy decays to zero?
\item how does the total energy split into the the fluid part ($\|\bm u\|^{2}$) and the magnetic part ($s\|\bm B\|^{2}$)? 
\end{itemize}
Helicity provides further information for these questions and thus characterise finer structures of MHD. Below we show the significance of helicity by a problem from solar physics. 

In 1972, Eugene N.~Parker made a conjecture about the evolution of initial states in magnetically-ideal plasmas~\cite{parkerTopologicalDissipationSmallScale1972}.
Although several versions and different statements have been discussed in the literature \cite{pontinParkerProblemExistence2020}, the Parker conjecture essentially claims that \emph{for almost all possible flows, the magnetic field develops tangential discontinuity (current sheets) during ideal magnetic relaxation to a force-free equilibrium.} 
This conjecture has many important consequences in solar physics, including explaining the mechanism of coronal heating.
For a comprehensive literature review, see \cite{pontinParkerProblemExistence2020}.

Below we focus on the case where in conducting fluids, fluid diffusion exists to dissipate the energy (the fluid Reynolds number $R_{e}<\infty$), while the magnetic part is ideal (the magnetic Reynolds number $R_{m}=\infty$).
The canonical energy estimate indicates that the total energy of the fluid and the magnetic field is non-increasing.
Furthermore, the magnetic helicity $\mathcal{H}_{m}$ is conserved. 
A significant consequence of the Arnold inequality is that, {\it despite dissipation in the energy, initial data with nonzero $\mathcal{H}$ cannot relax to zero}. This reflects a crucial topological mechanism in magnetic relaxation (free evolution of the MHD system).

\begin{figure}
\begin{center}
\begin{tikzpicture}
  \draw[blue,  thick] (0.5, 2) .. controls (0.7, 1.6) and (0.4, 0.8) .. (1.4, 1.2) .. controls (1.6, 0.8) and (0.8, 0.4) .. (0.7, 0);
  \draw[blue,  thick] (0.7, 2) .. controls (0.9, 1.4) and (1.4, 1.8) .. (1.1, 0.6) .. controls (0.6, 0.2) and (1.1, 0.4) .. (0.9, 0);
  \draw[blue,  thick] (0.9, 2) .. controls (1.1, 1.6) and (0.6, 1.0) .. (1.3, 0.8) .. controls (1.5, 0.4) and (0.9, 0.2) .. (1.1, 0);
  \draw[blue,  thick] (1.1, 2) .. controls (0.9, 1.4) and (1.7, 1.6) .. (1.2, 1.0) .. controls (0.8, 0.6) and (1.3, 0.2) .. (1.3, 0);
  \draw[blue,  thick] (1.3, 2) .. controls (1.5, 1.8) and (1.0, 0.8) .. (1.6, 1.2) .. controls (1.2, 0.6) and (1.1, 0.2) .. (1.5, 0);
  \draw[blue,  thick] (1.5, 2) .. controls (1.7, 1.6) and (1.1, 1.4) .. (0.9, 0.8) .. controls (1.2, 0.4) and (1.6, 0.2) .. (1.7, 0);

  \draw[blue,  thick] (4.5, 2) -- (4.5, 0);
  \draw[blue,  thick] (4.7, 2) -- (4.7, 0);
  \draw[blue,  thick] (4.9, 2) -- (4.9, 0);
  \draw[blue,  thick] (5.1, 2) -- (5.1, 0);
  \draw[blue,  thick] (5.3, 2) -- (5.3, 0);
  \draw[blue,  thick] (5.5, 2) -- (5.5, 0);

  \node at (3.25, 1) {$\nRightarrow$};
\end{tikzpicture}
\end{center}
  \caption{Knots of field lines provide a topological barrier for the decay of energy. Helicity conservation and the Arnold inequality imply that topologically nontrivial initial data cannot relax to a trivial state.}
  \label{fig:helicity-initial}
\end{figure}
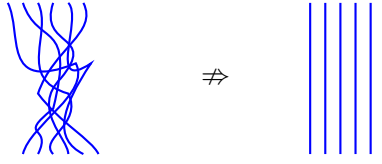

 
 This explains the reason for helicity-preserving algorithms and plain algorithms behaving dramatically differently in  Example \ref{example:relaxation}: algorithms that do not preserve helicity lost the topological mechanism (knots as barriers presenting the energy from decaying) on the discrete level, and thus lead to physically wrong solution in long-term evolution.
 
In fact, topology-preservation is a well known challenge in computational plasma physics. Quoting \cite{pontinParkerProblemExistence2020}:

\begin{quote}
{Direct computational assessment of Parker’s hypothesis brings a number of challenges. Foremost among these is the requirement to {\it precisely maintain the magnetic topology during the simulated evolution}, i.e., precisely maintain the magnetic field line mapping between the two line-tied boundaries. ...
In the following sections two methods are described which seek to mitigate against these difficulties. However, in all cases the {\it representation of current singularities remains problematic}....}
\end{quote}

As the $\curl$ structure is the key for the definition and properties of helicity, it is not hard to imagine that the de~Rham complex plays a significant role in studying questions related to helicity. In fact, the algorithm that preserves helicity used in Example \ref{example:relaxation} is based on finite element de~Rham complexes and finite element exterior calculus. 
Methods based on de~Rham complexes naturally address both issues mentioned in the quote above:
\begin{itemize}
 \item They preserve helicity, providing a discrete Arnold inequality (thus enforcing topological barriers). 
 \item The Raviart–Thomas (RT) finite element space allows tangential discontinuities, and $\nabla \times \mathrm{RT}$ accommodates current sheets. 
 \end{itemize}
 
Topology-preserving discretisation based on the de~Rham complex sheds light on tackling computational challenges in topological hydrodynamics \cite{arnold1999topological}.

 



\section{General relativity and the Einstein equations}\label{sec:einstein}

In general relativity, spacetime geometry is governed by the Einstein equations, which are complicated and
challenging to solve: in John Wheeler’s famous summary, “spacetime tells matter how to move; matter tells
spacetime how to curve”. Since controlled experiments are impossible and analytical solutions are only
available in exceptional cases, it is essential to simulate the Einstein equations on computers, a subject
referred to as numerical relativity.
Early attempts of numerical simulations of black holes (see, e.g., \cite{smarr1977space,wilson1979numerical}) suffer from severe instabilities. Numerical relativity witnessed a breakthrough in 2005 when accurate long-term evolutions of black holes were achieved \cite{pretorius2005evolution}. Researchers realised that mathematical properties, particularly the hyperbolicity and the constraint equations, of the Einstein equations played a vital role in the success of these numerical methods. Only certain reformulations with carefully chosen coordinates and gauge conditions are strongly hyperbolic, which lead to stable numerical discretisation \cite{reula1998hyperbolic,baumgarte1998numerical,shibata1995evolution}.   Nowadays numerical computation for the Einstein equations is a fast-growing area and plays a vital role in modern astrophysics by providing templates in the detection of gravitational waves.


General relativity is governed by the Einstein equations, which are nonlinear geometric PDEs and read in compact notation:
\begin{equation}\label{eqn:einstein}
G_{\mu\nu}+\Lambda g_{\mu\nu}=\kappa T_{\mu\nu}.
\end{equation}
Here $g_{\mu\nu}$ is a spacetime metric, which is the unknown; $G_{\mu\nu}$ is the Einstein tensor obtained from $g_{\mu\nu}$; $\Lambda$ is the cosmological constant; $\kappa$ is the Einstein gravitational constant; and $T_{\mu\nu}$ is the stress–energy tensor. 

High complexity is hidden in the compact notation \eqref{eqn:einstein}. Recall that the Christoffel symbols $\Gamma_{ij}^{k}$ can be obtained from the Riemann tensor:
$$
\Gamma_{ij}^{k}=g^{k\ell}(\frac{\partial g_{\ell i}}{\partial x^{j}}+\frac{\partial g_{\ell j}}{\partial x^{i}}-\frac{\partial g_{ij}}{\partial x^{\ell}}), 
$$
and the Riemann tensor can be obtained from the Christoffel symbols:
$$
R^\ell_{\ i j k} =
\frac{\partial \Gamma^\ell_{i k}}{\partial x^j}
-
\frac{\partial \Gamma^\ell_{i j}}{\partial x^k}
+
\Gamma^\ell_{j m}\,\Gamma^m_{i k}
-
\Gamma^\ell_{k m}\,\Gamma^m_{i j}.
$$
The Ricci tensor is the trace of the Riemann tensor: $R_{ik}=R^\ell_{\ i \ell k}$; while the Einstein tensor is further obtained by modifying the trace of Ricci:
$$
G_{ik}=R_{ik}-\frac{1}{2}Rg_{ik},
$$
with the scalar curvature $R:=g^{ik}R_{ik}$. With spacetime coordinates, the metric is a 4-by-4 symmetric matrix field with the Lorenzian signature. The above calculation shows that the Einstein tensor $G_{k\ell}$ is a highly complicated nonlinear function of the metric $g_{ij}$. 

Another complication of the Einstein equations comes from the fact that the Einstein equation is a geometric PDE. One may choose different coordinates to describe the same geometry. This means that once one writes down the Einstein equation in coordinates, solutions to the PDEs will not be unique. Any $\tilde{g}_{ij}$ obtained from $g_{ij}$ by a coordinate transform will also solve the equations. {\it Gauge conditions}, or coordinates, can be fixed to select a unique solution. The properties of the PDEs depend on the chosen coordinates and gauge conditions. This will be a particularly important issue in numerical computation, as equivalent formulations on the continuous level can lead to rather different numerical performance.  
In the study of geometric dynamics of spacetime and for the purpose of computation, one often carries out a 3+1 decomposition of the spacetime into space and time components. Choosing such decompositions is often equivalent to choosing coordinates. 

The well known formulation with a spacetime decomposition is the Arnowitt--Deser--Misner (ADM) formalism, which is a Hamiltonian formulation of general relativity and plays an important role in quantum gravity.

Even though the Einstein equations have a highly complicated form, one may still observe the equations through linearisation. The linearised Arnowitt--Deser--Misner (ADM) formulation (around the Minkowski metric) is the following:
\begin{equation}\label{eqn:ADM}
  \begin{aligned}
   \ten \gamma_{tt} + \mathcal{S}\inc\ten\gamma-2\hess\alpha-2\deff\vec \beta_t &= 0 \\
    \div \mathcal{S}(\ten\gamma_t - 2\deff\vec\beta) &= 0\\
    \div\div \mathcal{S}\ten \gamma &= 0,
  \end{aligned}
\end{equation}
where  $\ten \gamma$ is the perturbation of the spatial metric, $\alpha $ is the lapse indicating the separation of spacial hypersurfaces, and 
$\vec\beta$ is the shift vector characterising the shift of each spacial hypersurface.
As before, $\mathcal{S}(\ten A) := \ten A^\top - \tr(\ten A) \ten I$.
  The lapse $\alpha $ and the shift $\vec\beta$ can be chosen arbitrarily as they reflect the freedom to choose coordinates.  The first equation of \eqref{eqn:ADM} is the evolutionary equation, and the rest are the constraint equations, which are automatically preserved by the time evolution.

Although the ADM formulation has a clear geometric and physical structure, it is not well suited for numerical discretisation in its original form, as it is not strongly hyperbolic. Various modifications of the ADM formulation have been proposed and used in numerical computation, such as the successful BSSNOK formulation \cite{alcubierre2008introduction,nakamura1987general,shibata1995evolution,baumgarte1998numerical} and the approaches based on generalised harmonic coordinates \cite{pretorius2005evolution}. The equation \eqref{eqn:ADM} allows us freedom of reformulation. In fact, one can choose proper lapse $\alpha$ and shift $\vec \beta$, and add certain forms of the constraint equations to the evolution equations to obtain better mathematical properties.

Various formulations of the Einstein equations have been studied in the literature with different use of unknowns and different spacetime decompositions etc. In particular, first-order hyperbolic formulations are attractive for numerical purposes. One may obtain formulations with a similar structure as the Maxwell equations. 
One of such examples is the {\it Einstein--Bianchi formulation} \cite{friedrich1985hyperbolicity,anderson1997einstein,choquet2008general} that has been recently investigated in various numerical works \cite{quenneville2015new,hu2021conforming2}. Next, we briefly discuss the derivation of the Einstein-Bianchi formulation. 

From the Bianchi identity, we have
$$
\nabla_{\alpha}R^{\alpha}_{~~\beta, \lambda\mu}+\nabla_{\mu}R_{\lambda\beta}-\nabla_{\lambda}R_{\mu\beta}=0.
$$ 
Using the Einstein equations $R_{\alpha\beta}=\rho_{\alpha\beta}$,
\begin{align*}
\nabla_{\alpha}R^{\alpha}_{~~\beta,\lambda\mu}=\nabla_{\lambda}\rho_{\mu\beta}-\nabla_{\mu}\rho_{\lambda\beta}.
\end{align*}
The constraint and evolutionary equations come from different components.

Define
$$
\ten {E}_{ij}=R^{0}_{ \ i0j}, \quad\ten{D}_{ij}=\frac{1}{4}\eta_{ihk}\eta_{jlm}R^{hk, lm},
$$
$$
\ten{H}_{ij}=\frac{1}{2}N^{-1}\eta_{ihk}R^{hk}_{\ \ 0j}, \quad \ten{B}_{ji}=\frac{1}{2}N^{-1}\eta_{ihk}R_{0j}^{\ \ hk}.
$$
Now $\ten E, \ten D, \ten H, \ten B$ satisfy an equation of Maxwell's type. Linearisation around the Minkowski spacetime gives
\begin{subeqnarray}\label{einstein-bianchi}
\ten B_{t}+\nabla\times \ten E&=&0, \\
\ten E_{t}-\nabla\times \ten B& =&0.
\end{subeqnarray}
Here $\ten E$ and  $\ten B$ are Traceless-Transverse (TT) tensors, i.e., symmetric, tracefree, and divergence-free matrix fields. These conditions are automatically preserved by the time evolution.

The Einstein-Bianchi system \eqref{einstein-bianchi} has a similar form as the Maxwell equations, except that the variables are TT tensors, rather than vector fields or differential forms. Correspondingly, the differential structures of \eqref{einstein-bianchi} are encoded in other complexes than the de~Rham complex. The curl of TT tensors are encoded in the conformal Hessian complex \eqref{conformal-hessian}.
The relevant Hodge wave equation (a wave version of the Hodge--Laplacian \cite{quenneville2015new}) reads 
\begin{align*}
\sigma_{t}&=\div\div \ten E, \\
\ten E_{t}&=-\dev\hess \sigma-\curl \ten B,\\
\ten B_{t}&=\sym\curl\ten E,
\end{align*}
where the variables are taken as below (here we use $\mathbb{R}$ and $ \mathbb{S}\cap \mathbb{T}$ to denote the corresponding scalar and $ \mathbb{S}\cap \mathbb{T}$-valued function spaces etc.):
\begin{equation*}
\begin{tikzcd}
\mathbb{R} {\arrow[r,  ->,  "\dev\hess", shift left=1]\arrow[r,  <-, swap, "\div\div", shift right=1]} & \mathbb{S}\cap \mathbb{T} {\arrow[r,  ->,  "\sym\curl", shift left=1]\arrow[r,  <-, swap, "\curl", shift right=1]} &  \mathbb{S}\cap \mathbb{T}
\end{tikzcd}
\end{equation*}
$$
\hspace{-0.5cm}\sigma \hspace{1.5cm}  \ten E \hspace{1.5cm} \ten B
$$
As discretising conformal complexes leads to significant challenge, for numerical purposes, one may relax some symmetry conditions and let them free propagate (on the continuous level, if these conditions hold at the initial time, then they are preserved at any later time -- this property may not hold on the discrete level precisely though, due to discretisation errors). Two such formulations are based on the de~Rham complex and the Hessian complex (or the divdiv complex) respectively:
\begin{itemize}
\item Option 1 for relaxed symmetries: using the de~Rham complex
\begin{align*}
\vec\sigma_{t}&=\div \ten E, \\
\ten E_{t}&=-\grad\vec \sigma-\curl \ten B,\\
\ten B_{t}&=\curl\ten E.
\end{align*}
\begin{equation*}
\begin{tikzcd}
\mathbb{V} {\arrow[r,  ->,  "\grad", shift left=1]\arrow[r,  <-, swap, "-\div", shift right=1]} & \mathbb{M}  {\arrow[r,  ->,  "\curl", shift left=1]\arrow[r,  <-, swap, "\curl", shift right=1]} & \mathbb{M} 
\end{tikzcd}
\end{equation*}
$$
\hspace{-0.cm}\vec \sigma \hspace{1.2cm}  \ten E \hspace{1.2cm} \ten B
$$
\item Option 2 for relaxed symmetries: using the divdiv or the Hessian complex
\begin{align*}
\sigma_{t}&=\div\div \ten E, \\
\ten  E_{t}&=-\hess \sigma-\sym\curl \ten B,\\
\ten B_{t}&=\curl\ten E.
\end{align*}
\begin{equation*}
\begin{tikzcd}
\mathbb{R} {\arrow[r,  ->,  "\hess", shift left=1]\arrow[r,  <-, swap, "\div\div", shift right=1]} &  \mathbb{S}  {\arrow[r,  ->,  "\curl", shift left=1]\arrow[r,  <-, swap, "\sym\curl", shift right=1]} &  \mathbb{T} 
\end{tikzcd}
\end{equation*}
$$
\hspace{0.2cm}\sigma \hspace{1.2cm}  \ten E \hspace{1.2cm} \ten B
$$
 \end{itemize}

 Before closing this section, we recall some other connections between general relativity and differential complexes. 
 
 In electromagnetism and Maxwell equations, one often uses a magnetic potential $\vec A$ satisfying $\nabla\times \vec A=\vec B$. At least in the linearised case, it is not hard to imagine that one can obtain potential fields for the Einstein equations. For example, the conformal deformation complex \eqref{conformal-complex-smooth} exactly claims that TT tensors can be written as the Cotton-York tensors of symmetric and traceless tensors when there is no nontrivial cohomology. Einstein constraint equations for linearised gravity were investigated in \cite{beig2020linearised}, where the authors developed various complexes on Einstein manifolds from scratch.  The ``scalar'', ``momentum'', and ``conformal'' complexes developed in \cite{beig2020linearised} correspond to the $\div\div$, $\curl\div$, and the elasticity complex in this paper and in the BGG construction \cite{arnold2021complexes,vcap2023bgg}, respectively.

Another connection roots in the analogy between general relativity and generalised continuum models. General relativity and the Cosserat models are the two main motivations for Cartan's development of the concept of torsion \cite{scholz2019cartan}. Quoting Kr\"oner in this habilitation \cite[Section 18]{kroner1960general}, 
\begin{quote}
Every expert on general relativity theory who studies the general continuum theory of dislocations and internal stresses will recognise the great similarity between the two theories, which will considerably enhance his understanding of the latter one. ...  we believe that an ongoing investigation of the connections between general relativity and the general continuum theory of dislocations and internal stresses can be of considerable benefit to both theories.
\end{quote}
The Einstein-Cartan theory can be regarded as a space-time defect theory \cite{ruggiero2003einstein}.



\section{Graphs, networks, and data analysis}\label{sec:graph}

Finite elements are piecewise polynomials on a triangulation. The lowest order Whitney $k$-forms have degrees of freedom on $k$-cells. These degrees of freedom reflect a discrete topological structure. Particularly, the cohomology of the discrete de Rham complex is isomorphic to the continuous version.  Such topological structures can be established on other discrete structures as well. In this section, we briefly review the construction on graphs. 

Let $G = (V,E)$ be an undirected graph with vertex set $V = \{1,2,\dots,n\}$ and edge set $E \subset \{[i,j]: i<j\}$. 

\begin{definition}[Cliques]
Let \(G = (V,E)\) be an undirected graph with vertex set \(V\) and edge set \(E \).
A subset of vertices \(Q \subseteq V\) is called a \emph{$k$-clique} if, for $Q$ has $k$ vertices and every pair of distinct vertices \(u, v \in Q\), the edge \((u,v)\) is in \(E\). 
\end{definition}
 We denote $K_{k}$ to be the set of $k$-cliques. One can generalise the definition of boundary operators to cliques: 
 $$
 \partial [i_{1}, \cdots, i_{k}]=\sum_{j=1}^{k}(-1)^{j+1}[i_{1}, \cdots, \widehat{i_{j}}, \cdots,  i_{k}].
 $$
 Consequently, cliques form a simplicial complex. Similarly, 
we denote by $C^q(G,\mathbb{R})$ the space of real-valued $q$-cochains on $G$. Here a $q$-cochain $f$ is a skew-symmetric function on $q$-cliques, i.e., 
$$
f  ([i_{1}, \cdots, i_{k}, \cdots,   i_{\ell}, \cdots, i_{q}])=-f([i_{1}, \cdots, i_{\ell}, \cdots,   i_{k}, \cdots, i_{q}]).
$$ 
Note that in the definition above, a $q$-clique or a $q$-cochain involves $q$ vertices. In simplicial homology theory, this is often referred to as a $(q-1)$-simplex or $(q-1)$-cochain. The dimension of the $q$-th homology space $\mathcal{H}_{q}:=\ker(\partial_{q})/\ran(\partial_{q+1})$ is called the $q$-th Betti number.

Similar to simplicial cohomology, the \emph{coboundary operator} $d^q: C^q(G,\mathbb{R}) \;\longrightarrow\; C^{q+1}(G,\mathbb{R})$ is defined by 
$$
d^{q}f([i_{1}, \cdots, i_{p+1}])=\sum_{j=1}^{q+1}(-1)^{j+1}f ([i_{1}, \cdots,\widehat{i_{j}}, \cdots i_{q+1}]).
$$
Assume that the space of $p$-cochains are equipped with an inner product:
$$
\langle f, g\rangle:=\sum_{i_{1}<\cdots<i_{q}}\omega_{i_{1}, \cdots, i_{q}}f([i_{1}, \cdots, i_{q}])g([i_{1}, \cdots, i_{q}]),
$$
where $\omega_{i_{1}, \cdots, i_{q}}$ is the weight on the $q$-clique $[i_{1}, \cdots, i_{q}]$. Then 
the \emph{codifferential operator} $\delta^{q} : C^{q}(G,\mathbb{R}) \to C^{q-1}(G,\mathbb{R})$ is defined as the adjoint of $d^{q-1}$ with respect to the given inner product, i.e., 
\[
\langle d^{q} f, g\rangle_{C^{q+1}} = \langle f,\delta^{q+1} g\rangle_{C^q},\quad f \in C^q(G,\mathbb{R}), ~~ g\in C^{q+1}(G,\mathbb{R}),
\]
where two levels of inner product are used. With the coboundary (exterior derivative) and the codifferential operators, one can define the $q$-th Hodge--Laplacian operator as usual:
\[
\Delta^q =d^{q-1}\delta^{q} + \delta^{q+1}d^q.
\]
Then many definitions and properties follow as the the case of simplicial cohomology. For example, we have a complex
\begin{equation*}
\begin{tikzcd}
\cdots \arrow{r} & C^{q-1}(G, \mathbb{R}) \arrow{r}{d^{q-1}} & C^{q}(G, \mathbb{R}) \arrow{r}{d^{q}} & C^{q+1}(G, \mathbb{R}) \arrow{r} &  \cdots.
\end{tikzcd}
\end{equation*}
Correspondingly, we have the Hodge decomposition
\[
C^q(G,\mathbb{R}) = \ran(d^{q-1}) \oplus \ran(\delta^{q+1}) \oplus \ker(\Delta^q).
\]

Cliques and cochains on graphs can be represented by vectors and matrices. Then the coboundary and the codifferential operators can be expressed by matrix operations, enabling computation with these concepts. Problems of significant mathematical and scientific importance arise and can be potentially investigated numerically. Below, we show an example of {\it topology of random graphs}. 

Random graphs has many important applications. There are various models incorporating randomness in graphs. A classical one is the Erdős--Rényi model, where each edge has a fixed probability $p$ of being present or absent, independently of the other edges. Topology and spectral properties of random graphs are of great interest. A celebrated result in the theory of random graphs is the Erdős--Rényi theorem on the threshold for connectivity \cite{erdds1959random}. 
\begin{theorem}[Erdős and Rényi]
If $p=(\log n+\omega(n))/n$ and $\omega\rightarrow \infty$ as $n\rightarrow \infty$, then $G(n, p)$ is almost always connected. If $\omega(n)\rightarrow \infty$, then $G(n, p)$ is almost always disconnected. 
\end{theorem}
The Erdős--Rényi theorem can be viewed as a statement on the $0$-th homology of the graph. A natural question would be its higher dimensional analogues, i.e., how the $k$-th Betti number $\beta_{k}$ behaves as $p$ changes. Intuitively, one can imagine that for small and large $p$, $\beta_{k}$ is almost always zero (see Figure \ref{fig:betti}). This is because for small $p$, the connection is so week that no $k$-dimensional cliques are formed; while for large $p$, the graph is highly connected and is therefore topologically trivial.  Further questions include when the Betti number starts to grow and when to vanish (points of \emph{phase transition}). Various results in this direction can be found in \cite{kahle2009topology}. For example, it was shown that if $p=n^{\alpha}$ with $\alpha<1/k$, then the $k$-th homology is almost always zero \cite{kahle2009topology}. We also refer to  \cite{kahle2014topology} for a survey of topology of random simplicial complexes.

\begin{figure}
	   \includegraphics[width=0.49\textwidth]{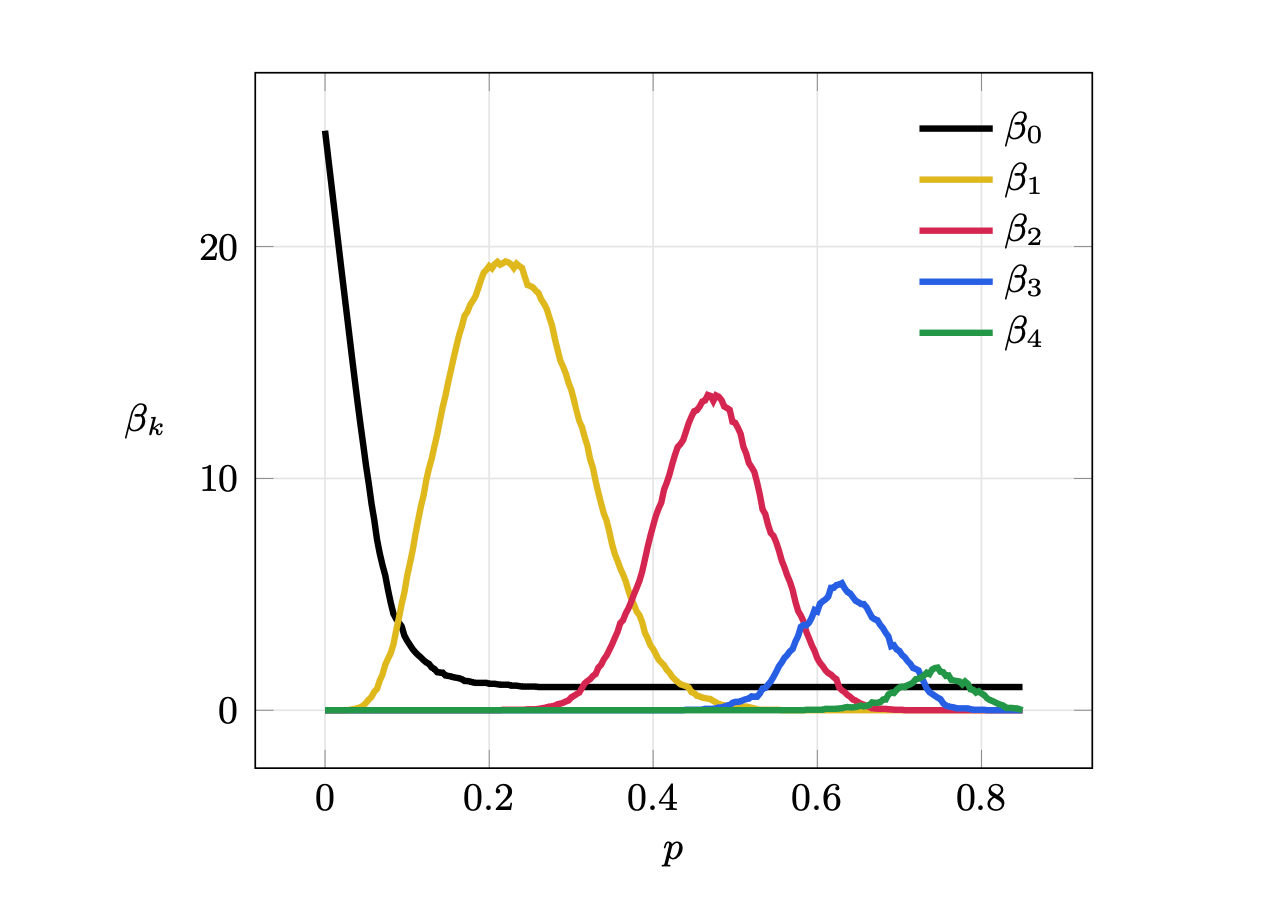}
	   \caption{Average Betti numbers (\(\beta_0\) to \(\beta_4\)) of random graphs with 25 vertices as a function of edge probability \(p\). For small \(p\), the graphs are sparse, and higher Betti numbers vanish. As \(p\) increases, Betti numbers grow, peak, and decrease to zero when the graph becomes fully connected. The figure and computation are part of L\'az\'ar Bert\'ok's MSc thesis \cite{lazar-thesis}.}
	   \label{fig:betti}
\end{figure}

In Singer's words, the spirit of investigating topological properties of random graphs is summarised as follows: \cite{kahle2014topology,raussen2005interview}:
\begin{quote}
I predict a new subject of statistical topology. Rather than
count the number of holes, Betti numbers, etc., one will be
more interested in the distribution of such objects on noncompact manifolds as one goes out to infinity.\\
\null\hfill -- Isadore Singer,  \cite{raussen2005interview}
\end{quote}
``Statistical topology'' leaves open questions to answer. From a computational perspective, there are open questions such as improving the computational efficiency for finding cliques (see Figure \ref{fig:betti}). 

In many real-world applications, such as social networks and biological networks, {interactions} often go beyond simple pairwise links. This is a motivation for developing \emph{hypergraphs} \cite{bick2023higher,antelmi2023survey} and higher dimensional cliques and discrete differential forms.  Back to our discussions on (co)homology of graphs (not necessarily random graphs), we mention several applications in data sciences, and particularly, in Topological Data Analysis (TDA).



  \begin{example}[Ranking problem]
 A classical application of Hodge theory in data science is the ranking problem, see, e.g., \cite{austin2012s,jiang2011statistical}.  Suppose that we have nodes $x_{i}, i=1, \cdots, N$ and want to rank them. A voter may vote for each pair of the nodes. For example, voter $\alpha$ may assign a number $Y_{ij}^{\alpha}$ to the pair of node $i$ and node $j$. The node $i$ is preferred If $Y_{ij}^{\alpha}>0$, and $j$ is preferred if $Y_{ij}^{\alpha}<0$. The voter may not compare all the nodes. We set $Y_{ij}^{\alpha}=0$ of $i$ and $j$ are not compared by $\alpha$. As a convention, we require $Y_{ij}^{\alpha}=-Y_{ji}^{\alpha}$. We may do an average for $\alpha$, leading to a pairwise ranking function $Y_{ij}$. We still have $Y_{ij}=-Y_{ji}$. Therefore $Y_{ij}$ is an alternating function on 2-cliques, which fits with the definition of 2-cochains on graphs. For a consistent ranking, we should have $Y_{ij}+Y_{jk}=Y_{ik}$, or, in another form, 
 \begin{equation}\label{Y-curl-free}
 Y_{ij}+Y_{jk}+Y_{ki}=0.
 \end{equation}
 This is exactly the exterior derivative ($\curl$) of $Y$ acting on the 3-clique $[i, j, k]$. Therefore the $\curl$ of $Y$ measures the local inconsistency of the ranking. The ranking is consistent if $Y_{ij}$ is obtained by a pointwise comparison, i.e., $Y_{ij}=s_{i}-s_{j}$ for some function $s$ defined at vertices (here $s_{i}:=s([i])$ is the evaluation of $s$ on 1-cliques), i.e., $Y$ is a gradient field. The above discussions give an intuitive interpretation for the Hodge decomposition:
 \begin{equation}\label{graph-hodge}
 Y=\grad s+\curl z +H,
 \end{equation}
 where $H$ is the harmonic form which is both $\curl$-free and $\div$-free. Now $\grad s$ and $H$ are $\curl$-free and thus are the consistent components; while $\curl Y=\curl\curl Z$ characterises the local consistency. The gradient component $\grad s$ is the best place where we can rank the nodes. Therefore as an algorithm, one computes the Hodge decomposition \eqref{graph-hodge}, and $s$ can be regarded as the score for the ranking. 
 
 We mention similarities between continuum defects discussed in Section \ref{sec:solid} and the ranking problem, which seem to be two rather distinct areas. In fact, in continuum mechanics, for compatible deformations, we have compatibility conditions expressed by $\curl$-free or $\inc$-free conditions (see Section \ref{sec:solid}). In general, we use the violation of these conditions to model continuum defects. In the ranking problem, consistent ranking satisfies the $\curl$-free condition \eqref{Y-curl-free}. In general, we can use $curl$ of the edge flow to measure the inconsistency. The analogue between these two different areas is exactly one of the many facets of the common cohomological structure, or the discrete topological structure, which is emphasised in this article. 
 \end{example}
 



 \begin{example}[Persistent homology and topological data analysis]
 Although topological data analysis (TDA) often uses homology directly, the harmonic space $\ker(\Delta^{p})$ is intimately related to homology groups. Therefore, understanding and computing the discrete de~Rham complex on graphs and their cohomologies can be relevant to topological feature detection (e.g.\ loops, holes) in high-dimensional data.
 
In typical contexts of persistent homology, one can build a sequence of complexes $C_{\alpha}^{\bs}$ with a family of parameters $\alpha$. For example, $\alpha$ can be distance thresholds (i.e., two points with a distance less than $\alpha$ will be considered to be connected to each other). Then, one may track events of homology classes as $\alpha$ changes. This is a quite general idea, though.
 For example, one can also track cohomological information, such as the Betti numbers. We refer to \cite{su2024persistent,basu2024harmonic,cang2020persistent,liu2021hypergraph} for recent discussions on combining the idea of persistent homology with Hodge--Laplacian and cohomology theories. 
 \end{example}

 We refer to \cite{lim2020hodge} and the references therein for further applications of the Hodge--Laplacian on graphs. 
 
 \begin{remark}[Applications of sheaf theory in finite elements and machine learning]
As another example illustrating the unified structures underlying finite elements, data, and networks, we mention sheaf theory \cite{bredon1997sheaf}. In geometry, a sheaf assigns to every open set of a topological space a collection of data, together with restriction maps that ensure compatibility of local data on overlaps. In essence, a sheaf is a tool for glueing local information into a global picture. In finite element theory, the problem of constructing smooth finite elements or splines can be formulated as matching derivatives of piecewise polynomials across cell boundaries, which is inherently related to the notion of sheaves \cite{christiansen2010finite,hu2025sharpness}. In machine learning, new network architectures can be obtained by introducing restriction maps that gather and align information across graph nodes \cite{ayzenberg2025sheaf}. This example again demonstrates universal structures in numerical methods and data architecture, and indicates the potential for establishing further connections.
\end{remark}

 \begin{remark}[Applications of sheaf theories in finite elements and machine learning]
 As another example demonstrating the unified structures in finite elements and data and network theories, we mention sheaf theories \cite{bredon1997sheaf}. Sheaf theory in geometry assigns to every open set of a topological space a collection of data, together with restriction maps that ensure local data is compatible on overlaps. A sheaf is a tool for glueing local information into a global picture. In the finite element theory, the problem of constructing smooth finite elements or splines can be formulated as matching derivatives of piecewise polynomials across the boundary of cells, therefore related to sheaves in essence \cite{christiansen2010finite,hu2025sharpness}; in machine learning, one may obtain new network architectures by introducing a map gathering information on graph nodes \cite{ayzenberg2025sheaf}. This example further demonstrates the universal structures in numerical methods and data applications, and the potential in establishing new connections.  
 \end{remark}


\section{Finite elements: a glimpse}

The article discussed various topics from a differential complex and cohomological point of view. As complexes and cohomology encode notions of existence, uniqueness, stability and rigidity, it is not surprising that differential complexes play a fundamental role in these seemingly unrelated topics, justifying {\it many facets of cohomology}. Each problem has unique algebraic structures, and their role differs from case to case. Therefore, a programme pursued in recent years is to investigate algebraic and geometric structures systematically and {\it discretise the whole theory}. 


Structure-preserving discretisation serves as an important motivation for structure-aware formulations, or more precisely, clarifying the underlying differential structures. Discretisation or discrete theories are not the focus of this article. In particular, the whole machinery of discretising PDEs with Finite Element Exterior Calculus is not covered in this paper.  However, we provide a glimpse of some recent developments of finite element versions of differential complexes. 

For the de~Rham complex, the Whitney forms \cite{hiptmair2001higher,whitney2012geometric,bossavit1988whitney} 
are widely accepted as the canonical discretisation of differential forms. 
\begin{figure}
\begin{center}
\includegraphics[width=4in]{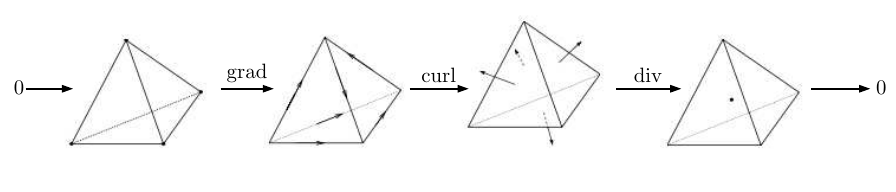} 
\end{center}
\caption{A finite element de Rham complex in $\mathbb{R}^{3}$ consisting of Whitney forms. Here, each slot denotes a finite element space (piecewise polynomials with certain continuity on a triangulation).  The four spaces are, respectively, the Lagrange element (classical $C^{0}\mathcal{P}_{1}$ finite element, or Courant triangle), the Nédélec element \cite{Nedelec.J.1980a}, the Raviart--Thomas element \cite{Raviart.P;Thomas.J.1977a}, and piecewise constants. The first and last spaces discretise scalar functions, while the two middle spaces discretise vector-valued functions. The degrees of freedom are located at vertices, edges, faces, and inside tetrahedra, reflecting the discrete topological structure. The differential operators $\grad$, $\curl$, and $\div$ map each finite element space to the next.  
\\
 Bossavit \cite{bossavit1988whitney} realised that these spaces in fact coincide with the discrete differential forms defined from a topological viewpoint by Whitney (referred to as the Whitney forms \cite{whitney2012geometric}). Hiptmair \cite{hiptmair1999canonical, hiptmair2001higher} and Arnold, Falk, and Winther \cite{Arnold.D;Falk.R;Winther.R.2006a, Arnold.D;Falk.R;Winther.R.2010a} systematically developed Finite Element Exterior Calculus. 
\\  In discrete topology or Discrete Exterior Calculus \cite{hirani2003discrete}, variables are purely discrete, defined on discrete topological entities (vertices, edges...). In Finite Element Exterior Calculus, these degrees of freedom uniquely determine a polynomial field inside the tetrahedron ({\it unisolvency}). Therefore, finite elements, as piecewise functions, provide a more natural setup for a rigorous mathematical analysis than merely discrete definitions.}
\label{fig:whitney}
\end{figure}
In the lowest order case (Figure \ref{fig:whitney}), $k$-forms are discretised on $k$-cells; the finite elements are unisolvent (the cochains are extended to the interior of simplices to polynomials); and the discrete forms have continuous traces across the faces of the cells; and importantly, the cohomology of the discrete de~Rham complex is isomorphism to the continuous version, reflecting topology of the underlying domain. A periodic table of finite elements \cite{arnold2014periodic} incorporates the Lagrange, Nédélec, and Raviart-Thomas finite elements and their high-order and high-dimensional versions.

In this article, we have shown that different problems involve different differential structures. A natural but challenging question is to seek analogues of Whitney forms for complexes beyond the classical de~Rham sequence. This pursuit is motivated by important applications in continuum mechanics, differential geometry, and general relativity. In recent years, considerable effort has been devoted to discrete BGG complexes (see, e.g., \cite{chen2021geometric,chen2020discrete,chen2021finite,chen2020finite,chen2021finite2,hu2021conforming,hu2021conforming2,christiansen2020discrete,Arnold2006a,arnold2008finite,christiansen2019finite}), to make them computable. Despite progress such as tensor-product constructions \cite{bonizzoni2025discrete} and conforming finite element complexes
\cite{chen2024h,chen2025complexes}, a systematic and canonical discretisation was still missing, until some recent progress.



 As the development of canonical finite elements for the de~Rham complex, some ingredients for addressing this problem have already been in place for specific applications.  Particularly, Christiansen \cite{christiansen2011linearization} interpreted Regge calculus, a discrete scheme for quantum and numerical gravity, as a finite element and fitted it into a discrete elasticity complex. There is perhaps little hesitation in appreciating the construction in \cite{christiansen2011linearization} as the canonical discretisation for the elasticity complex. The discrete metric is piecewise constant (piecewise flat manifolds), and correspondingly, the discrete curvature is distributional with conic singularities (represented by an angle deficit). The Christiansen-Regge complex is formally self-adjoint and the degrees of freedom enjoy a simple and elegant form. In another direction of research, Schöberl and collaborators systematically investigated distributional finite elements for solving problems in continuum mechanics and a posteriori error estimators \cite{braess2008equilibrated,pechstein2011tangential,gopalakrishnan2020mass}.  These finite elements have relaxed continuity, and de~Rham's currents and generalisations appear in derivatives and thus in the complexes. 

By gathering these results and incorporating some new ingredients, the recent work \cite{hu2025finite} investigated an extended version of the periodic table of finite elements discretising $\ell$-form-valued $k$-forms $\Lambda^{k, \ell}$ (double forms) with algebraic symmetries encoded in the BGG diagrams  \cite{arnold2014periodic} (see Figure \ref{fig:table} for the lowest-order case in 3D). 
The table shows several features. First, its discrete cohomology remains isomorphic to the continuous version. Second, in the diagram, the lower triangular region ($k \leq \ell$) uses function-based discretisations, while the upper triangular region ($k > \ell$) involves Dirac delta distributions. The transition from functions to distributions within each complex occurs in the “zig-zag” step of the BGG construction. Moreover, the Hessian complex and the div\,div complex form a dual pair, while the elasticity complex is self-adjoint. In the function region ($k \leq \ell$), $\ell$-form valued $k$-forms are discretised on $k$-simplices; in the distribution region ($k > \ell$), they are attached to the dual $k$-simplices. This framework is not limited to three dimensions but can be extended to arbitrary dimensions and polynomial orders \cite{hu2025finite}, providing a flexible tool for a wide range of numerical applications.  
 
The classical de~Rham complex \cite{arnold2014periodic} is a very special case within this extended periodic table: in that complex all spaces are \emph{classical finite elements (piecewise polynomials)}. In the more general case, the first few spaces of a complex are classical finite elements, while the later spaces are Dirac measures. 

Symmetries and decompositions of tensors (particularly for those arising in geometry, such as the metric and various notions of curvature) are encoded in finite element spaces and fit in discrete versions of BGG diagrams, extending results on differential forms (fully skew-symmetric tensors). This yields a FEEC version of discrete curvature in any space dimensions.

The tensorial finite elements, particularly the Regge element and generalisations, provide a new perspective for investigating discrete differential geometry. Recent progress on approximating geometric quantities with Regge-like finite elements can be found in \cite{gawlik2020high,gawlik2023finite,gopalakrishnan2024improved,gawlik2024finite,gopalakrishnan2023analysis}. Extending the Christiansen--Regge complex to the twisted complex, a pattern connected to Riemann--Cartan geometry emerged in \cite{christiansen2023extended}. In addition to being theoretically elegant, these intrinsic finite elements also shed light on challenging computational problems with complex geometry, such as shells \cite{neunteufel2021avoiding}.

A different perspective was recently given in \cite{berchenko2025finite}. 

We also refer to \cite{khu-review} for a further review of these {\it distributional} or {\it intrinsic finite elements}. However, detailed discussions of these advances are beyond the scope of this paper.

\section{Conclusion}

In this paper, we surveyed the role of differential complexes, exactness, and cohomology across various areas. From a mathematical perspective, the existence, uniqueness, and stability of differential equations are encoded in the notion of complexes. This echoes the fundamental roles that differential complexes play:
\begin{itemize}
\item  {Analysis:} Exactness and finite-dimensional cohomology imply that linear differential operators have closed range. As a result, many analytic results, such as Poincar\'e--Korn inequalities, follow. 
\item  {Continuum mechanics:} Differential complexes and the BGG machinery provide a systematic framework for encoding compatibility conditions, defects (incompatibility), and mixed-dimensional structures in continuum modeling. 
\item  {Fluid mechanics and magnetohydrodynamics:} Vortex and knot structures of field lines can be described in terms of de~Rham complexes. 
\item  {Differential geometry:} Fundamental objects such as the metric, curvature tensors (Gauss, Ricci, Riemann), Cotton--York, and torsion tensors naturally fit into sequences derived from the BGG machinery. Moreover, results such as rigidity and special cases of fundamental theorems of Riemannian geometry can be expressed as the exactness of nonlinear complexes. 
\item  {General relativity:} The Einstein equations can be reformulated into first-order systems, mimicking the structure of electromagnetism and the Maxwell equations. BGG complexes encode tensor symmetries and differential structures, playing a role analogous to the de~Rham complex for electromagnetism. 
\item  {Graphs and networks:} Discrete topological and geometric structures arise naturally in graphs and networks. 
\end{itemize} 

The de~Rham complex is purely topological, while other complexes derived from the BGG construction may require additional structures. For example, the elasticity complex requires a background metric. This both enriches and complicates the theory. Curved or nonlinear cases remain open for further investigation, and new applications can be expected. For example, replacing exterior derivatives $d^{\bs}$ by covariant exterior derivatives $d^{\bs}_{\nabla}$ in general destroys the complex property, since 
\[
d_{\nabla}^{k+1}\circ d_{\nabla}^{k}\neq 0
\]
involves nonvanishing curvature terms. Nevertheless, even in this case, differential sequences (rather than full complexes) should provide useful structures for establishing analytic results.  

Nonlinearity introduces additional challenges, as continuous structures may be lost in discretisation. For instance, every nonnegative smooth function can be written as the square of another function ($\forall w \ge 0, \, \exists u = \sqrt{w}$, such that $w = u^{2}$). This property is expressed by the exactness of the sequence
\[
 \begin{tikzcd}[column sep=large]
 0 \arrow{r} & C^{\infty} \arrow{r}{u\mapsto u^{2}} & C^{\infty}_{+} \arrow{r} & 0.
 \end{tikzcd}
\]
However, the same does not hold for polynomial spaces; the polynomial analogue
\[
 \begin{tikzcd}[column sep=large]
 0 \arrow{r} & \mathcal{P}_{k} \arrow{r}{u\mapsto u^{2}} & \mathcal{P}_{2k}^{+} \arrow{r} & 0 
 \end{tikzcd}
\]
is not exact. Furthermore, it remains to systematically establish nonlinear results, such as the nonlinear Korn inequalities \cite{ciarlet2006nonlinear}, from a nonlinear complex perspective.

These examples illustrate the broader potential and questions of differential complexes, extending beyond the topics presented in this paper.  

Finally, we did not discuss numerical analysis in detail, since the development of Finite Element Exterior Calculus \cite{arnold2018finite,Arnold.D;Falk.R;Winther.R.2006a,Arnold.D;Falk.R;Winther.R.2010a} over the past decades calls for a separate review.

\section*{Acknowledgement}

The work of KH was supported by a Royal Society University Research Fellowship (URF$\backslash$R1$\backslash$221398), an ERC Starting Grant (project 101164551, GeoFEM), and a Royal Society International Exchanges Grant (IEC$\backslash$NSFC$\backslash$233594). Views and opinions expressed are however those of the authors only and do not necessarily reflect those of the European Union or the European Research Council. Neither the European Union nor the granting authority can be held responsible for them.

Part of the article is based on short courses that the author gave at Institut Montpelliérain Alexander Grothendieck (IMAG), Université de Montpellier, hosted by ERC NEMESIS project PIs Daniele Di Pietro and Jérôme Droniou, and at University of Luxembourg (Spotlights in Computational Physics and Engineering), hosted by Adam Sky. The author would like to thank the invitations. The author also thanks anonymous reviewers for their helpful comments and Yuechen Zhu for reading a first version of the article with feedbacks.

\begin{figure}[ht]
\flushleft
\begin{tikzpicture}

  \node (deRham3D1) at (-2,0) {\includegraphics[width=0.7\textwidth]{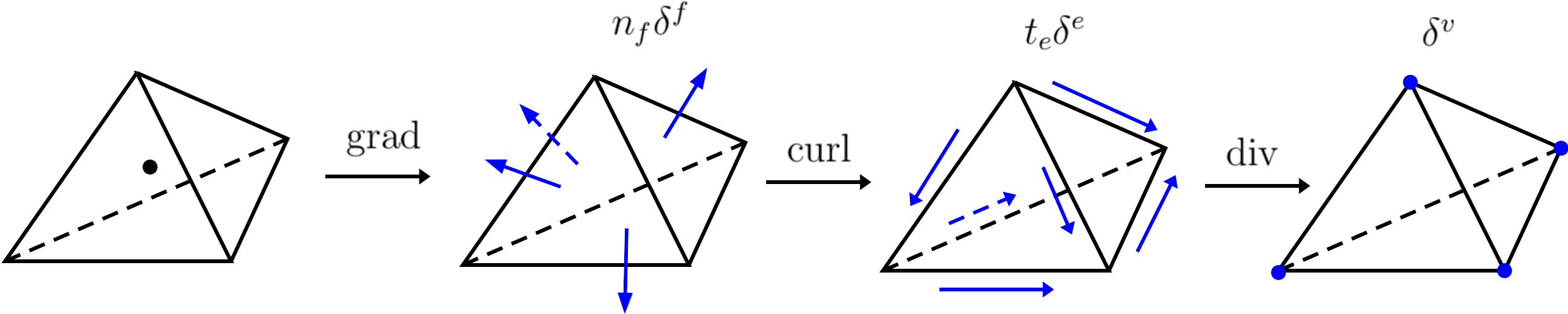}};
  \node at (-5,0.7) {\color{purple}$\Lambda^{0,0}$}; 
  \node at (-2.4,0.7) {\color{purple}$\Lambda^{1,0}$}; 
  \node at (0.5,0.7) {\color{purple}$\Lambda^{2,0}$}; 
  \node at (3.2,0.7) {\color{purple}$\Lambda^{3,0}$}; 
    \node at (5.4,0.) {\color{purple} \small distributional de Rham}; 
        \node at (5.4,-0.4) {\color{purple} \small Braess-Sch\"oberl 2008}; 
                \node at (5.4,-0.8) {\color{purple} currents};
                                \node at (1.7,1.8) {\color{blue} \textbf{distributions, Dirac measures}}; 
  \node (hessian) at (-2,-2.5) {\includegraphics[width=0.7\textwidth]{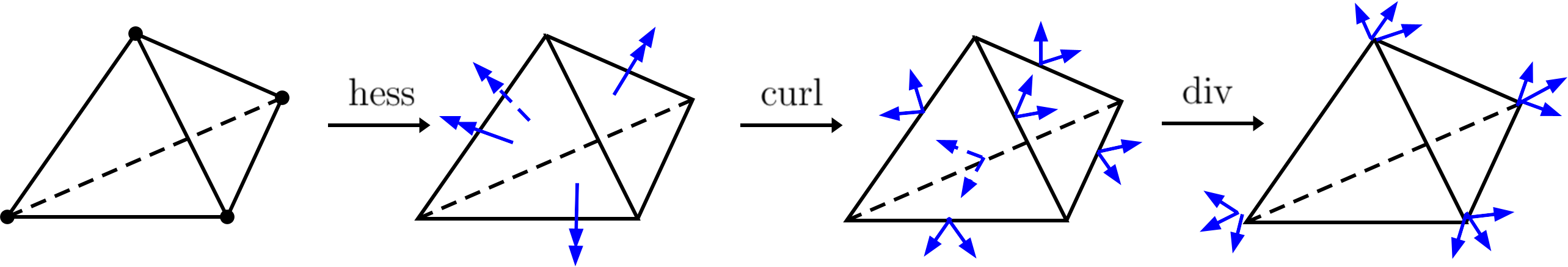}};
  \node at (-5,-1.7) {\color{purple}$\Lambda^{0,0}$}; 
  \node at (-2.4,-1.7) {\color{purple}$\Lambda^{1,1}$}; 
  \node at (0.5,-1.7) {\color{purple}$\Lambda^{2,1}$}; 
  \node at (3.2,-1.7) {\color{purple}$\Lambda^{3,1}$}; 
    \node at (5.4,-2.3) {\color{purple}Hu-Lin-Zhang 2025}; 
    
  \node (regge0) at (-2,-5) {\includegraphics[width=0.7\textwidth]{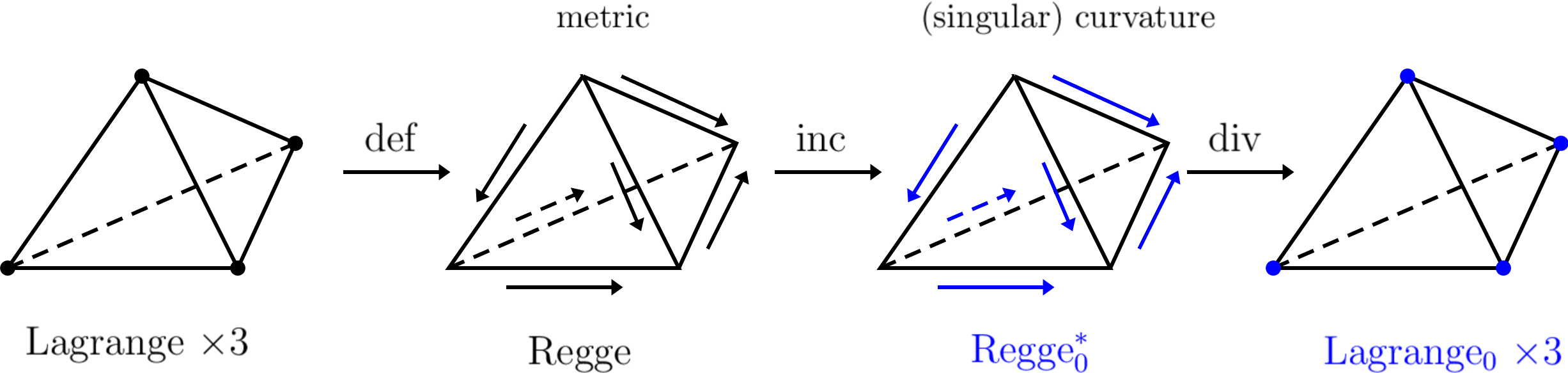}};
  \node at (-5,-4.2) {\color{purple}$\Lambda^{0,1}$}; 
  \node at (-2.4,-4.2) {\color{purple}$\Lambda^{1,1}$}; 
  \node at (0.5,-4.2) {\color{purple}$\Lambda^{2,2}$}; 
  \node at (3.2,-4.2) {\color{purple}$\Lambda^{3,2}$}; 
      \node at (5.4,-4.6) {\color{purple}Regge, Christiansen}; 
            \node at (5.4,-5.) {\color{purple}1961, 2011}; 
                        \node at (5.4,-5.4) {\color{purple}metric, curvature}; 
\node at (-5,-12) {\color{red} \textbf{finite elements, p.w. polynomials}};
  \node (divdiv) at (-2,-7.5) {\includegraphics[width=0.7\textwidth]{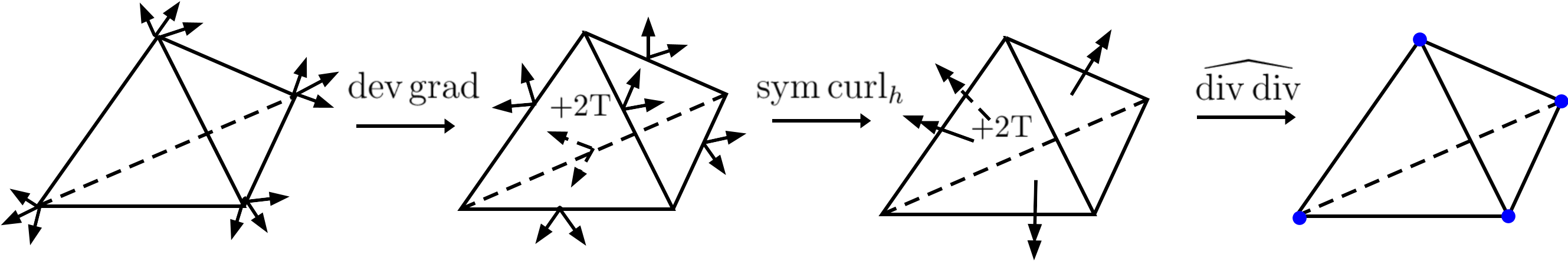}};
  \node at (-5,-6.7) {\color{purple}$\Lambda^{0,2}$}; 
  \node at (-2.4,-6.7) {\color{purple}$\Lambda^{1,2}$}; 
  \node at (0.5,-6.7) {\color{purple}$\Lambda^{2,2}$}; 
  \node at (3.2,-6.7) {\color{purple}$\Lambda^{3,3}$}; 
        \node at (5.6,-7.2) {\color{purple} \small Hellan, Herrmann, Johnson }; 
                \node at (5.4,-7.6) {\color{purple}\small Pechstein-Sch\"oberl 2011}; 
            \node at (5.4,-8) {\color{purple} \small Hu-Lin-Zhang 2025}; 
  \node (deRham3D) at (-2,-10) {\includegraphics[width=0.7\textwidth]{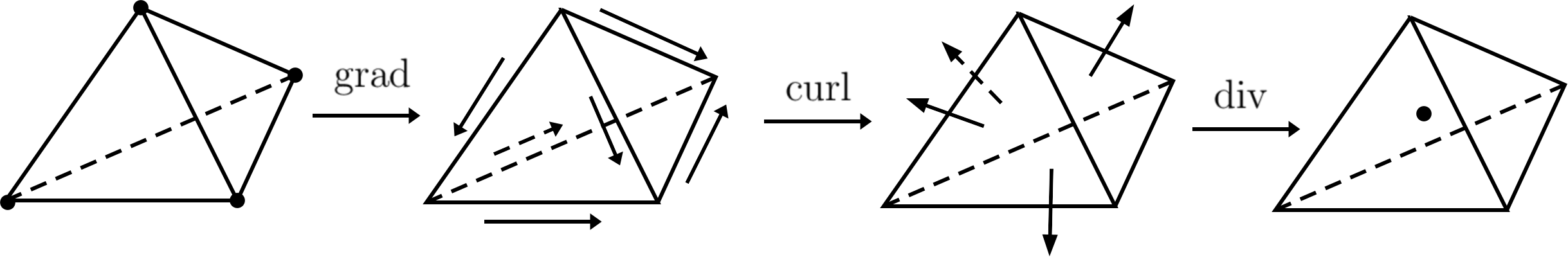}};
  \node at (-5,-9.2) {\color{purple}$\Lambda^{0,4}$}; 
  \node at (-2.4,-9.2) {\color{purple}$\Lambda^{1,4}$}; 
  \node at (0.5,-9.2) {\color{purple}$\Lambda^{2,4}$}; 
  \node at (3.2,-9.2) {\color{purple}$\Lambda^{3,4}$}; 
          \node at (5.8,-9.5) {\color{purple}\small N\'ed\'elec, Raviart--Thomas }; 
                \node at (5.4,-9.9) {\color{purple}\small 1980, 1977}; 
            \node at (5.4,-10.3) {\color{purple}\small Whitney forms 1957}; 
                        \node at (5.4,-10.7) {\color{purple} \small Bossavit, Hiptmair}; 
                        \node at (5.4, -11.1) {\color{purple} \small  Arnold-Falk-Winther};
  \draw[line width=2pt, red!30, rounded corners=5pt,dotted] (-7.8,-1.4) -- (-7.8,-11) -- (3.8,-11) -- (3.8, -9.7) -- (-6.6, -1.4) -- cycle;
  \draw[line width=2pt, blue!30, rounded corners=5pt, dotted] (-7.8,1.2) -- (3.8,1.2) -- (3.8,-8.4) -- (2.6, -8.4) -- (-7.8, -0.1) -- cycle;
\end{tikzpicture}

\caption{An extended periodic table of finite elements for form-valued forms in 3D \cite{hu2023distributional,hu2025finite}. The first row is the distributional de Rham complex (dual Whitney forms) \cite{braess2008equilibrated,licht2017complexes}; the last row consists of Whitney forms; the elasticity complex is discretised by the Christiansen--Regge complex \cite{christiansen2011linearization}; the Hessian and divdiv complexes are due to \cite{hu2023distributional}.
}
 \label{fig:table}
\end{figure}
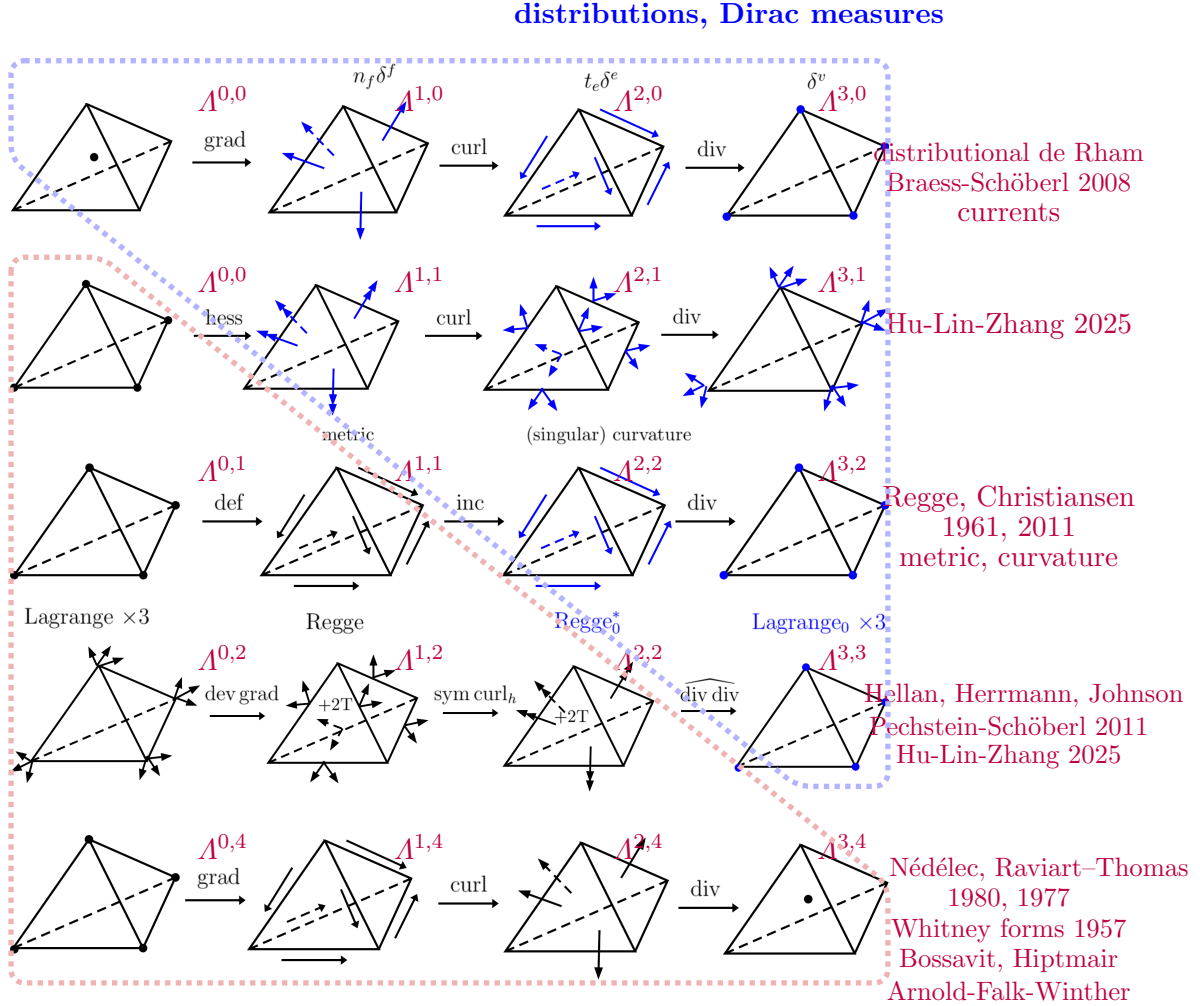

\begin{figure}
\begin{minipage}{0.45\textwidth}
 \begin{center}
	   \includegraphics[width=0.29\textwidth]{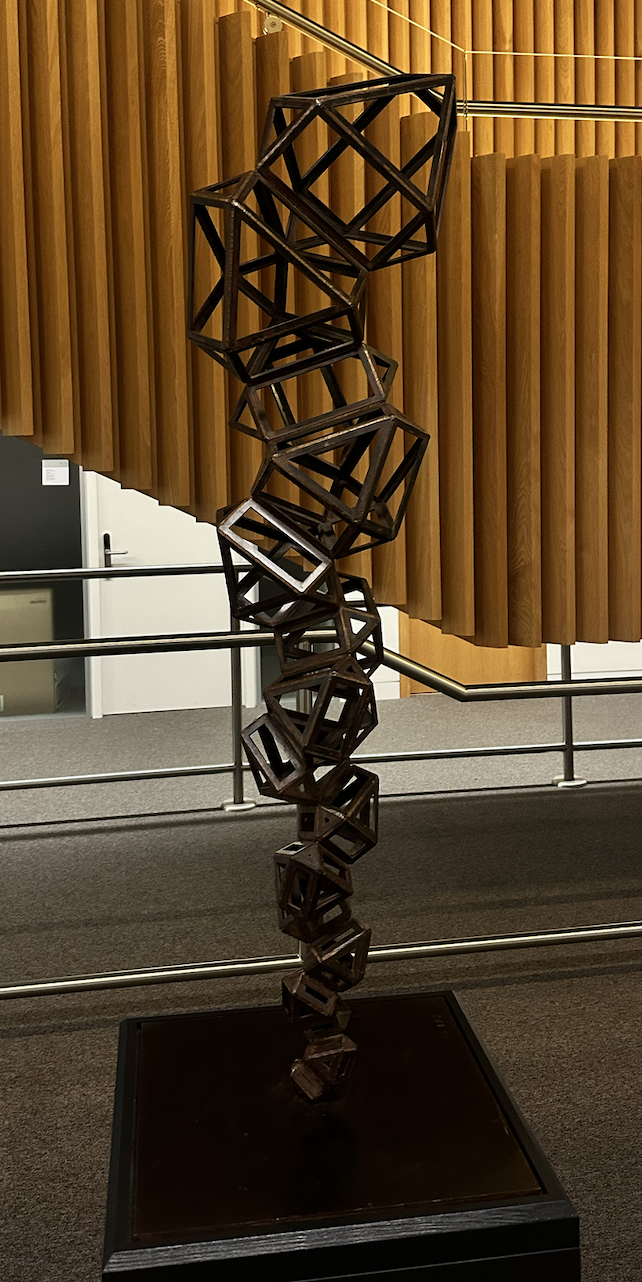}\quad
\end{center}
 \begin{center}
  {Discrete Differential Geometry (DDG): frame (edges lengths, angle etc.)\footnotetext{{Picture: Exhibition: Cascading Principles: Expansions within Geometry, Philosophy, and Interference (https://www.maths.ox.ac.uk/node/61184), Author:  Conrad Shawcross. }}}
  \end{center}
\end{minipage}
\begin{minipage}{0.45\textwidth}
 \begin{center}
	   \includegraphics[width=1.17\textwidth]{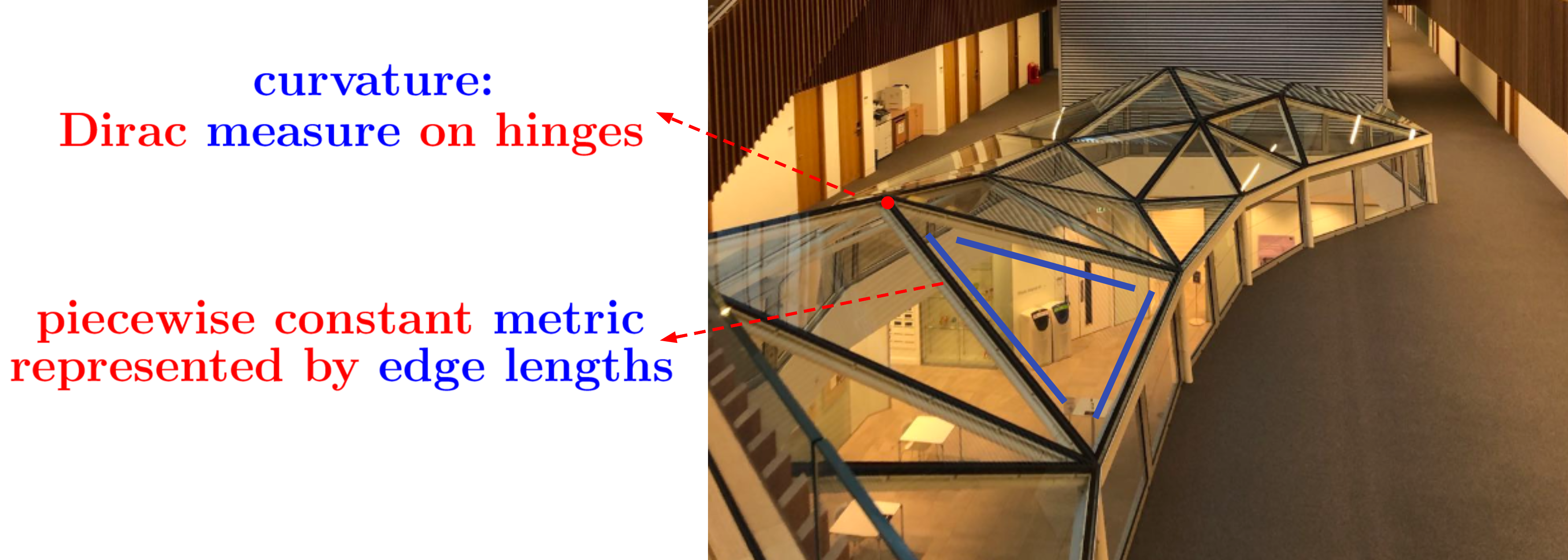}\quad
	     {Finite elements: piecewise functions (measure)}
\end{center}
\end{minipage}
\caption{Finite elements extend discrete definitions (such as using edge lengths to represent a metric) to piecewise defined functions or fields. This shift of point of view facilitates a rigorous mathematical and numerical analysis.}
\end{figure}

\bibliographystyle{siam}      
\bibliography{reference}{}   

\begin{thebibliography}{100}

\bibitem{alcubierre2008introduction}
{\sc M.~Alcubierre}, {\em {Introduction to 3+1 numerical relativity}},
  vol.~140, Oxford University Press, 2008.

\bibitem{anderson1997einstein}
{\sc A.~Anderson, Y.~Choquet-Bruhat, and J.~W. York~Jr}, {\em Einstein-bianchi
  hyperbolic system for general relativity}, arXiv preprint gr-qc/9710041,
  (1997).

\bibitem{anderson2017theory}
{\sc P.~M. Anderson, J.~P. Hirth, and J.~Lothe}, {\em Theory of dislocations},
  Cambridge University Press, 2017.

\bibitem{andrews2024high}
{\sc B.~D. Andrews and P.~E. Farrell}, {\em High-order conservative and
  accurately dissipative numerical integrators via auxiliary variables}, arXiv
  preprint arXiv:2407.11904,  (2024).

\bibitem{antelmi2023survey}
{\sc A.~Antelmi, G.~Cordasco, M.~Polato, V.~Scarano, C.~Spagnuolo, and
  D.~Yang}, {\em A survey on hypergraph representation learning}, ACM Computing
  Surveys, 56 (2023), pp.~1--38.

\bibitem{arnold2008finite}
{\sc D.~Arnold, G.~Awanou, and R.~Winther}, {\em {Finite elements for symmetric
  tensors in three dimensions}}, Mathematics of Computation, 77 (2008),
  pp.~1229--1251.

\bibitem{arnold2018finite}
{\sc D.~N. Arnold}, {\em Finite element exterior calculus}, SIAM, 2018.

\bibitem{arnold1989uniformly}
{\sc D.~N. Arnold and R.~S. Falk}, {\em {A uniformly accurate finite element
  method for the Reissner--Mindlin plate}}, SIAM Journal on Numerical Analysis,
  26 (1989), pp.~1276--1290.

\bibitem{Arnold2006a}
{\sc D.~N. Arnold, R.~S. Falk, and R.~Winther}, {\em {Differential complexes
  and stability of finite element methods II: The elasticity complex}},
  Compatible spatial discretizations,  (2006), pp.~47--67.

\bibitem{Arnold.D;Falk.R;Winther.R.2006a}
{\sc D.~N. Arnold, R.~S. Falk, and R.~Winther}, {\em {Finite element exterior
  calculus, homological techniques, and applications}}, Acta numerica, 15
  (2006), p.~1.

\bibitem{Arnold.D;Falk.R;Winther.R.2010a}
\leavevmode\vrule height 2pt depth -1.6pt width 23pt, {\em {Finite element
  exterior calculus: from Hodge theory to numerical stability}}, Bulletin of
  the American Mathematical Society, 47 (2010), pp.~281--354.

\bibitem{arnold2021complexes}
{\sc D.~N. Arnold and K.~Hu}, {\em Complexes from complexes}, Foundations of
  Computational Mathematics,  (2021), pp.~1--36.

\bibitem{arnold2014periodic}
{\sc D.~N. Arnold and A.~Logg}, {\em {Periodic table of the finite elements}},
  SIAM News, 47 (2014).

\bibitem{arnold2002range}
{\sc D.~N. Arnold, A.~L. Madureira, and S.~Zhang}, {\em {On the range of
  applicability of the Reissner-Mindlin and Kirchhoff-Love plate bending
  models}}, Journal of elasticity and the physical science of solids, 67
  (2002), pp.~171--185.

\bibitem{arnold1992quadratic}
{\sc D.~N. Arnold and J.~Qin}, {\em {Quadratic velocity/linear pressure Stokes
  elements}}, Advances in computer methods for partial differential equations,
  7 (1992), pp.~28--34.

\bibitem{arnold2002mixed}
{\sc D.~N. Arnold and R.~Winther}, {\em {Mixed finite elements for
  elasticity}}, Numerische Mathematik, 92 (2002), pp.~401--419.

\bibitem{arnold1992ordinary}
{\sc V.~I. Arnold}, {\em Ordinary differential equations}, Springer Science \&
  Business Media, 1992.

\bibitem{arnold1974asymptotic}
{\sc V.~I. Arnold and V.~I. Arnold}, {\em {The asymptotic Hopf invariant and
  its applications}}, Vladimir I. Arnold-Collected Works: Hydrodynamics,
  Bifurcation Theory, and Algebraic Geometry 1965-1972,  (1974), pp.~357--375.

\bibitem{arnold1999topological}
{\sc V.~I. Arnold and B.~A. Khesin}, {\em Topological {{Methods}} in
  {{Hydrodynamics}}}, vol.~125 of Applied {{Mathematical Sciences}}, Springer
  International Publishing, Cham, 2021.

\bibitem{austin2012s}
{\sc D.~Austin}, {\em {Who’s number 1? Hodge theory will tell us}}, AMS
  Feature Column, December,  (2012).

\bibitem{ayzenberg2025sheaf}
{\sc A.~Ayzenberg, T.~Gebhart, G.~Magai, and G.~Solomadin}, {\em Sheaf theory:
  from deep geometry to deep learning}, arXiv preprint arXiv:2502.15476,
  (2025).

\bibitem{basu2024harmonic}
{\sc S.~Basu and N.~Cox}, {\em Harmonic persistent homology}, SIAM Journal on
  Applied Algebra and Geometry, 8 (2024), pp.~189--224.

\bibitem{baumgarte1998numerical}
{\sc T.~W. Baumgarte and S.~L. Shapiro}, {\em {Numerical integration of
  Einstein’s field equations}}, Physical Review D, 59 (1998), p.~024007.

\bibitem{beig2020linearised}
{\sc R.~Beig and P.~T. Chrusciel}, {\em {On linearised vacuum constraint
  equations on Einstein manifolds}}, Classical and Quantum Gravity,  (2020).

\bibitem{bell2019continuous}
{\sc J.~L. Bell}, {\em The continuous, the discrete and the infinitesimal in
  philosophy and mathematics}, vol.~82, Springer, 2019.

\bibitem{berchenko2025finite}
{\sc Y.~Berchenko-Kogan and E.~S. Gawlik}, {\em Finite element spaces of double
  forms}, arXiv preprint arXiv:2505.17243,  (2025).

\bibitem{lazar-thesis}
{\sc L.~Bert\'ok}, {\em {Random Graph Topology and Phase Transition}}, Master's
  thesis, University of Edinburgh, 2024.

\bibitem{bick2023higher}
{\sc C.~Bick, E.~Gross, H.~A. Harrington, and M.~T. Schaub}, {\em What are
  higher-order networks?}, SIAM review, 65 (2023), pp.~686--731.

\bibitem{billera1991dimenstion}
{\sc L.~Billera and R.~Haas}, {\em The dimenstion and bases of divergence-free
  splines; a homological approach}, Approximation Theory and its Applications,
  7 (1991), pp.~91--99.

\bibitem{billera1988homology}
{\sc L.~J. Billera}, {\em {Homology of smooth splines: generic triangulations
  and a conjecture of Strang}}, Transactions of the american Mathematical
  Society, 310 (1988), pp.~325--340.

\bibitem{bobenko2008discrete}
{\sc A.~I. Bobenko and Y.~B. Suris}, {\em Discrete differential geometry:
  Integrable structure}, vol.~98, American Mathematical Soc., 2008.

\bibitem{bonizzoni2025discrete}
{\sc F.~Bonizzoni, K.~Hu, G.~Kanschat, and D.~Sap}, {\em {Discrete tensor
  product BGG sequences: splines and finite elements}}, Mathematics of
  Computation, 94 (2025), pp.~517--549.

\bibitem{boon2022hodge}
{\sc W.~M. Boon, D.~F. Holmen, J.~M. Nordbotten, and J.~E. Vatne}, {\em {The
  Hodge-Laplacian on the C\v{e}ch-de Rham complex governs coupled problems}},
  arXiv preprint arXiv:2211.04556,  (2022).

\bibitem{boon2023mixed}
{\sc W.~M. Boon and J.~M. Nordbotten}, {\em Mixed-dimensional poromechanical
  models of fractured porous media}, Acta Mechanica, 234 (2023),
  pp.~1121--1168.

\bibitem{boon2021functional}
{\sc W.~M. Boon, J.~M. Nordbotten, and J.~E. Vatne}, {\em Functional analysis
  and exterior calculus on mixed-dimensional geometries}, Annali di Matematica
  Pura ed Applicata (1923-), 200 (2021), pp.~757--789.

\bibitem{bossavit1988whitney}
{\sc A.~Bossavit}, {\em {Whitney forms: A class of finite elements for
  three-dimensional computations in electromagnetism}}, IEE Proceedings A
  (Physical Science, Measurement and Instrumentation, Management and Education,
  Reviews), 135 (1988), pp.~493--500.

\bibitem{Bossavit.A.1998a}
{\sc A.~Bossavit}, {\em Computational Electromagnetism}, Academic Press
  (Boston), 1998.

\bibitem{bott2013differential}
{\sc R.~Bott and L.~W. Tu}, {\em {Differential forms in algebraic topology}},
  vol.~82, Springer Science \& Business Media, 2013.

\bibitem{Brackbill.J;Barnes.D.1980a}
{\sc J.~U. Brackbill and D.~C. Barnes}, {\em {The effect of nonzero
  $\nabla\cdot B$ on the numerical solution of the magnetohydrodynamic
  equations}}, Journal of Computational Physics, 430 (1980), pp.~426--430.

\bibitem{braess2008equilibrated}
{\sc D.~Braess and J.~Sch{\"o}berl}, {\em Equilibrated residual error estimator
  for edge elements}, Mathematics of Computation, 77 (2008), pp.~651--672.

\bibitem{bredon1997sheaf}
{\sc G.~E. Bredon}, {\em Sheaf theory}, vol.~170, Springer Science \& Business
  Media, 1997.

\bibitem{brenner2008mathematical}
{\sc S.~C. Brenner and R.~Scott}, {\em The mathematical theory of finite
  element methods}, vol.~15, Springer Science \& Business Media, 2008.

\bibitem{brezzi2008mixed}
{\sc F.~Brezzi, D.~Boffi, L.~Demkowicz, R.~Dur{\'a}n, R.~Falk, and M.~Fortin},
  {\em {Mixed finite elements, compatibility conditions, and applications}},
  Springer, 2008.

\bibitem{brezzi1985two}
{\sc F.~Brezzi, J.~Douglas~Jr, and L.~D. Marini}, {\em {Two families of mixed
  finite elements for second order elliptic problems}}, Numerische Mathematik,
  47 (1985), pp.~217--235.

\bibitem{bru1992hilbert}
{\sc J.~Bru, M.~Lesch, et~al.}, {\em Hilbert complexes}, Journal of Functional
  Analysis, 108 (1992), pp.~88--132.

\bibitem{calabi1961compact}
{\sc E.~Calabi}, {\em {On compact Riemannian manifolds with constant
  curvature}}, in I, Proc. Sympos. Pure Math, vol.~3, 1961, pp.~155--180.

\bibitem{cang2020persistent}
{\sc Z.~Cang and G.-W. Wei}, {\em Persistent cohomology for data with
  multicomponent heterogeneous information}, SIAM journal on mathematics of
  data science, 2 (2020), pp.~396--418.

\bibitem{vcap2023bounded}
{\sc A.~{\v{C}}ap and K.~Hu}, {\em {Bounded Poincar{\'e} operators for twisted
  and BGG complexes}}, Journal de Math{\'e}matiques Pures et Appliqu{\'e}es,
  179 (2023), pp.~253--276.

\bibitem{vcap2023bgg}
\leavevmode\vrule height 2pt depth -1.6pt width 23pt, {\em {BGG sequences with
  weak regularity and applications}}, Foundations of Computational Mathematics,
  24 (2024), pp.~1145--1184.

\bibitem{cartan1899certaines}
{\sc {\'E}.~Cartan}, {\em {Sur certaines expressions diff{\'e}rentielles et le
  probl{\`e}me de Pfaff}}, in Annales scientifiques de l'{\'E}cole normale
  sup{\'e}rieure, vol.~16, 1899, pp.~239--332.

\bibitem{cesaro1906sulle}
{\sc E.~Ces{\`a}ro}, {\em Sulle formole del volterra fondamentali nella teoria
  delle distorsioni elastiche}, Il Nuovo Cimento (1901-1910), 12 (1906),
  pp.~143--154.

\bibitem{chen2020discrete}
{\sc L.~Chen and X.~Huang}, {\em {Discrete Hessian complexes in three
  dimensions}}, arXiv preprint arXiv:2012.10914,  (2020).

\bibitem{chen2020finite}
\leavevmode\vrule height 2pt depth -1.6pt width 23pt, {\em Finite elements for
  divdiv-conforming symmetric tensors}, arXiv preprint arXiv:2005.01271,
  (2020).

\bibitem{chen2021finite}
\leavevmode\vrule height 2pt depth -1.6pt width 23pt, {\em A finite element
  elasticity complex in three dimensions}, arXiv preprint arXiv:2106.12786,
  (2021).

\bibitem{chen2021finite2}
\leavevmode\vrule height 2pt depth -1.6pt width 23pt, {\em Finite elements for
  divdiv-conforming symmetric tensors in arbitrary dimension}, arXiv preprint
  arXiv:2106.13384,  (2021).

\bibitem{chen2021geometric}
\leavevmode\vrule height 2pt depth -1.6pt width 23pt, {\em Geometric
  decompositions of div-conforming finite element tensors}, arXiv preprint
  arXiv:2112.14351,  (2021).

\bibitem{chen2024h}
\leavevmode\vrule height 2pt depth -1.6pt width 23pt, {\em
  {$H(\mathrm{div})$-conforming finite element tensors with constraints}},
  Results in Applied Mathematics, 23 (2024), p.~100494.

\bibitem{chen2025complexes}
\leavevmode\vrule height 2pt depth -1.6pt width 23pt, {\em {Complexes from
  complexes: Finite element complexes in three dimensions}}, Mathematics of
  Computation,  (2025).

\bibitem{chipot2021inequalities}
{\sc M.~Chipot}, {\em {On inequalities of Korn's type}}, Journal de
  Math{\'e}matiques Pures et Appliqu{\'e}es, 148 (2021), pp.~199--220.

\bibitem{chodura19813d}
{\sc R.~Chodura and A.~Schl{\"u}ter}, {\em {A 3D code for MHD equilibrium and
  stability}}, Journal of Computational Physics, 41 (1981), pp.~68--88.

\bibitem{choquet2008general}
{\sc Y.~Choquet-Bruhat}, {\em {General relativity and the Einstein equations}},
  OUP Oxford, 2008.

\bibitem{christiansen2010finite}
{\sc S.~H. Christiansen}, {\em Finite element systems of differential forms},
  arXiv preprint arXiv:1006.4779,  (2010).

\bibitem{christiansen2011linearization}
{\sc S.~H. Christiansen}, {\em {On the linearization of Regge calculus}},
  Numerische Mathematik, 119 (2011), pp.~613--640.

\bibitem{christiansen2020discrete}
{\sc S.~H. Christiansen, J.~Gopalakrishnan, J.~Guzm{\'a}n, and K.~Hu}, {\em {A
  discrete elasticity complex on three-dimensional Alfeld splits}}, arXiv
  preprint arXiv:2009.07744,  (2020).

\bibitem{christiansen2018generalized}
{\sc S.~H. Christiansen and K.~Hu}, {\em {Generalized finite element systems
  for smooth differential forms and Stokes' problem}}, Numerische Mathematik,
  140 (2018), pp.~327--371.

\bibitem{christiansen2019finite}
\leavevmode\vrule height 2pt depth -1.6pt width 23pt, {\em {Finite Element
  Systems for vector bundles: elasticity and curvature}}, accepted, Foundations
  of Computational Mathematics,  (2021).

\bibitem{christiansen2023extended}
{\sc S.~H. Christiansen, K.~Hu, and T.~Lin}, {\em {Extended Regge complex for
  linearized Riemann-Cartan geometry and cohomology}}, arXiv preprint
  arXiv:2312.11709,  (2023).

\bibitem{christiansen2020poincare}
{\sc S.~H. Christiansen, K.~Hu, and E.~Sande}, {\em Poincar{\'e} path integrals
  for elasticity}, Journal de Math{\'e}matiques Pures et Appliqu{\'e}es, 135
  (2020), pp.~83--102.

\bibitem{ciarlet1997mathematical}
{\sc P.~G. Ciarlet}, {\em {Mathematical Elasticity: Volume II: Theory of
  Plates}}, Elsevier, 1997.

\bibitem{ciarlet2013linear}
\leavevmode\vrule height 2pt depth -1.6pt width 23pt, {\em Linear and nonlinear
  functional analysis with applications}, vol.~130, SIAM, 2013.

\bibitem{ciarlet2006nonlinear}
{\sc P.~G. Ciarlet, L.~Gratie, and C.~Mardare}, {\em {A nonlinear Korn
  inequality on a surface}}, Journal de math{\'e}matiques pures et
  appliqu{\'e}es, 85 (2006), pp.~2--16.

\bibitem{ciarlet2009intrinsic}
\leavevmode\vrule height 2pt depth -1.6pt width 23pt, {\em Intrinsic methods in
  elasticity: a mathematical survey}, Discrete \& Continuous Dynamical
  Systems-A, 23 (2009), p.~133.

\bibitem{ciarlet2009cesaro}
{\sc P.~G. Ciarlet, L.~Gratie, and M.~Serpilli}, {\em {Ces{\`a}ro--Volterra
  path integral formula on a surface}}, Mathematical Models and Methods in
  Applied Sciences, 19 (2009), pp.~419--441.

\bibitem{cosserat1909theorie}
{\sc E.~M.~P. Cosserat and F.~Cosserat}, {\em Th{\'e}orie des corps
  d{\'e}formables}, A. Hermann et fils, 1909.

\bibitem{costabel2010bogovskiui}
{\sc M.~Costabel and A.~McIntosh}, {\em {On Bogovski{\u\i} and regularized
  Poincar{\'e} integral operators for de Rham complexes on Lipschitz domains}},
  Mathematische Zeitschrift, 265 (2010), pp.~297--320.

\bibitem{cotter2023compatible}
{\sc C.~J. Cotter}, {\em Compatible finite element methods for geophysical
  fluid dynamics}, Acta Numerica, 32 (2023), pp.~291--393.

\bibitem{dain2006generalized}
{\sc S.~Dain}, {\em {Generalized Korn's inequality and conformal Killing
  vectors}}, Calculus of variations and partial differential equations, 25
  (2006), pp.~535--540.

\bibitem{dalmonte2016lattice}
{\sc M.~Dalmonte and S.~Montangero}, {\em Lattice gauge theory simulations in
  the quantum information era}, Contemporary Physics, 57 (2016), pp.~388--412.

\bibitem{dupuis2012discrete}
{\sc M.~Dupuis, J.~P. Ryan, S.~Speziale, et~al.}, {\em Discrete gravity models
  and loop quantum gravity: a short review}, SIGMA. Symmetry, Integrability and
  Geometry: Methods and Applications, 8 (2012), p.~052.

\bibitem{erdds1959random}
{\sc P.~Erdős and A.~Rényi}, {\em {On random graphs I}}, Publicationes
  Mathematicae Debrecen, 6 (1959), p.~18.

\bibitem{fengkang}
{\sc K.~Feng}, {\em Elliptic equations and combinatorial elastic structures on
  combinatorial manifolds (in chinese)}, Mathematica Numerica Sinica, 1 (1979),
  pp.~199--208.

\bibitem{feng1987symplectic}
{\sc K.~Feng and M.-z. Qin}, {\em {The symplectic methods for the computation
  of Hamiltonian equations}}, in Numerical Methods for Partial Differential
  Equations: Proceedings of a Conference held in Shanghai, PR China, March
  25--29, 1987, Springer, 1987, pp.~1--37.

\bibitem{friedrich1985hyperbolicity}
{\sc H.~Friedrich}, {\em {On the hyperbolicity of Einstein's and other gauge
  field equations}}, Communications in Mathematical Physics, 100 (1985),
  pp.~525--543.

\bibitem{gawlik2024finite}
{\sc E.~Gawlik and M.~Neunteufel}, {\em Finite element approximation of scalar
  curvature in arbitrary dimension}, Mathematics of Computation,  (2024).

\bibitem{gawlik2020high}
{\sc E.~S. Gawlik}, {\em {High-order approximation of Gaussian curvature with
  Regge finite elements}}, SIAM Journal on Numerical Analysis, 58 (2020),
  pp.~1801--1821.

\bibitem{gawlik2022finite}
{\sc E.~S. Gawlik and F.~Gay-Balmaz}, {\em {A finite element method for MHD
  that preserves energy, cross-helicity, magnetic helicity, incompressibility,
  and $\operatorname{div} B= 0$}}, Journal of Computational Physics, 450
  (2022), p.~110847.

\bibitem{gawlik2023finite}
{\sc E.~S. Gawlik and M.~Neunteufel}, {\em {Finite element approximation of the
  Einstein tensor}}, arXiv preprint arXiv:2310.18802,  (2023).

\bibitem{gopalakrishnan2020mass}
{\sc J.~Gopalakrishnan, P.~L. Lederer, and J.~Sch{\"o}berl}, {\em {A mass
  conserving mixed stress formulation for the Stokes equations}}, IMA Journal
  of Numerical Analysis, 40 (2020), pp.~1838--1874.

\bibitem{gopalakrishnan2023analysis}
{\sc J.~Gopalakrishnan, M.~Neunteufel, J.~Sch{\"o}berl, and M.~Wardetzky}, {\em
  {Analysis of distributional Riemann curvature tensor in any dimension}},
  arXiv preprint arXiv:2311.01603,  (2023).

\bibitem{gopalakrishnan2024improved}
\leavevmode\vrule height 2pt depth -1.6pt width 23pt, {\em {On the improved
  convergence of lifted distributional Gauss curvature from Regge elements}},
  arXiv preprint arXiv:2401.12734,  (2024).

\bibitem{gourgoulhon20123}
{\sc E.~Gourgoulhon}, {\em 3+ 1 formalism in general relativity: bases of
  numerical relativity}, vol.~846, Springer Science \& Business Media, 2012.

\bibitem{guzman2021estimation}
{\sc J.~Guzman and A.~J. Salgado}, {\em {Estimation of the continuity constants
  for Bogovski{\u\i} and regularized Poincar{\'e} integral operators}}, Journal
  of Mathematical Analysis and Applications, 502 (2021), p.~125246.

\bibitem{guzman2019scott}
{\sc J.~Guzm{\'a}n and L.~R. Scott}, {\em {The Scott-Vogelius finite elements
  revisited}}, Mathematics of Computation, 88 (2019), pp.~515--529.

\bibitem{hairer2006geometric}
{\sc E.~Hairer, C.~Lubich, and G.~Wanner}, {\em {Geometric numerical
  integration: structure-preserving algorithms for ordinary differential
  equations}}, vol.~31, Springer Science \& Business Media, 2006.

\bibitem{he2025topology}
{\sc M.~He, P.~E. Farrell, K.~Hu, and B.~D. Andrews}, {\em {Topology-preserving
  discretization for the magneto-frictional equations arising in the Parker
  conjecture}}, arXiv preprint arXiv:2501.11654,  (2025).

\bibitem{heumann2012fully}
{\sc H.~Heumann, R.~Hiptmair, K.~Li, and J.~Xu}, {\em {Fully discrete
  semi-Lagrangian methods for advection of differential forms}}, BIT Numerical
  Mathematics, 52 (2012), pp.~981--1007.

\bibitem{hiptmair1999canonical}
{\sc R.~Hiptmair}, {\em Canonical construction of finite elements}, Mathematics
  of Computation of the American Mathematical Society, 68 (1999),
  pp.~1325--1346.

\bibitem{hiptmair2001higher}
\leavevmode\vrule height 2pt depth -1.6pt width 23pt, {\em Higher order whitney
  forms}, Progress in Electromagnetics Research, 32 (2001), pp.~271--299.

\bibitem{hirani2003discrete}
{\sc A.~N. Hirani}, {\em Discrete exterior calculus}, PhD thesis, Caltech,
  2003.

\bibitem{hormander1994analysis}
{\sc L.~H{\"o}rmander}, {\em {The analysis of linear partial differential
  operators III: Pseudo-differential operators}}, vol.~274, Springer Science \&
  Business Media, 1994.

\bibitem{hu2021conforming}
{\sc J.~Hu and Y.~Liang}, {\em {Conforming discrete Gradgrad-complexes in three
  dimensions}}, Mathematics of Computation, 90 (2021), pp.~1637--1662.

\bibitem{hu2021conforming2}
{\sc J.~Hu, Y.~Liang, and R.~Ma}, {\em {Conforming finite element DIVDIV
  complexes and the application for the linearized Einstein-Bianchi system}},
  arXiv preprint arXiv:2103.00088,  (2021).

\bibitem{hu2025sharpness}
{\sc J.~Hu, T.~Lin, Q.~Wu, and B.~Yuan}, {\em The sharpness condition for
  constructing a finite element from a superspline}, Mathematics of
  Computation,  (2025).

\bibitem{hu2022nonlinear}
{\sc K.~Hu}, {\em Nonlinear elasticity complex and a finite element diagram
  chase}, in INdAM Meeting: Approximation Theory and Numerical Analysis meet
  Algebra, Geometry, Topology, Springer, 2022, pp.~231--252.

\bibitem{khu-review}
\leavevmode\vrule height 2pt depth -1.6pt width 23pt, {\em {Intrinsic finite
  elements: currents, discrete geometry and periodic table}}, Mathematica
  Numerica Sinica, 47 (2025), pp.~385--417.

\bibitem{hu2025finite}
{\sc K.~Hu and T.~Lin}, {\em {Finite element form-valued forms (I):
  Construction}}, arXiv preprint arXiv:2503.03243,  (2025).

\bibitem{hu2023distributional}
{\sc K.~Hu, T.~Lin, and Q.~Zhang}, {\em Distributional hessian and divdiv
  complexes on triangulation and cohomology}, SIAM Journal on Applied Algebra
  and Geometry, 9 (2025), pp.~108--153.

\bibitem{hu2013positivity}
{\sc X.~Y. Hu, N.~A. Adams, and C.-W. Shu}, {\em {Positivity-preserving method
  for high-order conservative schemes solving compressible Euler equations}},
  Journal of Computational Physics, 242 (2013), pp.~169--180.

\bibitem{immirzi1997quantum}
{\sc G.~Immirzi}, {\em {Quantum gravity and Regge calculus}}, Nuclear Physics
  B-Proceedings Supplements, 57 (1997), pp.~65--72.

\bibitem{jariwala2017mixed}
{\sc D.~Jariwala, T.~J. Marks, and M.~C. Hersam}, {\em {Mixed-dimensional van
  der Waals heterostructures}}, Nature materials, 16 (2017), pp.~170--181.

\bibitem{jensen2020force}
{\sc O.~E. Jensen, E.~Johns, and S.~Woolner}, {\em {Force networks, torque
  balance and Airy stress in the planar vertex model of a confluent
  epithelium}}, Proceedings of the Royal Society A, 476 (2020), p.~20190716.

\bibitem{jiang2011statistical}
{\sc X.~Jiang, L.-H. Lim, Y.~Yao, and Y.~Ye}, {\em {Statistical ranking and
  combinatorial Hodge theory}}, Mathematical Programming, 127 (2011),
  pp.~203--244.

\bibitem{jin2022asymptotic}
{\sc S.~Jin}, {\em Asymptotic-preserving schemes for multiscale physical
  problems}, Acta Numerica, 31 (2022), pp.~415--489.

\bibitem{kahle2009topology}
{\sc M.~Kahle}, {\em Topology of random clique complexes}, Discrete
  mathematics, 309 (2009), pp.~1658--1671.

\bibitem{kahle2014topology}
{\sc M.~Kahle et~al.}, {\em Topology of random simplicial complexes: a survey},
  AMS Contemp. Math, 620 (2014), pp.~201--222.

\bibitem{fengkang-fem}
{\sc F.~Kang}, {\em Finite difference schemes based on variational principles
  (in chinese)}, Applied Mathematics and Computational Mathematics,  (1965).

\bibitem{kelvin1867vortex}
{\sc L.~Kelvin}, {\em On vortex atoms}, in Proc. R. Soc. Edin, vol.~6, 1867,
  pp.~94--105.

\bibitem{kroner1960general}
{\sc E.~Kr{\"o}ner}, {\em General continuum theory of dislocations and proper
  stresses}, Arch. Rat. Mech. Anal, 4 (1960), pp.~273--334.

\bibitem{kroner1963dislocation}
\leavevmode\vrule height 2pt depth -1.6pt width 23pt, {\em {Dislocation: a new
  concept in the continuum theory of plasticity}}, Studies in Applied
  Mathematics, 42 (1963), pp.~27--37.

\bibitem{kroner1985incompatibility}
{\sc E.~Kr{\"o}ner}, {\em Incompatibility, defects, and stress functions in the
  mechanics of generalized continua}, International journal of solids and
  structures, 21 (1985), pp.~747--756.

\bibitem{kroner1995dislocations}
\leavevmode\vrule height 2pt depth -1.6pt width 23pt, {\em {Dislocations in
  crystals and in continua: a confrontation}}, International journal of
  engineering science, 33 (1995), pp.~2127--2135.

\bibitem{kroner1981continuum}
{\sc E.~Kr{\"o}ner et~al.}, {\em Continuum theory of defects}, Physics of
  defects, 35 (1981), pp.~217--315.

\bibitem{lai2007spline}
{\sc M.-J. Lai and L.~L. Schumaker}, {\em {Spline functions on
  triangulations}}, no.~110, Cambridge University Press, 2007.

\bibitem{lanini2025approximation}
{\sc M.~Lanini, C.~Manni, and H.~Schenck}, {\em Approximation Theory and
  Numerical Analysis Meet Algebra, Geometry, Topology}, Springer, 2025.

\bibitem{le2017mixed}
{\sc B.~L{\'e}, G.~Legrain, and N.~Mo{\"e}s}, {\em Mixed dimensional modeling
  of reinforced structures}, Finite elements in Analysis and Design, 128
  (2017), pp.~1--18.

\bibitem{lee1985discrete}
{\sc T.~Lee}, {\em Discrete mechanics}, in How far are we from the gauge
  forces, Springer, 1985, pp.~15--114.

\bibitem{lee1987difference}
\leavevmode\vrule height 2pt depth -1.6pt width 23pt, {\em Difference equations
  and conservation laws}, Journal of Statistical Physics, 46 (1987),
  pp.~843--860.

\bibitem{leonard1993positivity}
{\sc B.~Leonard, M.~MacVean, and A.~Lock}, {\em Positivity-preserving numerical
  schemes for multidimensional advection}, tech. rep., 1993.

\bibitem{licht2017complexes}
{\sc M.~W. Licht}, {\em Complexes of discrete distributional differential forms
  and their homology theory}, Foundations of Computational Mathematics, 17
  (2017), pp.~1085--1122.

\bibitem{lim2020hodge}
{\sc L.-H. Lim}, {\em {Hodge Laplacians on graphs}}, SIAM Review, 62 (2020),
  pp.~685--715.

\bibitem{liu2021hypergraph}
{\sc X.~Liu, X.~Wang, J.~Wu, and K.~Xia}, {\em {Hypergraph-based persistent
  cohomology (HPC) for molecular representations in drug design}}, Briefings in
  Bioinformatics, 22 (2021), p.~bbaa411.

\bibitem{moffatt1969degree}
{\sc H.~K. Moffatt}, {\em The degree of knottedness of tangled vortex lines},
  Journal of Fluid Mechanics, 35 (1969), pp.~117--129.

\bibitem{nakamura1987general}
{\sc T.~Nakamura, K.~Oohara, and Y.~Kojima}, {\em General relativistic collapse
  to black holes and gravitational waves from black holes}, Progress of
  Theoretical Physics Supplement, 90 (1987), pp.~1--218.

\bibitem{Nedelec.J.1980a}
{\sc J.~N\'{e}d\'{e}lec}, {\em {Mixed finite elements in $\mathbb{R}^{3}$}},
  Numerische Mathematik, 35 (1980), pp.~315--341.

\bibitem{Nedelec.J.1986a}
\leavevmode\vrule height 2pt depth -1.6pt width 23pt, {\em {A new family of
  mixed finite elements in $\mathbb{R}^{3}$}}, Numerische Mathematik, 50
  (1986), pp.~57--81.

\bibitem{neff2010reissner}
{\sc P.~Neff, K.-I. Hong, and J.~Jeong}, {\em {The Reissner--Mindlin plate is
  the $\Gamma$-limit of Cosserat elasticity}}, Mathematical Models and Methods
  in Applied Sciences, 20 (2010), pp.~1553--1590.

\bibitem{neff2008linear}
{\sc P.~Neff and J.~Jeong}, {\em The linear isotropic cosserat (micropolar)
  model},
  https://www.uni-due.de/$\sim$hm0014/cosserat\_files/web\_lin\_coss.pdf,
  (2008).

\bibitem{neunteufel2021avoiding}
{\sc M.~Neunteufel and J.~Sch{\"o}berl}, {\em {Avoiding membrane locking with
  Regge interpolation}}, Computer Methods in Applied Mechanics and Engineering,
  373 (2021), p.~113524.

\bibitem{nye1953some}
{\sc J.~F. Nye}, {\em Some geometrical relations in dislocated crystals}, Acta
  metallurgica, 1 (1953), pp.~153--162.

\bibitem{pagliantini2016computational}
{\sc C.~Pagliantini}, {\em {Computational Magnetohydrodynamics with Discrete
  Differential Forms}}, PhD thesis, 2016.

\bibitem{park2008variational}
{\sc S.~Park and X.-L. Gao}, {\em Variational formulation of a modified couple
  stress theory and its application to a simple shear problem}, Zeitschrift
  f{\"u}r angewandte Mathematik und Physik, 59 (2008), pp.~904--917.

\bibitem{parkerTopologicalDissipationSmallScale1972}
{\sc E.~N. Parker}, {\em Topological {{dissipation}} and the {{small-scale
  fields}} in {{turbulent gases}}}, The Astrophysical Journal, 174 (1972),
  p.~499.

\bibitem{pauly2019global}
{\sc D.~Pauly}, {\em {A global div-curl-lemma for mixed boundary conditions in
  weak Lipschitz domains and a corresponding generalized
  $A_{0}^{\ast}$-$A_{1}$-lemma in Hilbert spaces}}, Analysis, 39 (2019),
  pp.~33--58.

\bibitem{pechstein2011tangential}
{\sc A.~Pechstein and J.~Sch{\"o}berl}, {\em Tangential-displacement and
  normal--normal-stress continuous mixed finite elements for elasticity},
  Mathematical Models and Methods in Applied Sciences, 21 (2011),
  pp.~1761--1782.

\bibitem{philippe1978finite}
{\sc G.~C. Philippe}, {\em {The finite element method for elliptic problems}},
  1978.

\bibitem{pietraszkiewicz2009natural}
{\sc W.~Pietraszkiewicz and V.~Eremeyev}, {\em On natural strain measures of
  the non-linear micropolar continuum}, International Journal of Solids and
  Structures, 46 (2009), pp.~774--787.

\bibitem{pontinParkerProblemExistence2020}
{\sc D.~I. Pontin and G.~Hornig}, {\em The {{Parker}} problem: existence of
  smooth force-free fields and coronal heating}, Living Reviews in Solar
  Physics, 17 (2020), p.~5.

\bibitem{pretorius2005evolution}
{\sc F.~Pretorius}, {\em Evolution of binary black-hole spacetimes}, Physical
  review letters, 95 (2005), p.~121101.

\bibitem{quenneville2015new}
{\sc V.~Quenneville-B{\'e}lair}, {\em {A New Approach to Finite Element
  Simulations of General Relativity}}, PhD thesis, University of Minnesota,
  2015.

\bibitem{raussen2005interview}
{\sc M.~Raussen and C.~Skau}, {\em {Interview with Michael Atiyah and Isadore
  Singer}}, Notices of the AMS, 52 (2005).

\bibitem{Raviart.P;Thomas.J.1977a}
{\sc P.~Raviart and J.~Thomas}, {\em A mixed finite element method for second
  order elliptic problems}, Lecture Notes in Mathematics, 606 (1977),
  pp.~292--315.

\bibitem{regge1961general}
{\sc T.~Regge}, {\em {General relativity without coordinates}}, Il Nuovo
  Cimento (1955-1965), 19 (1961), pp.~558--571.

\bibitem{rendall1989insufficiency}
{\sc A.~D. Rendall}, {\em {Insufficiency of the Ricci and Bianchi identities
  for characterising curvature}}, Journal of Geometry and Physics, 6 (1989),
  pp.~159--163.

\bibitem{reula1998hyperbolic}
{\sc O.~A. Reula}, {\em {Hyperbolic methods for Einstein’s equations}},
  Living Reviews in Relativity, 1 (1998), p.~3.

\bibitem{rognes2022waterscales}
{\sc M.~E. Rognes}, {\em {Waterscales: Mathematical and computational
  foundations for modelling cerebral fluid flow}}, European Mathematical
  Society Magazine,  (2022), pp.~13--26.

\bibitem{rothe2012lattice}
{\sc H.~J. Rothe}, {\em Lattice gauge theories: an introduction}, World
  Scientific Publishing Company, 2012.

\bibitem{ruggiero2003einstein}
{\sc M.~L. Ruggiero and A.~Tartaglia}, {\em {Einstein--Cartan theory as a
  theory of defects in space--time}}, American Journal of Physics, 71 (2003),
  pp.~1303--1313.

\bibitem{schoberl2014c++}
{\sc J.~Sch{\"o}berl}, {\em {C++ 11 implementation of finite elements in
  NGSolve}}, Institute for analysis and scientific computing, Vienna University
  of Technology, 30 (2014).

\bibitem{scholz2019cartan}
{\sc E.~Scholz}, {\em {E. Cartan's attempt at bridge-building between Einstein
  and the Cosserats--or how translational curvature became to be known as
  torsion}}, The European Physical Journal H, 44 (2019), pp.~47--75.

\bibitem{shibata1995evolution}
{\sc M.~Shibata and T.~Nakamura}, {\em {Evolution of three-dimensional
  gravitational waves: Harmonic slicing case}}, Physical Review D, 52 (1995),
  p.~5428.

\bibitem{smarr1977space}
{\sc L.~Smarr}, {\em Space-times generated by computers: Black holes with
  gravitational radiation}, Annals of the New York Academy of Sciences, 302
  (1977), pp.~569--604.

\bibitem{smiet2017ideal}
{\sc C.~B. Smiet, S.~Candelaresi, and D.~Bouwmeester}, {\em {Ideal relaxation
  of the Hopf fibration}}, Physics of Plasmas, 24 (2017).

\bibitem{steinmann2015geometrical}
{\sc P.~Steinmann}, {\em Geometrical foundations of continuum mechanics},
  Lecture Notes in Applied Mathematics and Mechanics, 2 (2015).

\bibitem{su2024persistent}
{\sc Z.~Su, Y.~Tong, and G.-W. Wei}, {\em {Persistent de Rham-Hodge Laplacians
  in Eulerian representation for manifold topological learning}}, arXiv
  preprint arXiv:2408.00220,  (2024).

\bibitem{Toth.G.2000a}
{\sc G.~T\'{o}th}, {\em {The $\nabla\cdot B= 0$ constraint in shock-capturing
  magnetohydrodynamics codes}}, Journal of Computational Physics,  (2000),
  pp.~1--24.

\bibitem{van2016frank}
{\sc N.~Van~Goethem}, {\em {The Frank tensor as a boundary condition in
  intrinsic linearized elasticity}}, Journal of Geometric Mechanics, 8 (2016),
  pp.~391--411.

\bibitem{vardoulakis2019cosserat}
{\sc I.~Vardoulakis}, {\em Cosserat continuum mechanics}, Lecture Notes in
  Applied and Computational Mechanics, 87 (2019).

\bibitem{volterra1907equilibre}
{\sc V.~Volterra}, {\em Sur l'{\'e}quilibre des corps {\'e}lastiques
  multiplement connexes}, in Annales scientifiques de l'{\'E}cole normale
  sup{\'e}rieure, vol.~24, Soci{\'e}t{\'e} math{\'e}matique de France, 1907,
  pp.~401--517.

\bibitem{vonneumann1950method}
{\sc J.~VonNeumann and R.~D. Richtmyer}, {\em A method for the numerical
  calculation of hydrodynamic shocks}, Journal of applied physics, 21 (1950),
  pp.~232--237.

\bibitem{whitney2012geometric}
{\sc H.~Whitney}, {\em {Geometric integration theory}}, Courier Corporation,
  2012.

\bibitem{wilson1979numerical}
{\sc J.~R. Wilson}, {\em A numerical method for relativistic hydrodynamics},
  Sources of gravitational radiation,  (1979), pp.~423--445.

\bibitem{woltjer1958theorem}
{\sc L.~Woltjer}, {\em A theorem on force-free magnetic fields}, Proceedings of
  the National Academy of Sciences, 44 (1958), pp.~489--491.

\bibitem{yavari2020applications}
{\sc A.~Yavari}, {\em Applications of algebraic topology in elasticity}, in
  Geometric Continuum Mechanics, Springer, 2020, pp.~143--183.

\bibitem{yavari2012riemann}
{\sc A.~Yavari and A.~Goriely}, {\em {Riemann--Cartan geometry of nonlinear
  dislocation mechanics}}, Archive for Rational Mechanics and Analysis, 205
  (2012), pp.~59--118.

\bibitem{yavari2013riemann}
\leavevmode\vrule height 2pt depth -1.6pt width 23pt, {\em {Riemann--Cartan
  geometry of nonlinear disclination mechanics}}, Mathematics and Mechanics of
  Solids, 18 (2013), pp.~91--102.

\end{thebibliography}

\end{document}